\documentclass[a4paper, reqno, 12pt]{amsart}

\usepackage{geometry}       

\usepackage{float}
\usepackage{bm}
\usepackage{fullpage,xcolor}
\usepackage[mathscr]{euscript}
\usepackage[all]{xy}
\usepackage{epsfig}
\usepackage{amsfonts}
\usepackage{mathptmx}

\usepackage{graphicx}
\usepackage{amssymb} 
\usepackage{amsmath}
\usepackage{amsthm}
\usepackage{mathrsfs}
\usepackage{epstopdf}
\usepackage{url}
\usepackage[msc-links,alphabetic]{amsrefs}
\usepackage{tikz}
\usepackage{color}
\makeatletter

\@addtoreset{equation}{section}
\makeatother 

\theoremstyle{plain}
\newtheorem{theorem}{Theorem}[section]

\newtheorem{proposition}[theorem]{Proposition}

\newtheorem{remarks}[theorem]{Remarks}

\theoremstyle{definition}
\newtheorem{definition}[theorem]{Definition}

\newtheorem{remark}[theorem]{Remark}
\newtheorem{note}[theorem]{Note}

\usepackage{latexsym}
 
\title[From annular to toroidal knotoids]
 {From annular to toroidal knotoids \\ and their universal bracket polynomials}

\author{Ioannis Diamantis}
\address{Department of Data Analytics and Digitalisation,
Maastricht University, School of Business and Economics,
P.O.Box 616, 6200 MD, Maastricht,
The Netherlands.}
\email{i.diamantis@maastrichtuniversity.nl}

\author{Sofia Lambropoulou}
\address{School of Applied Mathematical and Physical Sciences, National Technical University of Athens, Zografou campus, GR-15780 Athens, Greece.}
\email{sofia@math.ntua.gr}
\urladdr{http://www.math.ntua.gr/~sofia}

\author{Sonia Mahmoudi}
\address{Advanced Institute for Materials Research, Tohoku University, 2-1-1 Katahira, Aoba-ku, Sendai 980-8577, Japan; RIKEN iTHEMS, 2-1 Hirosawa, Wako, Saitama 351-0198, Japan}
\email{sonia.mahmoudi@tohoku.ac.jp}



\thanks{We acknowledge with gratitude the support and hospitality of the Advanced Institute for Materials Research (WPI-AIMR) and the Tohoku Forum for Creativity at Tohoku University, as well as the International Institute for Sustainability with Knotted Chiral Meta Matter (WPI-SKCM$^2$).The third author was also supported by JSPS KAKENHI Grant-in-Aid for Early-Career Scientists, Grant Number 25K17246.}

\begin{document}

\setcounter{section}{-1}

\begin{abstract} 
In this paper we study the theory of multi-knotoids in the annulus and in the torus, building  up from the theory of planar knotoids to the theory of toroidal knotoids through the theory of annular knotoids. We introduce the concept of lifted annular and toroidal knotoids and examine inclusion relations arising naturally from the topology of the supporting manifolds. We also introduce the concept of mixed knotoids as special cases of planar knotoids, containing a fixed unknot for representing the thickened annulus or a fixed Hopf link for representing the thickened torus. We then extend the Turaev loop bracket for planar knotoids to bracket polynomials for annular and for toroidal  knotoids, whose universal analogues recover the Kauffman bracket knotoid skein modules of the thickened annulus and the thickened torus.
\end{abstract}

\maketitle

%%%%%%%%%%%%%%%%%%%% 

\section{Introduction}\label{sec:0}

Knot theory is a continuously expanding core area of Topology, with significant implications across various fields. The ultimate goal of classical knot theory  is the classification of knots and links up to ambient isotopies, which is tackled with the construction of knot invariants. The area advanced significantly with the milestone discovery of the {\it Jones polynomial} in 1984 \cite{Jo}, soon reconstructed via the {\it Kauffman bracket polynomial} with the use of succinct diagrammatic techniques \cite{Kauffman1990}. 

In a further development Turaev introduced the theory of knotoids \cite{T}, a diagrammatic theory of open-ended  knotted  curves in surfaces, with the main motivation of reducing computational complexity for some knot invariants. A {\it knotoid} is an equivalence class of knottoid diagrams in an oriented surface $\Sigma$, under surface isotopies and the standard Reidemeister moves that take place away from the two endpoints (see Figure~\ref{mkd}(a)). The theory of {\it spherical knotoids} (for the case $\Sigma = S^2$) extends  the theory of classical knots and, in turn, the theory of {\it planar knotoids} (for the case $\Sigma = \mathbb{R}^2$ or $D^2$) generalizes the theory of spherical knotoids.  Similar to the notion of classical links, one may extend the notion of knotoid to that of a  multi-knotoid, which is a union of a knotoid with a finite numbers of knots (see Figure~\ref{mkd}(b)). Furthermore,  In \cite{T} Turaev related spherical knotoids to $\Theta$-curves in 3-space. In analogy, in \cite{GK} it is established that planar knotoids lift faithfully to (isotopy classes of) rail arcs in 3-space, called {\it spatial knotoids}, arcs with their ends attached on two parallel lines passing by the two endpoints. This last interpretation led to  modeling  proteins by planar knotoids. So, the theory of knotoids has become an important tool in their study and classification \cite{DGBS, GGLDSK, DS}. 

In this paper, we study the theory of multi-knotoids in the annulus and the torus. Figure~\ref{ex2} illustrates a planar, an annular and a toroidal multi-knotoid. We  introduce the notion of the \textit{lift} of an annular resp. a toroidal multi-knotoid into an open knotted `rail curve', linked with classical knots, with its ends attached to two parallel \textit{rail segments} piercing the thickened annulus (see Figure~\ref{an}) resp. the thickened torus (see Figure~\ref{ex3}). We further establish \textit{rail isotopy} for the lifts of   annular and  toroidal multi-knotoids  by representing them as specific types of mixed  planar multi-knotoids  and by exploiting various inclusion relations of the supporting manifolds  (Theorems~\ref{isopST} and~\ref{isopTT}).

The Kauffman bracket skein module of the plane (or the disc or the sphere) is one-dimensional, freely generated by the unknot \cite{Kauffman1990}.  In \cite{T} Turaev defines the bracket polynomial for spherical and planar knotoids applying the same rules as of the Kauffman bracket polynomial, except for the initial condition in which the unknot is replaced by the trivial knotoid. In \cite{GGLDSK} the \textit{loop bracket polynomial} for planar knotoids is defined as a two-variable extension of the \textit{Turaev bracket polynomial}, with the additional variable  assigned to nested loop components enclosing the trivial knotoid segment in a state diagram. 

In this paper, we also construct the \textit{universal planar, annular} and \textit{toroidal bracket polynomials} for multi-knotoids in the corresponding surfaces (Definitions~\ref{u-planarskein}, \ref{annularskein}, and~\ref{toroidalskein}), which generalize the loop bracket polynomial to an infinite variable Laurent polynomial in the case of planar knotoids. Our universal bracket polynomials capture the full complexity of the Kauffman bracket skein modules for multi-knotoids in these  surfaces, as well as for their three-dimensional lifts in 3-space, the thickened annulus and the thickened torus (Theorems~\ref{th:universal-bracket-regular-isotopy}, \ref{th:pkaufbst} and ~\ref{th:pkaufbtt}).  For computational purposes we also introduce finite variable specializations for annular and toroidal multi-knotoids (Definitions~\ref{pkaufbst},~\ref{pkaufbtt} and~\ref{def:reduced-tor-bracket}). For all bracket polynomials we also provide closed state summation formuli. By further normalizing the regular isotopy invariants under the Reidemeister move R1, we obtain the corresponding {\it annular} and {\it  toroidal Jones polynomials} as new invariants for classifying annular and toroidal multi-knotoids, see Theorems~\ref{th:jones-an} and ~\ref{th:jones-tor}. 

\begin{figure}[ht]
\begin{center}
\includegraphics[width=6in]{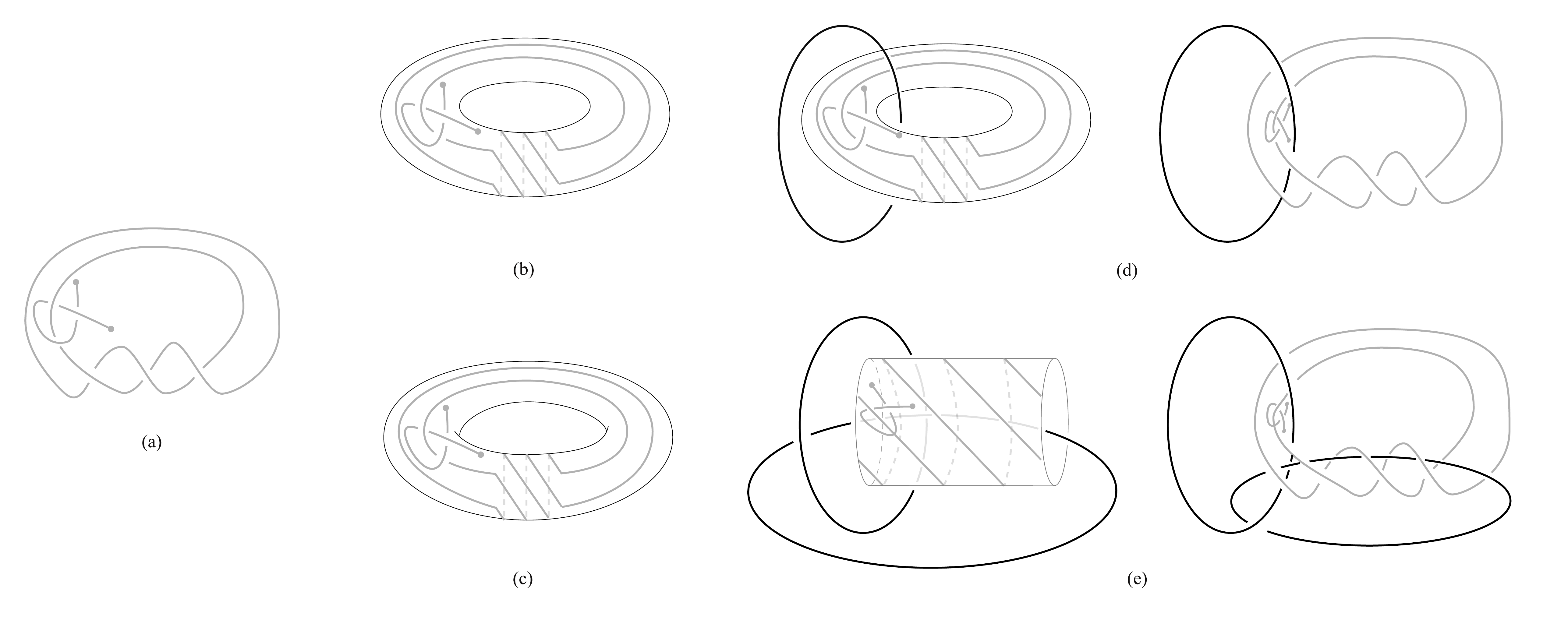} 
\end{center}
\caption{ A planar multi-knotoid (a),  the multi-knotoid in the torus (b),  a mixed link representing the multi-knotoid in the thickened annulus (c) and in the thickened torus (d).}
\label{ex2}
\end{figure}

More specifically, for comparing the theories of knotoids in the different surfaces under consideration, we use the {\it inclusion relations} of the supporting manifolds. The inclusion of a disc in the annulus resp. the torus enables us to view planar multi-knotoid diagrams as annular resp. toroidal ones,  leading to  {\it injections} of the theory of planar multi-knotoids into the theory of annular resp. toroidal multi-knotoids (Propositions~\ref{discannulus} and~\ref{disctorus}). Conversely, the inclusion of the  annulus in a disc induces a {\it surjection} of the theory of annular multi-knotoids onto planar multi-knotoids (Proposition~\ref{discannulus}). We use these inclusion relations to prove the isotopy generalized Reidemeister theorem for lifted annular multi-knotoids (Theorem~\ref{isopST}), leading to the analogous injection and surjection relations.

The notion of lift for toroidal multi-knotoids enables us to consider the inclusion of the thickened torus into an enclosing three-ball (Figure~\ref{planarannulartorus}), allowing any multi-knotoid in the thickened torus to be regarded as a spatial multi-knotoid, by extending the two rail segments to rails in the three-ball and then in three-space. However, and unlike the case of the thickened annulus, this construction  does {\it not} yield  {\it a well-defined map} from toroidal to planar multi-knotoids, since isotopic multi-knotoids in the thickened torus may correspond to non-isotopic spatial multi-knotoids. The obstruction arises because, the rail segments in the torus do not initially extend across the 3-ball, so isotopies occurring below the rails may induce forbidden moves upon projection (Remarks~\ref{torus_no_surjection}). 

We finally explore how the theory of annular multi-knotoids relates to that of toroidal multi-knotoids. The inclusion of the  (thickened) annulus in the (thickened) torus induces a {\it well-defined map} from annular multi-knotoids  to toroidal multi-knotoids (and respectively their lifts). Yet, this map is {\it not injective}, since in the torus we also have the longitudinal toroidal move (see Figure~\ref{toroidal-move-knotoid}) which  is not permitted for annular multi-knotoids (Proposition~\ref{prop:no-injective}). Finally, the inclusion of the thickened torus in the solid torus (Figure~\ref{annulartorus}), which is homeomorphic to a thickened annulus, allows us to view knotoids in the thickened torus as knotoids in the thickened annulus, where the meridional windings trivialize, by extending at the same time the two piercing rail segments in the width of the thickened annulus. However, similarly to the case of toroidal and planar multi-knotoids, this inclusion  does {\it not} induce {\it a well-defined map} of toroidal multi-knotoids onto annular multi-knotoids (Remark~\ref{rem:incl-TT-ST}).

Despite the simplicity of the bracket skein relation, computing the Kauffman bracket skein module of annular and toroidal multi-knotoids proves to be quite subtle. The difficulty lies in finding independent initial conditions for assigning different variables that will ensure the well-definedness of the bracket polynomials. Our key observation in analyzing the state diagrams in either setting is that: in the case of the annulus the trivial knotoid is {\it trapped}  between two sets of essential unknots, the inner and the outer ones, along with its nesting null-homotopic unknots  (see Figure~\ref{fig:annular-turaev} and Theorem~\ref{thm:annularskein}), while for classical link diagrams in the annulus it is only the numbers of the essential and null-homotopic unknots that count invariantly. In contrast, in the case of the torus the  trivial knotoid is {\it not trapped}: neither within a $(p,q)$-torus knot, for $p, q$ coprime, nor among the components of a $(p,q)$-torus link, while the nesting of null-homotopic unknots is still present (see Figure~\ref{fig:toroidal-turaev} and Theorem~\ref{th:univ-tor}). For classical link diagrams in the torus it is only the numbers of $(p,q)$-torus knots  and null-homotopic unknots that count. It is worth pointing out that in the theory of toroidal multi-knotoids a $(p,q)$-torus knot is distinct from a $(q,p)$-torus knot as essential curves but isotopic as non-essential ones (see Remark~\ref{rem:inclusionpq}).

The notions of lifts for annular and toroidal multi-knotoids enabled us to establish further connections of these theories to the theory of mixed links  \cite{LR1} (see right-hand side of Figure~\ref{ex2}): in the first case to \textit{${\rm O}$-mixed} multi-knotoids in $S^{3}$, which are spatial multi-knotoids that contain a pointwise fixed unknotted component  representing the complementary solid torus of our thickened annulus \cite{La2} (Theorem~\ref{isopmixed}). In the second case to  \textit{${\rm H}$-mixed} multi-knotoids in $S^{3}$, which are spatial multi-knotoids that contain a  Hopf link as a pointwise fixed sublink, whose complement is our thickened torus (Theorem~\ref{H-isopmixed}). Representing annular and toroidal knotoids by mixed planar ones (with the assumption of some forbidden moves) can reduce their study to better understood objects. This approach was also used in \cite{DLM-pseudo} for representing annular and toroidal pseudo knots as planar pseudo knots.

We further adapt the annular and toroidal bracket polynomials to the mixed link settings, cf.~\cite{LR1}. We  introduce the finite and infinite variables polynomials {\it (universal) ${\rm O}$-mixed bracket} for planar ${\rm O}$-mixed multi-knotoids (Definition~\ref{mix-pkaufbst}), which we prove to be invariants under regular isotopy of ${\rm O}$-mixed multi-knotoids, equivalent to the (universal) annular bracket polynomial (Theorem~\ref{th:O-pkaufbst}). Further, we introduce the finite and infinite variables polynomials  {\it (universal, reduced) ${\rm H}$-mixed bracket} for planar ${\rm H}$-mixed multi-knotoids  (Definitions~\ref{mix-pkaufbtt} and~\ref{def:reduced-H-bracket}), which we prove to be invariants under regular isotopy of ${\rm H}$-mixed multi-knotoids, equivalent to the (universal) toroidal and the reduced toroidal bracket polynomials (Theorem~\ref{th:th-mix-in}). The  adaptation of the annular and toroidal bracket polynomials to  planar ${\rm O}$-mixed and ${\rm H}$-mixed  multi-knotoids was done by a direct translation of the corresponding state curves. As we note, if we started from planar ${\rm O}$-mixed and ${\rm H}$-mixed  multi-knotoids, the bracket states would involve larger sets of simple curves winding several times around the fixed components (see Figure~\ref{state-Hmixed}). Diagrams of this form have been used in the braid approach to skein modules of various 3-manifolds. For details cf. \cites{La2,D2} and references therein. The different sets of states are related via a change of bases in the corresponding Kauffman bracket skein modules.  

Finally, we normalize the (universal) mixed polynomials under the move R1, and thus obtain {\it ${\rm O}$-mixed and ${\rm H}$-mixed Jones polynomials}, invariants for annular and toroidal multi-knotoids through the theory of planar mixed multi-knotoids  (Theorem~\ref{th:Jones-H}).

We conclude this paper by computing the planar, annular and toroidal bracket polynomials of two examples in order to highlight cases where they  distinguish or not equivalent multi-knotoids. 

Apart from the interest in studying annular and toroidal multi-knotoids per se, another motivation for us is their relation to periodic tangloid diagrams in a ribbon and in the plane, respectively, through imposing one and two periodic boundary conditions by means of corresponding covering maps. See Figure~\ref{DP-multilinkoid} for an example of a doubly periodic tangloid, defined as the universal cover of a multi-knotoid in the thickened torus. This is the subject of ongoing work of the authors. For further details in periodic tangles, cf. for example \cite{DLM} and references therein. 

The paper is organized as follows. In \S~\ref{sec:setup} we recall the basic notions associated with spherical and planar multi-knotoids, including the Reidemeister equivalence, and the lift in three-dimensional space from \cite{T}. In \S~\ref{ankn} we extend the theory from planar to annular multi-knotoids, including results from \cite{D4}. We first present the annular Reidemeister equivalence. We then define the lifts of  annular multi-knotoids in the thickened annulus and their isotopy moves, which lead to their representation as planar ${\rm O}$-mixed multi-knotoids. We also explore the inclusion relations between annular, planar and ${\rm O}$-mixed multi-knotoids. In \S~\ref{sec:annular-bracket}, we extend the bracket polynomial for annular multi-knotoids and planar ${\rm O}$-mixed multi-knotoids, proving their invariance under regular isotopy, and we normalize them into Jones-type polynomials. In \S~\ref{sec:toroidal-knotoid} we introduce the theory of toroidal pseudo knots.  We first define the pseudo Reidemeister equivalence moves for toroidal pseudo knots. We then define the lift of toroidal pseudo knots into a closed pseudo curve in the thickened torus and the notion of  isotopy for such curves. The lift of toroidal pseudo knots leads to their representation as planar ${\rm H}$-mixed multi-knotoids. We also explore the inclusion relations between toroidal, annular and planar pseudo knots, as well as of ${\rm O}$-mixed and ${\rm H}$-mixed multi-knotoids. We then define the bracket polynomials for toroidal multi-knotoids and planar ${\rm H}$-mixed multi-knotoids in \S~\ref{sec:tor-bracket}, and we prove their invariance under regular isotopy. We further normalize these polynomials into Jones-type polynomials, defining new invariants for classifying toroidal multi-knotoids. Finally, in \S~\ref{computing_bracket} we make comparative computations of the planar, annular and toroidal bracket polynomials on specific examples. 

%%%%%%%%%%%%%%%%%%%%%%%%%%%%%%%%%%%%%%%%%%%%%%%%%%%%%%%%%%%%%%%%%%%%%%

\section{Preliminaries and notations}\label{sec:setup} 

In this section we recall basic notions from the theory of knotoids and we also add some extra remarks.

\subsection{Basics on knotoids}

Turaev introduced knotoids in \cite{T} as open knotted curves in oriented surfaces. More precisely, a {\it knotoid diagram} in a surface $\Sigma$ is a generic smooth immersion of the unit interval $[0, 1]$ in the interior of $\Sigma$, such that the only singularities are transversal double points endowed with over/under crossing information and shall be referred to as `crossings'. The images of 0 and 1 under the immersion are the {\it endpoints}, referred to by the terms `leg' and `head', respectively, and are distinct from  each other and the crossings. It is worth noting that the endpoints may be situated in  different regions of the diagram. Also, that a knotoid diagram inherently possesses a natural orientation from its leg to its head. An example of a knotoid diagram is illustrated in Figure~\ref{mkd}(a). 
 
A {\it knotoid} in the surface $\Sigma$ is defined as an equivalence class of knotoid diagrams in $\Sigma$, the equivalence relation being induced by surface isotopies and the well-known Reidemeister moves R1, R2, and R3, illustrated in Figure~\ref{rmoves}, all occurring away from the endpoints. A surface isotopy may swing an endpoint within its diagrammatic region. This move shall be called \textit{swing move}. Note also that a Reidemeister equivalence cannot exchange the roles of the endpoints. The {\it trivial knotoid} is an embedding of $[0, 1]$ in $\Sigma$. The set of knotoids in  $\Sigma$ is denoted as $\mathcal{K} (\Sigma)$.

\begin{figure}[H]
\begin{center}
\includegraphics[width=4.2in]{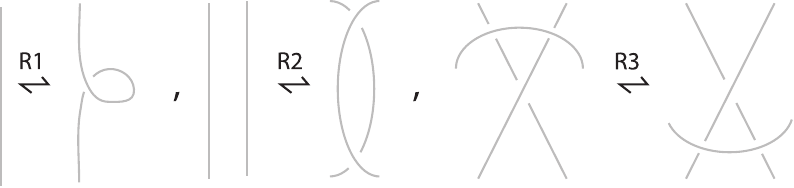}
\end{center}
\caption{The Reidemeister moves.}
\label{rmoves}
\end{figure}

Manipulating a strand adjacent to an endpoint, either by crossing over or under a transversal arc, is  prohibited, otherwise the theory trivializes. These {\it forbidden moves} are depicted in Figure~\ref{forbidoid}.

\begin{figure}[H]
\centerline{\includegraphics[width=4.3in]{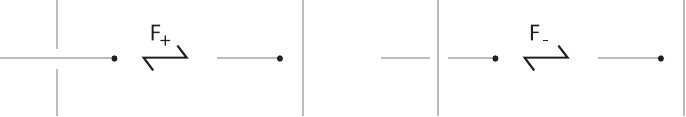}}
\vspace*{8pt}
\caption{The forbidden moves for knotoids.}
\label{forbidoid}
\end{figure}

Figure~\ref{ffm} illustrates two situations where apparently forbidden moves are performed. In the first case, an R1 move is followed by planar isotopy, while in the second case an R2 move is followed by planar isotopy. These moves are `fake forbidden moves'. 

\begin{figure}[H]
\begin{center}
\includegraphics[width=5.2in]{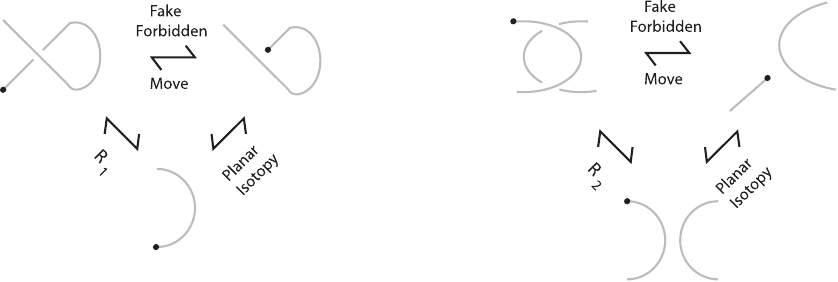}
\end{center}
\caption{Fake forbidden moves.}
\label{ffm}
\end{figure}

Knotoids can be extended into three broader concepts: {\it linkoids}, {\it multi-knotoids} and {\it multi-linkoids}. A {\it multi-knotoid} is defined as an equivalence class of a generic immersion  in the interior of $\Sigma$ of an oriented segment and a finite number of  circles, endowed with under/over data at the crossings, under the same equivalence relation as for knotoids \cite{T}. See Figure~\ref{mkd}(b) for an  example of a multi-knotoid.  An {\it oriented  multi-knotoid} is obtained by also assigning orientations to the closed components. Similarly, a linkoid is defined as an equivalence class of a generic immersion  in the interior of $\Sigma$ of finitely many oriented segments, endowed with under/over data at the crossings (see for example Figure~\ref{mkd}(c)), while a multi-linkoid is defined as an equivalence class of a generic immersion  in the interior of $\Sigma$ of a finite number of  oriented segments and  oriented circles, endowed with under/over data at the crossings (see Figure~\ref{mkd}(d)). In this paper we only focus on knotoids and multi-knotoids. Further, by  knotoids we shall be referring also to  multi-knotoids. 

\begin{figure}[H]
\begin{center}
\includegraphics[width=5.3in]{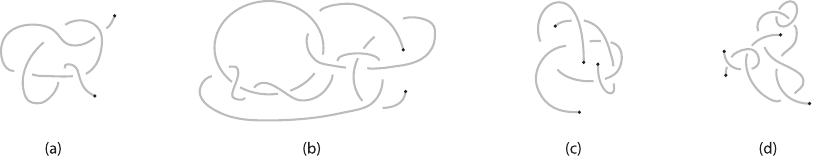}
\end{center}
\caption{(a) A knotoid, (b) a multi-knotoid, (c) a linkoid and (d) a multi-linkoid.}
\label{mkd}
\end{figure}

\subsection{Spherical and planar knotoids}\label{sec:sphere-plane}

In \cite{T} Turaev observes that there is a bijection between classical knot types and spherical knotoids having their endpoints lying in the same region of the diagram. 

In order, now, to understand the relation between spherical and planar knotoids let us first recall that: the 2-sphere $S^2$  can be viewed as the compactification of $\mathbb{R}^2$, $S^2=\mathbb{R}^2 \cup \{ \infty \}$, or equivalently, as the one-point identification of the boundary of a disc  or the gluing along the circular boundaries of two discs.  The inclusion  $\iota: \mathbb{R}^2 \hookrightarrow S^2$ resp. $\iota: D^2 \hookrightarrow S^2$ allows us to view any  knotoid diagram in the plane resp. in the disc as a spherical knotoid diagram. Only, this diagram is not allowed to undergo spherical moves. A {\it spherical move} is the sliding of an arc around the back surface of the sphere. On the other hand, spherical knotoid diagrams can be also represented as planar knotoid diagrams via the decompactification of $S^2$, for example using the stereographic projection, with the addition of the spherical move in their equivalence relation. In fact, as explained in \cite{T}, every spherical knotoid has a `normal' representative, a spherical knotoid diagram  containing the point at $\infty$ in the region of the leg. Then,  spherical knotoid diagrams correspond to  planar knotoid diagrams having the leg lying in the outer region.  This defines the inclusion map $\rho: \mathcal{K} (S^2) \hookrightarrow \mathcal{K} (\mathbb{R}^2)$ from spherical to planar knotoids, since normal diagrams are preserved by the Reidemeister moves and local surface isotopies. To summarize, planar knotoids surject to spherical knotoids but do not inject. This means that planar knotoids provide a much richer combinatorial structure than the spherical ones. See further comments in Subsection~\ref{sec:planar-lift}. 

\begin{remark}\label{rem:disk}
Since all moves in the planar equivalence are local, it follows clearly that the theory of planar multi-knotoids is equivalent to the theory of multi-knotoids in $D^2$.
\end{remark}

\begin{remark}\label{rem:planar-closure}
The above extend in an obvious manner to  multi-knotoids. There is a bijection between classical link types and spherical multi-knotoids having their endpoints lying in the same region of the diagram. Also, there is an injection of classical links to spherical  multi-knotoids, induced by the under-closure or over-closure of the endpoints \cite{T}. Further, there is an injection of spherical to planar ones. 
\end{remark}

\subsection{Lifting in three-space spherical and planar knotoids}\label{sec:planar-lift}

Turaev \cite{T} showed that there is an isomorphism of monoids of spherical knotoids and of simple $\Theta$-graphs, that gives rise to a geometric interpretation of spherical knotoids via $\Theta$-curves. Further, G\"ug\"umc\"u and Kauffman  \cite{GK} showed that a planar knotoid diagram can be lifted to a `rail curve' in three-space, an open curve in $\mathbb{R}^3$ using two specified parallel lines denoted $L$ and $H$, the `rails', passing through its endpoints $l$ and $h$.
The lift of a multi-knotoid together with its two rails shall be called \textit{spatial multi-knotoid}. See left hand side of Figure~\ref{planar-lifting} for an illustration.

\begin{figure}[H]
\begin{center}
\includegraphics[width=4.8in]{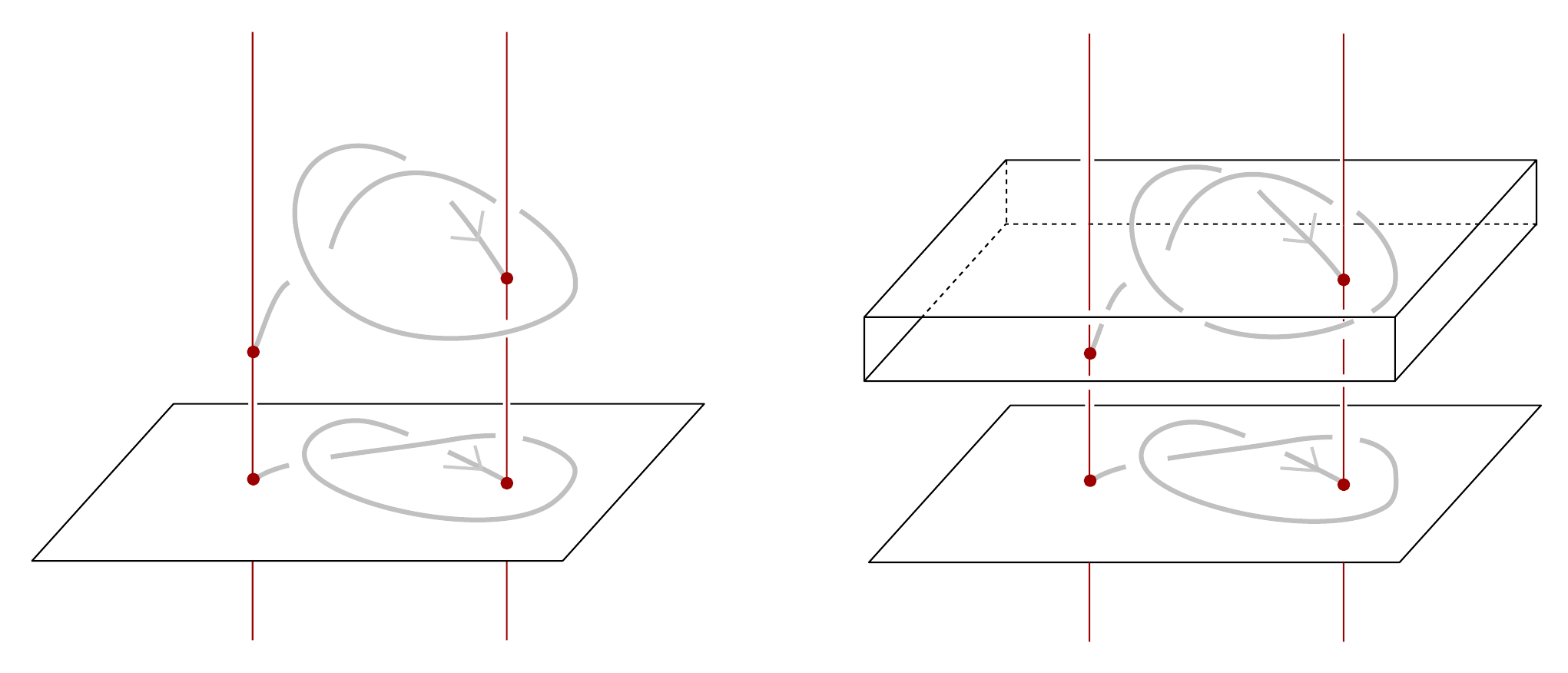}
\end{center}
\caption{Lifting a planar knotoid to a spatial knotoid in space and in the thickened square.}
\label{planar-lifting}
\end{figure}

Spatial multi-knotoids are considered up to {\it rail isotopy}. Rail isotopy consists in: (a), ambient isotopy of the rail curves in the complement of the rails  $L$ and $H$  in $\mathbb{R}^3$, where the endpoints of the curves remain on the rails throughout the deformation, but are allowed to slide along their rails, and (b), parallel shifts of the rails up to isotopy. These last moves correspond to endpoint swing moves on the diagrammatic level and shall be called {\it rail shifts}. Note that a rail is not permitted to cross  an arc of the rail curve, which would correspond to a forbidden move on the diagrammatic level. Clearly, a rail isotopy cannot exchange the roles of the rails $L$ and $H$. An example of a rail shift  and a sliding move are illustrated in Figure~\ref{railsliding} even if in a different context. The rail isotopy class of a spatial multi-knotoid shall be also referred to as {\it spatial multi-knotoid}.

A well-established result in \cite{GK} demonstrates that rail isotopy classes of lifted knotoids correspond bijectively to equivalence classes of planar knotoids, when projected on a plane perpendicular to the rails. For the proof of this equivalence, curves are considered to be piecewise-linear (PL). In the PL setting,  isotopy  is captured by a specific type of local  move, the triangular move or {\it $\Delta$-move}. This move and its inverse consists in replacing a line segment of the PL knotoid, corresponding to one side of a spatial triangle, by the other two sides of the triangle, provided that the interior of the triangle is free of any other arcs. 

\begin{remark}\label{thickdisc}
The locality of the isotopy moves makes the theory of spatial multi-knotoids equivalent to the  theory of isotopy classes of rail curves in the thickened disc $D^2 \times I$, where $I=[0,1]$ denotes the unit interval, or the thickened square $I \times I \times I$, such that the infinitely extended rails in the spatial setting correspond to two parallel rail segments, copies of the interval $I$ piercing the interior of $D^2 \times I$ or $I \times I \times I$, and where the curves are constrained in the interior of the space (with only the endpoints being allowed to touch the boundaries). See right hand side of Figure~\ref{planar-lifting}. We shall also refer to rail curves in this setting as {\it spatial multi-knotoids}. 
\end{remark}

The lifting of planar knotoids led in turn to the theory of rail knotoids, which are regular projection of rail curves on the plane of the two parallel lines  \cite{KoLa}. In an application direction, the fact that planar knotoids surject to spherical knotoids along with the lifting has found interesting implications in the study of proteins, initiated by  Goundaroulis \cite{GGLDSK}. Proteins are long chains of amino acids that sometimes form open ended knots. The lifting of planar knotoids was then used for representing open proteins by planar knotoids, thus advancing the study of their topological entanglement and their classification \cite{DGBS, GGLDSK}.

\subsection{The universal bracket polynomial for spherical and planar knotoids}\label{sec:planar-bracket}

We recall that the Kauffman bracket polynomial is an invariant of \textit{regular isotopy}, that is invariant under the Reidemeister moves R2 and R3. 
In \cite{T} the bracket polynomial $\langle \cdot \rangle$ is defined for spherical (multi-)knotoids, extending the Kauffman bracket polynomial for classical knots and links \cite{Kauffman1990} by an extra rule that evaluates the bracket of the trivial knotoid to 1. See rules (i)-(iii) and (v) in Definition~\ref{pkaufb} below. This polynomial is referred to as the {\it Turaev bracket polynomial} or, in this work, as the {\it spherical bracket polynomial}. In \cite{T} the spherical bracket polynomial is extended to the case of planar (multi-)knotoids. An adaptation of this polynomial in \cite{GGLDSK} is called the {\it Turaev loop bracket polynomial} or, here, the {\it planar bracket polynomial}. The planar bracket polynomial $\langle \cdot \rangle$ is a 2-variable generalization of the spherical bracket, with an extra rule introducing a new variable $v$ that counts nested loop components enclosing the trivial knotoid. More precisely:  

\begin{definition}\label{pkaufb}\rm
The {\it planar bracket polynomial} of a planar multi-knotoid diagram is a 2-variable Laurent polynomial in $\mathbb{Z}[A^{\pm 1}, v]$ defined by means of the following rules, inductively on the number of crossings in the diagram:
\begin{itemize}
    \item[i.] $\langle \ \raisebox{-2pt}{\includegraphics[scale=0.8]{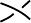}} \ \rangle = A \langle \ \raisebox{0pt}{\includegraphics[scale=0.5]{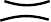}} \ \rangle + A^{-1} \langle \ \raisebox{-2.5pt}{\includegraphics[scale=0.6]{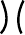}} \ \rangle$
    \vspace{.2cm}
    \item[ii.]  $\langle \ \raisebox{-2pt}{\includegraphics[scale=0.85]{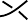}} \ \rangle = A^{-1} \langle \ \raisebox{0pt}{\includegraphics[scale=0.5]{skein-2.pdf}} \ \rangle + A \langle \ \raisebox{-2.5pt}{\includegraphics[scale=0.6]{skein-3.pdf}} \ \rangle $
     \vspace{.2cm}
    \item[iii.]$\langle L \sqcup \, \mathrm{O}^k \rangle  =  d^k \, \langle L \rangle$, where $d = -A^2 – A^{-2}$ and $\mathrm{O}^k$ stands for $k$ null-homotopic unknots.
    
     \vspace{.1cm}
     \item[iv.] $\langle  \raisebox{-6pt}{\includegraphics[scale=0.11]{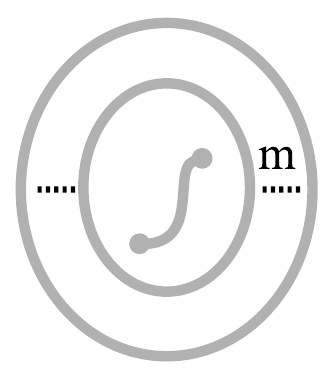}} \rangle = v^m \, \langle \raisebox{-1pt}{\includegraphics[scale=0.1]{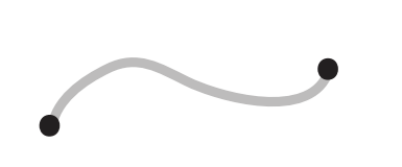}} \rangle$, where $m\in \mathbb{N}\cup \{0\}$ is the number of unknots enclosing the trivial knotoid.
     \vspace{.2cm}
 \item[v.] $\langle \raisebox{-1pt}{\includegraphics[scale=0.1]{arc.pdf}} \rangle = 1$
\end{itemize}
\end{definition}

Given a planar knotoid diagram $K$ one computes its planar bracket $\langle K \rangle$ by applying the rules in Definition~\ref{pkaufb} as follows: we first smoothen all crossings by rules (i) and (ii), and then get rid of all isolated unknots at the cost of the loop value $d$, using rule (iii). In the end all states reduce to configurations comprising a number of unknots enclosing the trivial knotoid, so we apply rule (iv) and finally rule (v). In Section~\ref{computing_bracket} we provide some computations of the planar bracket polynomial on specific examples.

In analogy to the classical bracket, the spherical and planar bracket polynomials for spherical and planar multi-knotoids are invariants of {\it regular isotopy}, that is, invariants under planar isotopy and the Reidemeister moves R2 and R3 (excluding R1). 

\begin{remark}
In the case of a spherical knot type knotoid, the evaluation of the spherical bracket polynomial coincides with that of the Kauffman bracket polynomial of the corresponding classical knot. Hence, and since the trivial knotoid, which appears in every state, is a knot type knotoid identifying with the unknot, rule (v) coincides with the initial condition of the classical bracket, so the spherical bracket polynomial coincides with the Kauffman bracket polynomial for classical knots. Analogous considerations apply for planar knot type knotoids, after specializing $v$ in rule (iv) to the loop value $d$. Indeed, identifying the trivial knotoid with the unknot in rule (iv), there are no more forbidden moves justifying rule (iv).
\end{remark}

\begin{remark}
When computing the planar bracket polynomial of a planar knotoid diagram $K$, rule (iv) may have to be applied. Viewing now $K$ as a spherical knotoid diagram, as discussed in Section~\ref{sec:sphere-plane}, and computing its  spherical  bracket polynomial, the loops appearing in rule (iv) reduce to loops as in rule (iii) by applications of the spherical move. Thus, in rule (iv) we obtain $m=0$, in other words the planar variable $v$  specializes to the loop value $d$. So the planar bracket polynomial  is consistent with the spherical  bracket polynomial. 
Conversely, let $K$ be a spherical knotoid diagram represented by a planar knotoid diagram $K^\prime$, using the relation between planar and spherical knotoids described in Section~\ref{sec:sphere-plane}. Since its leg lies in the outer region of the plane, when computing the planar bracket of $K^\prime$, there will be no occurence of rule (iv). So the spherical bracket polynomial  is consistent with and can be computed through the  planar bracket polynomial.  For the extension of the Turaev  bracket polynomial to linkoids in $S^2$  cf. \cite{Eleni}. 
\end{remark}

Rule (iv) can be which generalize the generalized to an infinite variable expression by substituting $v^m$ by $v_m$. 

\begin{definition}\label{u-planarskein}
The {\it universal planar bracket polynomial} for planar multi-knotoids, denoted $ \langle \cdot \rangle_U $, is defined by means of rules (i), (ii), (iii), (v) of Definition~\ref{pkaufb} and rule (iv$^{\prime}$) below:

\begin{center}
[iv$^\prime$.] $\langle  \raisebox{-6pt}{\includegraphics[scale=0.11]{turloop.pdf}} \rangle_U = v_m \, \langle \raisebox{-1pt}{\includegraphics[scale=0.1]{arc.pdf}} \rangle_U$, where $m$ is the number of unknots enclosing the trivial knotoid.
\end{center}
\end{definition}

Then we have the following:

\begin{theorem}\label{th:universal-bracket-regular-isotopy}
The universal planar bracket polynomial  is an infinite variable Laurent polynomial $ \langle \cdot \rangle_U  \in \mathbb{Z}\left[A^{\pm 1}, v_m \right], \ m \in \mathbb{N}\cup \{0\}$, which is a regular isotopy invariant for planar multi-knotoids.
\end{theorem}

\begin{proof}
For the well-definedness we argue as follows:  crossing smoothings use only rules (i)-(ii) for reducing to states, so independence of the sequence of crossings follows immediately as for the Turaev bracket for surface $\Sigma = D^2$. Then we are left with a number  of null-homotopic unknots and the trivial knotoid with a number of  unknots enclosing it, which can be all assumed unordered, since rules (iii) and (iv$^{\prime}$) (or rule (iv)) do not `see' any ordering. So. rule (iv$^{\prime}$) of Definition~\ref{u-planarskein} in place  of rule (iv) captures bijectively the infinitely many variables. Now, the forbidden moves ensure that all unknotted components are locked in their positions up to ordering, so there is no ambiguity in applying rule (iii) for the null-homotopic ones or rule (iv$^{\prime}$) (or rule (iv)) for the  ones nesting the trivial knotoid. Hence the well-definedness is concluded. Finally, invariance under regular isotopy of $\langle \cdot \rangle_U$ involves only rules (i)-(iii) and follows directly from the invariance of the planar bracket polynomial $\langle \cdot \rangle$ for planar multi-knotoids.
\end{proof}

The universal bracket polynomial realizes the Kauffman bracket skein module for planar knotoids, compare with \cite{GG}. In complete analogy to the classical bracket, one can arrive at an alternative definition of the spherical/planar bracket polynomials and the universal planar bracket polynomial via weighted state summations, where different types of closed components receive different evaluations: 

\begin{definition}
The spherical/planar bracket polynomial of a spherical/planar multi-knotoid diagram and the universal planar bracket polynomial of a planar multi-knotoid diagram $K$ are defined via the following weighted {\it state summation} formuli: 

\begin{equation*}\label{sskbst}
\langle K \rangle\ = \underset{s\in S(K)}{\sum}\, A^{\sigma_s}\, d^{k_s}\, v^{m_s}  \quad {\rm and} \quad \langle K \rangle_U\ = \underset{s\in S(K)}{\sum}\, A^{\sigma_s}\, d^{k_s}\, v_{m_s}
\end{equation*} 

\noindent where $\sigma_s$ is the number of A-smoothings minus the number of B-smoothings applied to the crossings of $K$ in order to obtain the state $s$ (see Figure~\ref{smoothing}), $d=-A^2-A^{-2}$, $k_s$ is the number of null-homotopic unknots and $m_s$ is the number of unknots enclosing the trivial knotoid in the state $s$. Note that $m_s = 0$ in the spherical case.
\end{definition}

\begin{figure}[H] 
\begin{center} 
\includegraphics[width=2.7in]{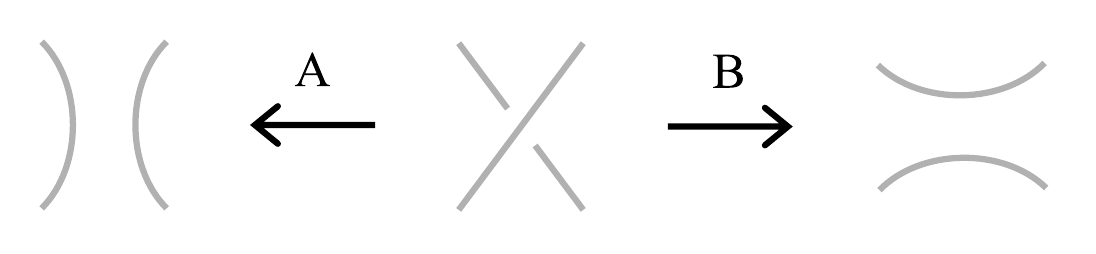} 
\end{center} 
\caption{The two types of crossing smoothing.} 
\label{smoothing} 
\end{figure} 

Denoting  $|s|$ the number of components of a state $s$ of $K$  implies that $ k_s + m_s  = |s|-1$. Further setting $v=d$ recovers the bracket polynomial for multi-knotoids as Turaev defined it for any oriented surface \cite{T}. 

\smallbreak
As in the case of the Kauffman bracket polynomial for classical link diagrams, the spherical and (universal) planar bracket polynomials can be normalized so as to also respect Reidemeister move 1, by considering the product of $\langle K \rangle$ and $\langle K \rangle_U$ by the factor $\left( -A^{-3}\right) ^{w(K)}$, where $w(K)$ is the writhe of the knotoid diagram $K$, defined as the number of positive crossings minus the number of negative crossings of (the oriented) $K$, where the rules of signs are depicted in Figure~\ref{si}.

\begin{figure}[H] 
\begin{center} 
\includegraphics[width=1.5in]{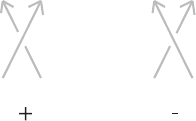} 
\end{center} 
\caption{The signs of oriented crossings.} 
\label{si} 
\end{figure} 

The {\it normalized spherical/planar bracket polynomial} for spherical resp. planar knotoids and the {\it normalized universal planar bracket polynomial} are analogous to the normalized  bracket polynomial for classical knots in $S^3$, which is equivalent to the classical Jones polynomial via the change of variable $A = t^{-1/4}$. So, we have the following (cf. \cite{T}, \cite{GGLDSK}):

\begin{theorem}\label{th:pl-bracket}
Let $K$ be a spherical/planar multi-knotoid diagram. The {\it spherical/planar Jones polynomials} and the {\it universal planar Jones polynomial} 
\[
V(K)\ =\ (-A^3)^{-w(K)}\, \langle K \rangle\ \ \ {\rm and}\ \ \ V_U(K)\ =\ (-A^3)^{-w(K)}\, \langle K \rangle_U
\]
\noindent where $w(K)$ is the writhe of $K$,  $A = t^{-1/4}$, $\langle K \rangle$ the spherical/planar bracket polynomial of $K$ and $\langle K \rangle_U$  the universal planar bracket polynomial of $K$, are isotopy invariants of spherical/planar multi-knotoids.
\end{theorem}

\section{Annular knotoids}\label{ankn}

In this section we introduce the theory of knotoids and multi-knotoids in the annulus $\mathcal{A}$, collectively referred to as annular multi-knotoids. The annulus $\mathcal{A}$ is the space $ S^1 \times D^1$ with two circular boundary components, and may be also represented as an once punctured disc, meaning a disc with a hole at its center, or as the identification space of a square with two parallel sides identified. We study annular knotoids in different geometric contexts, such as through their lifts to open curves with endpoints in the thickened annulus and through their representations as planar ${\rm O}$-mixed knotoids, that is, planar multi-knotoids with one fixed unknotted component.  We then highlight inclusion relations between planar knotoids and annular knotoids. Finally, we also define the Turaev loop bracket polynomial for annular knotoids and ${\rm O}$-mixed knotoids.

\subsection{Annular knotoid diagrams and their equivalence}\label{sec:annular-equivalence}

\begin{definition}
An {\it annular multi-knotoid diagram} is a  smooth immersion of the unit interval $[0, 1]$ and a number of circles in the interior of the annulus $\mathcal{A}$, such that the only singularities are transversal double points, the crossings, endowed with over/under crossing information. An {\it annular knotoid diagram} is a multi-knotoid diagram with no closed components.
\end{definition}

Figure~\ref{an}(a) illustrates an annular knotoid diagram. An annular knotoid diagram inherently possesses a natural orientation  from its {\it leg} (the image of $0$) to its {\it head } (the image of $1$). An {\it oriented annular multi-knotoid diagram} is obtained by also assigning orientations to the closed components.

\begin{definition}\label{anpl}
An (oriented)  {\it annular multi-knotoid} is defined as an equivalence class of (oriented)  annular multi-knotoid diagrams under surface isotopies and all versions of the (oriented)  classical Reidemeister moves, as exemplified in Figure~\ref{rmoves}, all together comprising the (oriented)  {\it Reidemeister equivalence} for annular multi-knotoids. Note that surface isotopies include the endpoint swing moves in their diagrammatic regions. 
\end{definition}

\subsection{Lifting annular knotoids}\label{sec:annular-lift}

We shall now define the lift of an annular multi-knotoid as a  `rail curve' and a collection of closed curves in the thickened annulus, in analogy to the theory of spatial  multi-knotoids (recall Subsection~\ref{sec:planar-lift}). We consider the thickening $\mathcal{A} \times I$, where $I$ denotes the unit interval (see Figure~\ref{an}). Note that this space is homeomorphic to the solid torus.

\begin{figure}[H]
\begin{center}
\includegraphics[width=3.9in]{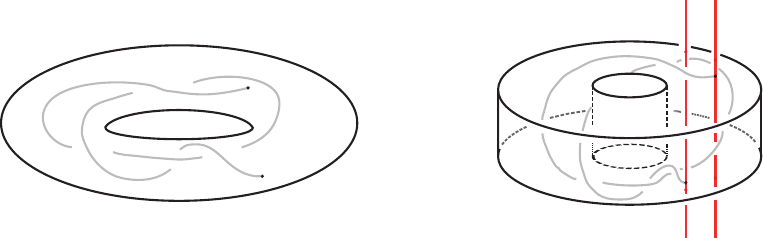}
\end{center}
\caption{(a) An annular knotoid diagram; (b) its lift in ST $\sim \mathcal{A} \times I$, where the rails are extended above and below for visual purposes. }
\label{an}
\end{figure}

\begin{definition} \label{solidlift}
The {\it lift} of an (oriented) annular multi-knotoid diagram in the thickened annulus $\mathcal{A} \times I$ is defined so that: each classical crossing is embedded in a sufficiently small 3-ball that lies in the interior of the thickened annulus,  the simple arcs connecting the crossings can be replaced by isotopic ones in the interior of $\mathcal{A} \times I$, and  the endpoints are attached to two fixed parallel segments, the rails $L$ and $H$,  which are perpendicular to the annulus $\mathcal{A}$ piercing its interior. The resulting lifted multi-knotoid together with its two rails is called a {\it multi-knotoid in the thickened annulus} and it is a collection of an open curve with its endpoints attached to the rails, and a number of closed curves, all constrained to the interior of the thickened annulus $\mathcal{A} \times I$, with the exception of the endpoints which may lie in the interiors of $\mathcal{A} \times \{0\}$ and $\mathcal{A} \times \{1\}$. See Figure~\ref{an}(b) for an illustration. 
\end{definition}

\begin{definition} \label{solidliftisotopy}
Two (oriented) multi-knotoids in the thickened annulus $\mathcal{A} \times I$ are said to be {\it rail isotopic} if they are related by an ambient isotopy of $\mathcal{A} \times I$ that takes place in its interior and in the complement of the two rails $L$ and $H$, keeping the endpoints of the open curve attached to their respective rails, being allowed to slide along them, together with parallel shifts of the rails up to isotopy, the rail shifts (compare with Section\ref{sec:planar-lift}). The rail isotopy class of a multi-knotoid in the thickened annulus shall be also referred to as {\it  multi-knotoid in the thickened annulus}.
\end{definition}

We now state the following  theorem, whose proof is detailed in the next Subsection~\ref{sec:annular-inclusions}:

\begin{theorem}\label{isopST}
Two multi-knotoids in the thickened annulus $\mathcal{A} \times I$ are rail isotopic if and only if any two corresponding annular multi-knotoid diagrams of theirs  projected to the annulus $\mathcal{A} \times \{0\}$, are Reidemeister equivalent.
\end{theorem}

\subsection{Inclusions}\label{sec:annular-inclusions}

The inclusion of a disc or a square in the annulus enables us to view planar multi-knotoid diagrams as annular ones. View left hand side of Figure~\ref{planarannular} and  left hand illustration of Figure~\ref{lift}. Further, the equivalence moves are preserved,  so we have a well-defined map $\iota: \mathcal{K}(D^2) \longrightarrow \mathcal{K}(\mathcal{A})$, which is in fact  an injection, since the equivalence moves are  the same in both settings. On the other hand, the inclusion of the annulus in a disc (see right hand side of Figure~\ref{planarannular}) induces a surjection of the theory of annular multi-knotoids onto planar  multi-knotoids, where essential components are mapped to usual components. Namely, we can state the following.

\begin{proposition}\label{discannulus}
The map $\iota: \mathcal{K}(D^2) \longrightarrow \mathcal{K}(\mathcal{A})$ is an injection. Moreover, there is a  surjection $\rho$ of the set $ \mathcal{K}(\mathcal{A})$  onto the set $\mathcal{K}(D^2)$.   
\end{proposition}

\begin{figure}[H] 
\begin{center} 
\includegraphics[width=6in]{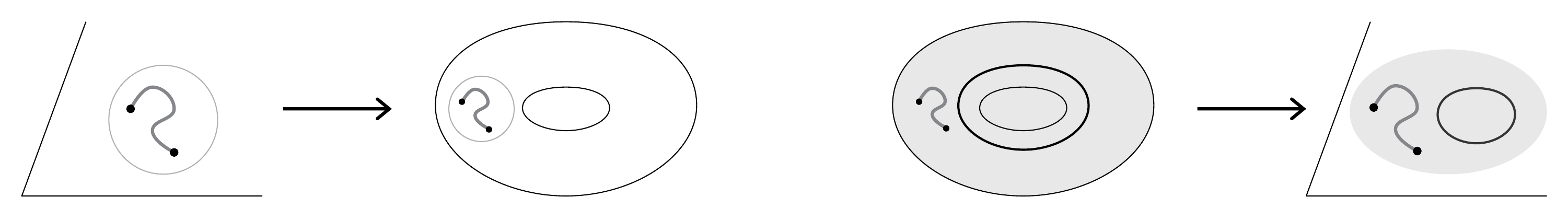} 
\end{center} 
\caption{Inclusion relations: a disc in the annulus and  the annulus in a disc.} 
\label{planarannular} 
\end{figure} 

In terms of lifts, the inclusion of the thickened disc $D^2 \times I$  or the thickened square $I \times I \times I$ (which are homeomorphic to a three-ball)  in the thickened annulus allows one to view spatial multi-knotoids as multi-knotoids in $\mathcal{A} \times I$, recall Remark~\ref{thickdisc}. On the other hand, the inclusion of the thickened annulus in a thickened disc (thickened version of Figure~\ref{planarannular}) allows one to view multi-knotoids in $\mathcal{A} \times I$ as spatial multi-knotoids. 

The inclusions above allow us now to prove Theorem~\ref{isopST} by adapting the arguments for spatial knotoids in $\mathbb{R}^3$ \cite{GK} to the setting of multi-knotoids in the thickened annulus $\mathcal{A} \times I$.

\smallbreak
\noindent \textit{Proof of Theorem~\ref{isopST}.}
For the proof we work in the PL category and we adapt the setting of multi-knotoids in the thickened annulus $\mathcal{A} \times I$ to that of spatial knotoids, so as to apply the arguments of \cite{GK}. 

An isotopy move applied to a multi-knotoid in $\mathcal{A} \times I$ away from the rails, is local and can thus be confined within a three-ball contained in the interior of $\mathcal{A} \times I$. Then, by the inclusion of the thickened annulus in a three-ball  in $\mathbb{R}^3$ (in the shape of $D^2\times I$), the isotopy move can be considered as a rail isotopy move in three-space on the corresponding spatial multi-knotoid,  as described in Section~\ref{sec:planar-lift} and Remark~\ref{thickdisc}, and this move respects the topology of the thickened annulus. On the diagrammatic level, by Theorem 2.1 in  \cite{GK}, the spatial isotopy move translates to a local move of the Reidemeister equivalence for planar multi-knotoids, away from the endpoints. Furthermore, by the inclusion of the annulus  $\mathcal{A} \times \{0\}$  in a disk in $\mathbb{R}^2$, and since the topology of the embedded thickened annulus is respected, this local move takes place in the interior of the annulus $\mathcal{A} \times \{0\}$. 

Finally, a sliding move of an endpoint along its rail can be sufficiently localized, so as to project perpendicularly on the annulus, resulting in the move to be invisible on the diagrammatic level.  Hence the proof is concluded. 

Alternatively, Theorem~\ref{isopST} can be proved by a direct adaptation of the arguments of \cite{GK} to the case of the thickened annulus. An isotopy move applied to a multi-knotoid in $\mathcal{A} \times I$ away from the rails, is local and can thus be confined within a 3-ball contained in the interior of $\mathcal{A} \times I$. Then, there is a  projection of the 3-ball to the upper annulus  $\mathcal{A} \times \{0\}$ which is regular with respect to  this move, by a general position argument, so that it translates into a diagrammatic move of  Reidemeister equivalence away from the endpoints, as in Theorem 2.1 of  \cite{GK}. For the sliding moves the same argument as above applies.
\hfill {\qedsymbol{}}

\smallbreak

Combining Theorem~\ref{isopST} and Proposition~\ref{discannulus} we obtain an injection on the level of isotopy classes of spatial multi-knotoids into the set of multi-knotoids in the thickened annulus (view also Figure~\ref{lift})  and a surjection of the theory of  multi-knotoids  in $\mathcal{A} \times I$ onto  spatial multi-knotoids, where essential components are mapped to usual components. Namely:

\begin{proposition}\label{thickdiscannulus}
The map $\iota: \mathcal{K}(D^2 \times I) \longrightarrow \mathcal{K}(\mathcal{A} \times I)$ is an injection. Moreover, there is a  surjection $\rho$ of the set $ \mathcal{K}(\mathcal{A}\times I)$  onto the set $\mathcal{K}(D^2\times I)$. 
\end{proposition}

\subsection{Representation via mixed planar knotoids}

In this subsection we represent  multi-knotoids in $\mathcal{A} \times I$ as spatial mixed multi-knotoids and annular multi-knotoids as mixed multi-knotoids  in the plane. In \cite{LR1} it is established that isotopy classes of (oriented) knots/links in a knot/link complement correspond bijectively to isotopy classes of mixed links in the three-sphere $S^3$ or in three-space $\mathbb{R}^3$, through isotopies that preserve a fixed sublink which represents the knot/link complement.  In particular,  $\mathcal{A} \times I$ is the complement of a solid torus in $S^3$. The complement solid torus can be represented unambiguously by the pointwise fixed  standard unknot ${\rm O}$. For the purposes of this section it is convenient to conceptualize $\mathcal{A}$ as a punctured disc. So ${\rm O}$ pierces the center of the thickened disc $D^2 \times I$ transversely (but not perpendicularly for our purposes). Then an (oriented) link in $\mathcal{A} \times I$ is represented unambiguously by a mixed link in $S^3$ that contains  ${\rm O}$ as a  fixed sublink.  Cf. \cite{La2} and related works referred therein. This approach  was also used in \cite{DLM-pseudo} for representing annular pseudo knots as planar pseudo knots.

For translating the theory  of annular multi-knotoids (definition and equivalence) in the mixed multi-knotoid setting we consider the lifts. A multi-knotoid $K$ in $\mathcal{A}$ is lifted in $\mathcal{A} \times I$ by means of the two rails, see Figure~\ref{an}. Then the lifted  $K$ is represented  unabiguously by the ${\rm O}$-mixed spatial multi-knotoid ${\rm O} \cup K$, as shown in Figure~\ref{mixedannularlift}, where the two rails are now extended as in spatial multi-knotoids. 

\begin{figure}[H]
\begin{center}
\includegraphics[width=3.5in]{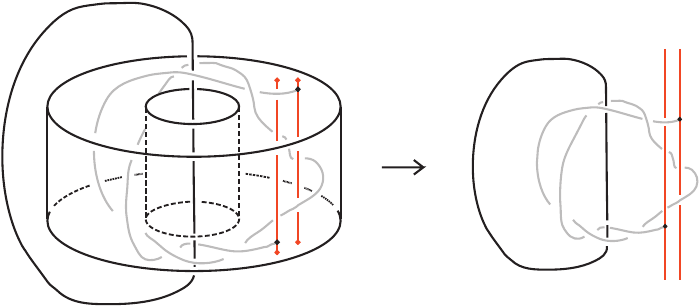}
\end{center}
\caption{The lift of an annular multi-knotoid in the thickened annulus with the complementary solid torus and its representation as an ${\rm O}$-mixed spatial multi-knotoid.}
\label{mixedannularlift}
\end{figure}

Using now the spatial lift of a multi-knotoid (recall Subsection~\ref{sec:planar-lift}), we define:

\begin{definition}\label{def:O-mixed}
A {\it spatial ${\rm O}$-mixed multi-knotoid} is a spatial multi-knotoid ${\rm O} \cup K$  which consists of the standard unknot ${\rm O}$ as a pointwise fixed part, and the spatial multi-knotoid $K$ (including its two rails) resulting by removing ${\rm O}$, as the \textit{moving part} or {\it $K$-part} of the mixed multi-knotoid. For an example view Figure~\ref{mixedannularlift}; compare with left-hand side of Figure~\ref{mknotoid}. Specifying orientations on the closed components of $K$ results in an \textit{oriented spatial ${\rm O}$-mixed multi-knotoid}. 

\noindent Furher, two spatial ${\rm O}$-mixed multi-knotoids are {\it rail isotopic} if they are related via a finite sequence of isotopies in the complement of the two rails, {\it keeping the ${\rm O}$-part pointwise fixed}, together with sliding moves of the endpoints along the rails (see Figure~\ref{railsliding}) and rail shifts. The rail isotopy class of a spatial ${\rm O}$-mixed multi-knotoid shall be also referred to as {\it spatial ${\rm O}$-mixed multi-knotoid}. 
\end{definition}

\begin{figure}[H]
\begin{center}
\includegraphics[width=5.5in]{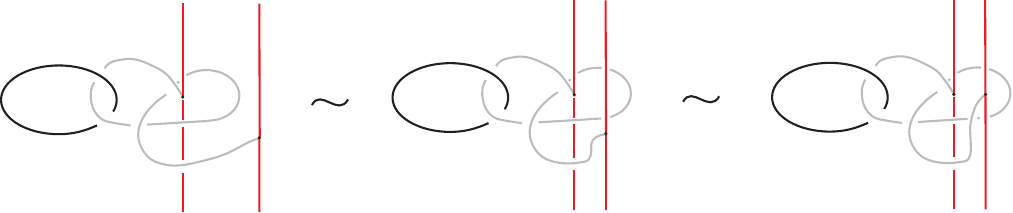}
\end{center}
\caption{A rail shift and a sliding
move of an endpoint along its rail.}
\label{railsliding}
\end{figure}

One can conceptualize a spatial ${\rm O}$-mixed multi-knotoid ${\rm O} \cup K$ as a {\it dichromatic multi-knotoid}, where the ${\rm O}$-part and the $K$-part, except for the rails, are distinguished by color or another identifying characteristic. By analogous arguments as for the theory of mixed links and links in the solid torus \cite{LR1, La2}, we prove that a multi-knotoid $K$ in the thickened annulus is represented unambiguously by a spatial ${\rm O}$-mixed multi-knotoid ${\rm O} \cup K$. Namely:

\begin{theorem}\label{isopmixed}
Rail isotopy classes of multi-knotoids in the thickened annulus are in bijective correspondence with rail isotopy classes of spatial ${\rm O}$-mixed multi-knotoids.
\end{theorem}

\begin{proof}
The correspondence follows directly from the definition of the lift of an annular multi-knotoid diagram in the thickened annulus $\mathcal{A} \times I$ (Definition~\ref{solidlift}) and the definition of rail isotopy of knotoids in $\mathcal{A} \times I$ (Definition~\ref{solidliftisotopy}).
In order to view a multi-knotoid in the thickened annulus as a spatial ${\rm O}$-mixed multi-knotoid, the oriented fixed  ${\rm O}$ component is introduced,  which encodes the topology of the thickened annulus. Rail isotopies of the multi-knotoid in the thickened annulus translate then bijectively to rail isotopies of the spatial ${\rm O}$-mixed multi-knotoid that fix ${\rm O}$ pointwise.
\end{proof}

We shall now take a diagrammatic approach. We first point out that diagrammatic classical knot theory (as well as pseudo knot theory) in the plane (or in the disc) is equivalent to the one in the 2-sphere. However, in the case of annular multi-knotoids, they have to be represented by planar multi-knotoids and not by spherical ones,  because of the natural inclusion of the annulus in a disc and the implication of this fact on the diagrammatic moves. Moreover, the theory of spherical knotoids injects faithfully in the theory of planar knotoids (as classical knot theory injects, in turn, faithfully in the theory of spherical knotoids \cite{T}).  

\begin{figure}[H]
\begin{center}
\includegraphics[width=3.5in]{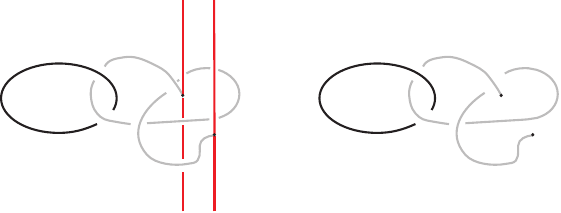}
\end{center}
\caption{An ${\rm O}$-mixed multi-knotoid and its diagram.}
\label{mknotoid}
\end{figure}

We project a spatial ${\rm O}$-mixed multi-knotoid ${\rm O} \cup K$ on the surface  $\mathcal{A} \times \{0\}$ or the once punctured disc, such that the ${\rm O}$-part is also represented as it is (recall that the piercing of the disc is not perpendicular). See Figures~\ref{mknotoid} and~\ref{annular-Omixed}.   Figure~\ref{annular-Omixed} shows the transition among the three different topological settings for annular multi-knotoids. In this context, two annular multi-knotoid diagrams are equivalent if they differ by a finite sequence of punctured disc isotopies and the  three standard Reidemeister moves for multi-knotoids.

\begin{figure}[H]
\begin{center}
\includegraphics[width=5in]{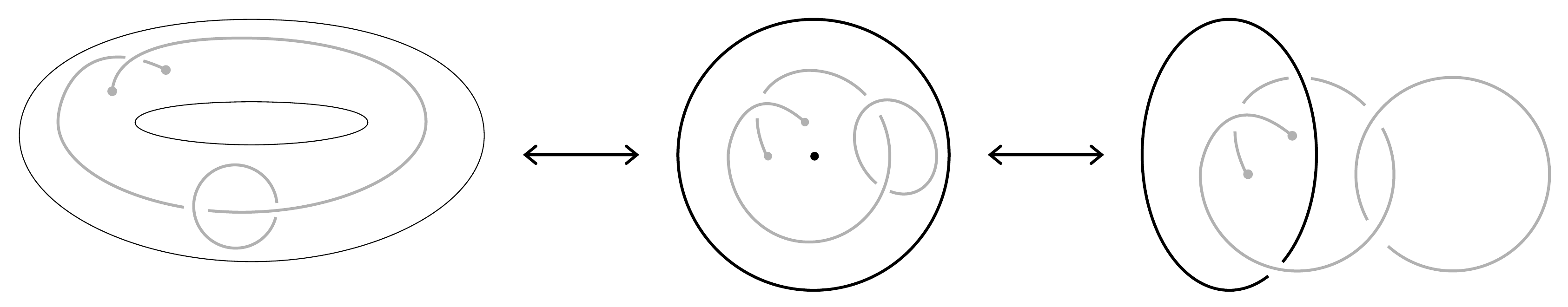}
\end{center}
\caption{From left to right: the transition from the annulus $\mathcal{A}$,  to the ${\rm O}$-mixed setting, to the punctured disc  representation. }
\label{annular-Omixed}
\end{figure}

\begin{definition}\label{def:O-mixed-diagram}
An ${\rm O}$-\textit{mixed multi-knotoid diagram} is a regular projection ${\rm O}\cup D$ of a spatial ${\rm O}$-mixed multi-knotoid ${\rm O}\cup K$ on the plane of ${\rm O}$, which is considered perpendicular to the rails. The double points are crossings, which are either crossings of arcs of the moving part or \textit{mixed crossings} between arcs of the moving and the fixed part, and are endowed with over/under information.  For an example view right-hand side of Figure~\ref{mknotoid}.  
\end{definition}
 
For establishing  a diagrammatic equivalence, we first make the following observation: the endpoints of an annular multi-knotoid, which is represented as an ${\rm O}$-mixed multi-knotoid, are permitted to pass over or under the arcs of the ${\rm O}$-part, contrasting with the restrictions imposed in standard knotoid theory (recall the forbidden moves illustrated in Figure~\ref{forbidoid}). Figure~\ref{almo} illustrates this interaction between the fixed and moving parts of an ${\rm O}$-mixed multi-knotoid diagram. 
We shall call these allowed moves {\it endpoint moves}. The endpoint moves arise naturally from the topology of the setting, since ${\rm O}$ represents the complementary solid torus. 

\begin{figure}[H]
\begin{center}
\includegraphics[width=3.9in]{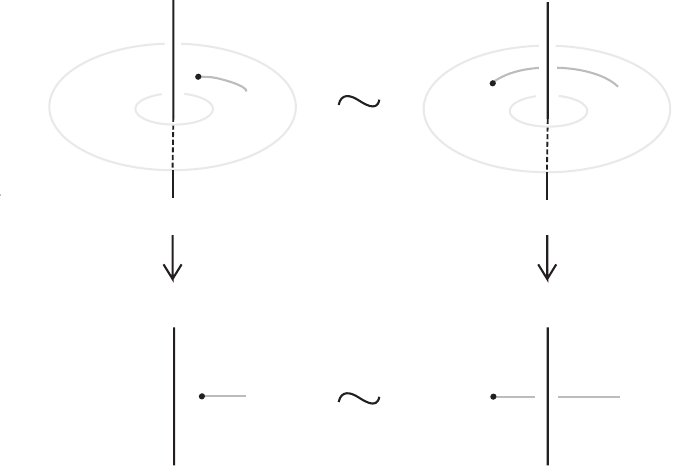}
\end{center}
\caption{The endpoint moves in the ${\rm O}$-mixed setting.}
\label{almo}
\end{figure}

\begin{definition}\label{def:O-mixed-diagram-equiv}
Two (oriented)  ${\rm O}$-mixed multi-knotoid diagrams are said to be  {\it ${\rm O}$-mixed equivalent} if they differ by planar isotopies (which include the endpoint swing moves), a finite sequence of the classical Reidemeister moves (recall Figure~\ref{rmoves}) for their moving parts,  moves that involve the fixed and the moving parts, called {\it mixed Reidemeister moves}, as exemplified in Figure~\ref{mreidm}, and the endpoint moves (exemplified in bottom row of Figure~\ref{almo}). An equivalence class of an ${\rm O}$-mixed multi-knotoid diagram shall be  referred to as {\it ${\rm O}$-mixed multi-knotoid}. 
\end{definition}

\begin{figure}[H]
\begin{center}
\includegraphics[width=2.8in]{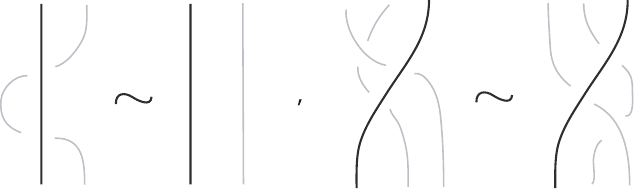}
\end{center}
\caption{The ${\rm O}$-mixed Reidemeister moves.}
\label{mreidm}
\end{figure}

\begin{remark}
 Note that ${\rm O}$-mixed equivalence of ${\rm O}$-mixed multi-knotoid diagrams differs from usual equivalence of multi-knotoid diagrams by the endpoint moves, which remain forbidded in the general theory of multi-knotoids.
\end{remark}

Combining the equivalence of planar multi-knotoid diagrams (recall Section~\ref{sec:sphere-plane}) and the theory of mixed links (\cite{LR1, La2}) we obtain the discrete diagrammatic equivalence of spatial ${\rm O}$-mixed multi-knotoids:

\begin{theorem} \label{Omixedreid}
[The ${\rm O}$-mixed Reidemeister equivalence]\label{O-reidplink}
Two (oriented) spatial ${\rm O}$-mixed multi-knotoids are rail isotopic if and only if any two (oriented) ${\rm O}$-mixed multi-knotoid diagrams of theirs are ${\rm O}$-mixed equivalent.
\end{theorem}

\begin{proof}
We work in the piece-wise linear category. A rail isotopy of a spatial ${\rm O}$-mixed multi-knotoid, ${\rm O}\cup K$, as described in Definition~\ref{def:O-mixed}, is a sequence of $\Delta$-moves applied only on the moving part and taking place in the complement of the two rails, together with sliding moves of endpoints along the rails and rail shifts. A $\Delta$-move on the moving part projects on the ${\rm O}$-plane to a planar isotopy or a Reidemeister move or a mixed Reidemeister move, according to \cite[Theorem 2.1]{GK}. Moreover, a sliding move of an endpoint along its rail can be sufficiently localized, so as to be invisible on the diagrammatic level. Finally, parallel shifts of the rails up to isotopy is translated by isotopies of the endpoints in the plane which do not intersect the component arcs, and thus do not generate a forbidden move. The above ensure that rail isotopy classes of ${\rm O}$-mixed multi-knotoids correspond precisely to ${\rm O}$-mixed equivalence classes of  ${\rm O}$-mixed multi-knotoid diagrams, according to Definition~\ref{def:O-mixed-diagram}.
\end{proof}

\begin{remark}
Theorem~\ref{Omixedreid} also applies to the case of linkoids. The definitions and moves extend naturally to linkoids, and the same Reidemeister equivalence holds for their spatial ${\rm O}$-mixed representations.
\end{remark}

\begin{remark}
In terms of ${\rm O}$-mixed knotoids, the inclusion of the annulus in a disc  corresponds to omitting  the fixed part ${\rm O}$, so that we are left with a planar knotoid. 
\end{remark}

In the above so far we lifted annular knotoid diagrams in the thickened annulus and we related multi-knotoids in the thickened annulus to spatial ${\rm O}$-mixed multi-knotoids and, subsequently, to planar ${\rm O}$-mixed multi-knotoids. Then, Theorems~\ref{isopST},~\ref{isopmixed} and~\ref{Omixedreid} combine into the following diagrammatic equivalence of the two different settings.

\begin{theorem}\label{annulartoOmixed}
Two annular knotoid diagrams are Reidemeister equivalent if and only if any two corresponding planar ${\rm O}$-mixed multi-knotoid diagrams of theirs are ${\rm O}$-mixed Reidemeister equivalent.
\end{theorem}

 \begin{note} In \cite{D4} for the study of  knotoids in the solid torus,  a multi-knotoid is defined as an ${\rm O}$-mixed multi-knotoid.
\end{note}

\section{The (universal) bracket polynomial for annular knotoids}\label{sec:annular-bracket} 

In this section we construct analogues of the Kauffman bracket polynomial for classical knots and links for the setting of annular multi-knotoids. In \cite{T} Turaev defines a bracket polynomial for multi-knotoids in any oriented surface, where any loop, essential, null-homotopic or nesting the trivial knotoid, gets the same value $d$. We define bracket polynomials for multi-knotoids in $\mathcal{A}$ that distinguish these different classes of components, leading to a 4-variable bracket polynomial and its extension to infinitely many variables. 
 
Consider now an annular multi-knotoid diagram and apply to its crossings smoothings as presented in \cite{T} and as illustrated in  Figure~\ref{smoothing}. After smoothing all crossings we arrive at a {\it state diagram} which may contain a number $n$ of inner essential unknots, in the sense that they do not enclose the trivial knotoid, a number $k$ of null-homotopic unknots, a number $m$ of non-essential unknots enclosing the trivial knotoid, and a number $l$ of outer essential unknots (which enclose all the rest). The inner and outer essential unknots are separated by the (possibly nested) trivial knotoid, due to the forbidden moves. An abstract example is shown in the left-hand side of Figure~\ref{fig:annular-turaev}. This formation is what we call a {\it generic state} for an annular multi-knotoid diagram.

\begin{figure}[H]
\begin{center}
\includegraphics[width=5in]{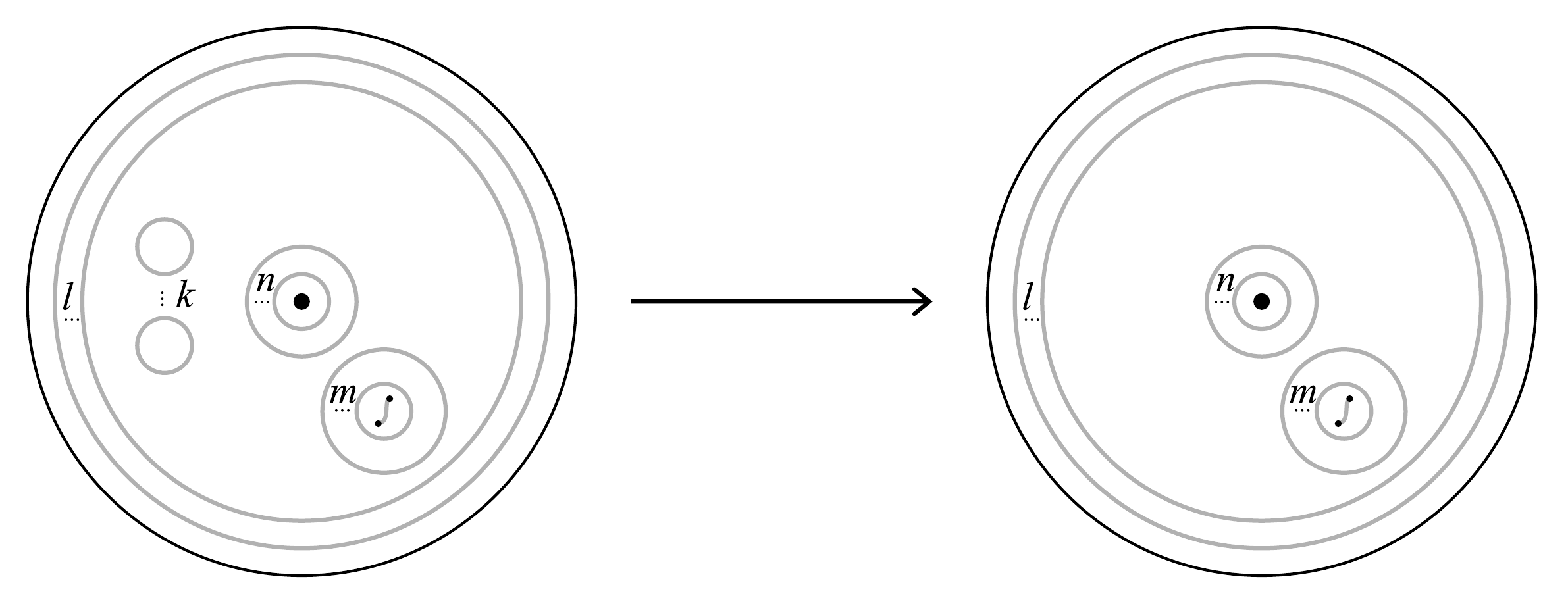}
\end{center}
\caption{A generic state for an annular multi-knotoid diagram before and after removing the null-homotopic components.}
\label{fig:annular-turaev}
\end{figure}

\newpage
\subsection{The annular bracket polynomial}

We are now ready to introduce:
 
\begin{definition}\label{pkaufbst}\rm
The {\it annular bracket polynomial} of an annular (multi)-knotoid diagram is defined inductively on the number of crossings in the diagram and the closed components in its states, by means of the following rules: 
\begin{itemize}
\item[i.] $\langle \ \raisebox{-2pt}{\includegraphics[scale=0.8]{4-1-1.pdf}} \ \rangle = A \langle \ \raisebox{0pt}{\includegraphics[scale=0.5]{skein-2.pdf}} \ \rangle + A^{-1} \langle \ \raisebox{-2.5pt}{\includegraphics[scale=0.6]{skein-3.pdf}} \ \rangle$
   \vspace{.1cm}
\item[ii.]  $\langle \ \raisebox{-2pt}{\includegraphics[scale=0.85]{4-1-2.pdf}} \ \rangle = A^{-1} \langle \ \raisebox{0pt}{\includegraphics[scale=0.5]{skein-2.pdf}} \ \rangle + A \langle \ \raisebox{-2.5pt}{\includegraphics[scale=0.6]{skein-3.pdf}} \ \rangle $
  \vspace{.1cm}
 \item[iii.]$\langle L \sqcup \, \mathrm{O}^k \rangle  = d^k \, \langle L \rangle$, where $d = -A^2-A^{-2}$ and $\mathrm{O}^k$ stands for $k \in \mathbb{N}\cup \{0\}$ null-homotopic unknots.
       
   \vspace{.1cm}
     \item[iv.] $\langle  \raisebox{-9pt}{\includegraphics[scale=0.05]{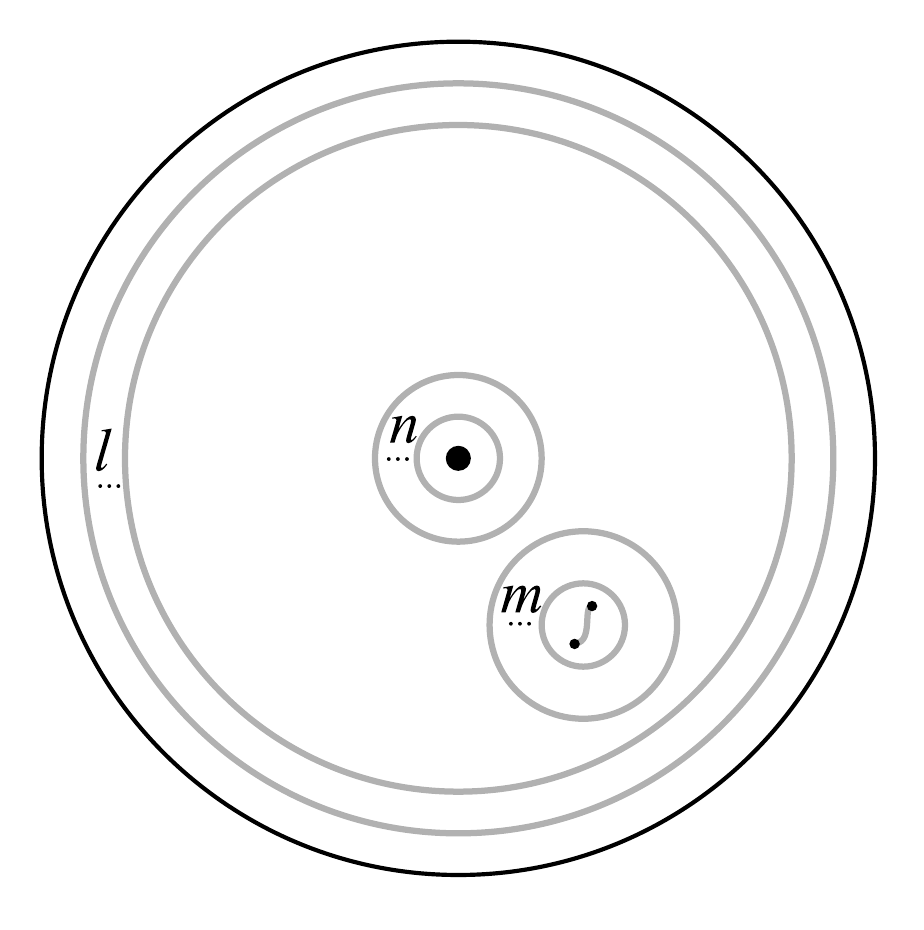}} \rangle  = x^n v^m t^l \langle \raisebox{-1pt}{\includegraphics[scale=0.1]{arc.pdf}} \rangle$ ,  where $n, m, l \in \mathbb{N}\cup \{0\}$,  $n$ is the number of inner essential unknots, $m$ is the number of non-essential unknots enclosing the trivial knotoid and $l$ is the number of outer essential unknots  (see right-hand side of  Figure~\ref{fig:annular-turaev} for an enlarged picture).
   %  \vspace{.2cm}
     \item[v.] $\langle \raisebox{-1pt}{\includegraphics[scale=0.1]{arc.pdf}} \rangle = 1$.
\end{itemize}
\end{definition}

Setting $x=v=t=d$ the annular bracket polynomial boils down to the  Turaev bracket for surface $\Sigma = \mathcal{A}$~\cite{T}.

\begin{theorem}\label{th:pkaufbst}
The annular bracket polynomial of an annular (multi)-knotoid diagram $K$ is a well defined 4-variable Laurent polynomial $ \langle K \rangle  \in \mathbb{Z}[A^{\pm 1}, x, v, t]$. Moreover, the annular bracket polynomial is a regular isotopy invariant, that is, invariant under Reidemeister moves R2 and R3.
\end{theorem}

\begin{proof}
Regarding first the well-definiteness: In analogy to the planar case, we can compute the annular bracket $\langle K \rangle$ of an annular multi-knotoid $K$ by applying the rules in Definition~\ref{pkaufbst}. We first  smoothen all crossings by rules (i) and (ii).  Considering the inclusion of the annulus in a disc, we  use the fact that the planar bracket is well-defined on planar (multi)-knotoid diagrams, that is, it is independent of the order of smoothings of the crossings. 

We then eliminate in each state all trivial unknots that are null-homotopic in $\mathcal{A}$, using rule (iii). Note that a null-homotopic unknot in $\mathcal{A}$ may lie in between the essential unknots, even between the non-essential unknots enclosing the trivial knotoid. Yet, its position is invisible when applying rule (iii). For this reason, in a generic state all $k$ null-homotopic unknots can be assumed to lie within a local disc in the same region of the diagram.  
 
We end up with generic state diagrams consisting of a number of non-essential unknots enclosing the trivial knotoid and a number of essential unknots, some  enclosing the trivial knotoid and some not, as in the right-hand side of  Figure~\ref{fig:annular-turaev}. In this reduced state diagram, we are now in a position to apply rule (iv). Note that the  relative positions of the unknots in each one of the above three sets  are interchangeable, since this is not visible by rule (iv), which only counts their numbers. At this point we observe that the status of each set of unknots in a generic state, like in Figure~\ref{fig:annular-turaev}, is  locked because of the forbidden moves and the topology of the annulus, so the integers $k,l,m,n$ are fixed in every state. Therefore,  application of rules (iii) and (iv) is unambiguous. Finally rule (v) applies to the trivial knotoid. Hence, the annular bracket polynomial is well-defined on any given (multi)-knotoid diagram by the rules (i)--(v).

Moreover, considering again the inclusion of the annulus in a disc, implies that the invariance of the annular bracket polynomial under the  order of smoothings of  crossings rests on the invariance of the planar bracket polynomial under this order, since the underlying closed curves in the states will be identical and coefficients carry through.

Finally, for  the invariance  of the annular bracket polynomial under the Reidemeister moves R2 we check that the same arguments as for the classical bracket polynomial carry through for any arcs involved in the moves, while the Reidemeister moves R3 follow from the R2 moves. Hence, the annular bracket polynomial is a regular isotopy invariant, that is, invariant under Reidemeister moves R2 and R3.
\end{proof}

\subsection{The universal annular bracket polynomial} 

In analogy to the planar case, rule (iv) can be generalized to an infinite variable expression:

\begin{definition}\label{annularskein}
The {\it universal annular bracket polynomial} for annular multi-knotoids, denoted $ \langle \cdot \rangle_U $, is defined by means of rules (i), (ii), (iii), (v) of Definition~\ref{pkaufbst} and rule (iv$^{\prime}$) below:
\begin{center}
iv$^\prime$. \qquad $\langle  \raisebox{-9pt}{\includegraphics[scale=0.05]{punctured-an-turaev.pdf}} \rangle_U   = x_n v_m t_l \langle \raisebox{-1pt}{\includegraphics[scale=0.1]{arc.pdf}} \rangle_U$
\end{center}
 where we alternatively  substitute $x^n, v^m, t^l$ by $x_n, v_m, t_l$ respectively, with $n, m, l$ being non-negative integers. 
\end{definition}

Definition~\ref{annularskein} leads to the following result.

\begin{theorem} \label{thm:annularskein}
The  universal annular bracket polynomial  is an infinite variable Laurent polynomial $ \langle \cdot \rangle_U  \in \mathbb{Z}\left[A^{\pm 1}, x_n, v_m, t_l,\, \ n, m, l \in \mathbb{N}\cup \{0\}\right]$, that is a regular isotopy invariant for annular multi-knotoids, and  realizes the Kauffman bracket skein module of annular multi-knotoids.
\end{theorem}

\begin{proof}
The proof of the theorem is completely analogous to the proof of Theorem~\ref{th:pkaufbst}, in combination with the proof of Theorem~\ref{th:universal-bracket-regular-isotopy} and taking  into account Definition~\ref{annularskein}.    
\end{proof}

In Section~\ref{computing_bracket} we provide some computations of the annular and the universal annular bracket polynomials on specific examples.

\begin{remark}\label{KBSMannulus}
The above generalized definition for the annular bracket polynomial realizes the {\it Kauffman bracket skein module for annular multi-knotoids} \cite{D4,GG}, which generalizes the Kauffman bracket skein module (KBSM) of the solid torus \cite{HP} and is conceptually analogous  to the KBSM of the handlebody of genus 2 (see \cite{P}). In particular, the KBSM for annular multi-knotoids is freely generated by the state diagrams illustrated in the left-hand side of  Figure~\ref{fig:annular-turaev} and described in the text above it. 
\end{remark}

\begin{remark} 
If we have an annular multi-knotoid diagram that contains no essential closed curves, by using the inclusion of a disc in $\mathcal{A}$ (see left-hand illustration in Figure~\ref{planarannular}), such annular diagrams can be viewed as planar diagrams and resolve into states as in the planar case.
In contrast, using the inclusion of $\mathcal{A}$ in a disc (see right-hand illustration in Figure~\ref{planarannular}), the $n$ essential unknots become null-homotopic, so substituting $x^n$ (resp.  $x_n$) by $d^{n} = (-A^2-A^{-2})^{n}$, the annular (resp. universal annular) bracket polynomial specializes to the planar (resp. universal planar) bracket of Definition~\ref{pkaufb} (resp. Definition~\ref{u-planarskein}) for planar knotoids.  
\end{remark}

\begin{remark}
The notion of a {\it knot type annular multi-knotoid}, whereby the endpoints lie in the same diagrammatic region, is not well-defined as explained by the following example: the trivial knotoid can be local and correspond to a null-homotopic unknot. But it can also be isotoped to an  essential unknot with a segment removed. So, the corresponding knot type in this case would be the essential unknot. This implies that the under- or over-closure \cite{T} of an annular multi-knotoid is also not well defined. So, the evaluation of the (universal) annular bracket polynomial  of the corresponding under- or over-closure link in the thickened annulus is not well defined either.  %(after specializing $v$ to the loop value $d$ (resp. $v_m$ to $d^m$) and $t$ to $x$ (resp. $t_l$ to $x_l$)). 
\end{remark}

\subsection{State summation formuli for the annular bracket polynomials} 
 
In complete analogy to the classical bracket, the spherical and the (universal) planar bracket polynomials, one can arrive at an alternative definition of the annular bracket and the universal  annular bracket polynomials via a state summation: 

\begin{definition}
The annular bracket and the universal annular bracket polynomials of an annular multi-knotoid diagram $K$ are defined as state summations: 
\begin{equation*}\label{sskbst}
\langle K \rangle\ = \underset{s\in S(K)}{\sum}\, A^{\sigma_s}\,  d^{k_s}\, x^{n_s}\, v^{m_s} \, t^{l_s} \quad {\rm and} \quad \langle K \rangle_U\ = \underset{s\in S(K)}{\sum}\, A^{\sigma_s}\, d^{k_s}\, x_{n_s} \, v_{m_s} \, t_{l_s} 
\end{equation*} 
\noindent where $d= -A^2-A^{-2}$,  $\sigma_s$ is the number of A-smoothings minus the number of B-smoothings applied to the crossings of $K$ in order to obtain the state $s$ (recall Figure~\ref{smoothing}), $k_s$ is the number of null-homotopic unknots in $s$, $n_s$ is the number of inner essential unknots in $s$, $m_s$ is the number of non-esssential unknots in $s$ that enclose the trivial knotoid, and $l_s$ is the number of outer essential unknots, enclosing all the rest, in the state $s$.
\end{definition}

\subsection{The annular Jones polynomials}

Furthermore, as in the case of the (universal) planar bracket polynomial,  the (universal) annular bracket polynomial can be normalized so as to also respect Reidemeister move R1, by considering the product of $\langle K \rangle$ by the writhe of an annular multi-knotoid diagram $K$, which is defined as follows.

\begin{definition} \label{def:annular_writhe}
The \textit{writhe} of an oriented annular multi-knotoid $K$, denoted as $w(K)$, is defined as the number of positive crossings minus the number of negative crossings of $K$ (recall Figure~\ref{si}). 
\end{definition}

\noindent This leads to obtaining the {\it normalized (universal)  annular  bracket polynomials}.

\begin{theorem}\label{th:jones-an}
Let $K$ be an annular multi-knotoid diagram. The {\rm annular Jones} and the {\rm universal annular Jones polynomials} defined as: 
\[
V(K)\ =\ (-A^3)^{-w(K)}\, \langle K \rangle\ \ \ {\rm and}\ \ \ V_U(K)\ =\ (-A^3)^{-w(K)}\, \langle K \rangle_U
\]
\noindent where $w(K)$ is the writhe of $K$,  $A = t^{-1/4}$, $\langle K \rangle$ the annular bracket polynomial of $K$ and $\langle K \rangle_U$  the universal annular bracket polynomial of $K$, are isotopy invariants of annular knotoids.
\end{theorem}

\subsection{The ${\rm O}$-mixed bracket and Jones polynomials}

In view of Theorem~\ref{Omixedreid}, the annular and universal annular bracket polynomials can be adapted to the setting of (oriented) ${\rm O}$-mixed multi-knotoids. Here the mixed crossings are not subjected to the inductive rules, since ${\rm O}$ must remain pointwise fixed throughout. 
 More precisely, the (universal) annular bracket polynomials translate as follows:

\begin{definition}\label{mix-pkaufbst}\rm
Let ${\rm O}\cup K$ be an (oriented) ${\rm O}$-mixed multi-knotoid. The \textit{${\rm O}$-mixed bracket polynomial} of ${\rm O}\cup K$ is defined by means of the same inductive rules as the ones for the annular bracket polynomial (Definition~\ref{pkaufbst}), except for the diagram in  rule (iv) which is substituted by the planar diagram in Figure~\ref{fig:KaufOmixed}. Here, the $n$ `inner' as well as the $l$ `outer' essential unknots are linked with the fixed unknot~${\rm O}$. Moreover, $L$ in rule (iii) stands now for an (oriented) ${\rm O}$-mixed multi-knotoid ${\rm O}\cup L$.

\noindent Similarly, the \textit{universal ${\rm O}$-mixed bracket polynomial} of ${\rm O}\cup K$ is defined by means of the same inductive rules as the  ${\rm O}$-mixed bracket polynomial, except that the variables $x, v, t$ in rule (iv) are replaced by the infinitely many variables $x_n, v_m,  t_l$ for $n, m, l \in \mathbb{N}\cup \{0\}$.
\end{definition}

\begin{figure}[H]
    \centering
    \includegraphics[width=0.3\linewidth]{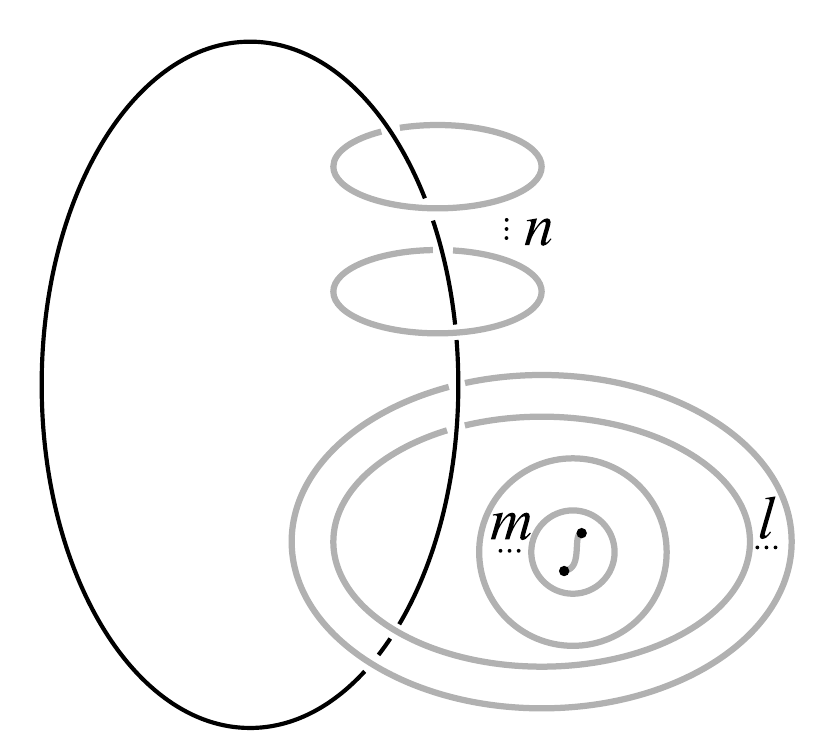}
    \caption{A generic state for an ${\rm O}$-mixed multi-knotoid diagram after removing the null-homotopic components.}
    \label{fig:KaufOmixed}
\end{figure}

\begin{remark} 
If we start from an ${\rm O}$-mixed multi-knotoid and do all smoothings of moving crossings, we would arrive at sets of closed essential curves winding around the fixed ${\rm O}$-component  some numbers of times, and not just once like the   $n$  and $l$ essential unknots depicted in Figure~\ref{fig:KaufOmixed}. This would lead to extended states (including the ones in Figure~\ref{fig:KaufOmixed}) and eventually to an alternative  definition of the ${\rm O}$-mixed bracket polynomial with extended initial conditions. The extended states have been used in the braid approach to the skein module of the solid torus. For details cf. \cites{La2,D2} and  references therein. This alternative definition is related to Definition~\ref{mix-pkaufbst} through a change of basis of the Kauffman bracket skein module of the solid torus. 
\end{remark}

We can now state the following:

\begin{theorem}\label{th:O-pkaufbst}
The ${\rm O}$-mixed bracket polynomial of an (oriented) ${\rm O}$-mixed multi-knotoid diagram $K \cup \, \mathrm{O}$ is a well-defined 4-variable Laurent polynomial $ \langle K \rangle  \in \mathbb{Z}[A^{\pm 1}, x, v, t]$, which is invariant for ${\rm O}$-mixed multi-knotoids under Reidemeister moves R2 and R3. Further, the universal ${\rm O}$-mixed bracket polynomial of $\mathrm{O} \cup \, K$ is a well-defined, infinite variable Laurent polynomial $ \langle \cdot \rangle  \in \mathbb{Z}\left[A^{\pm 1}, x_n, v_m, t_l,\, \ n, m, l \in \mathbb{N}\cup \{0\}\right]$, which is a regular isotopy invariant for ${\rm O}$-mixed multi-knotoids that realizes the Kauffman bracket skein module of ${\rm O}$-mixed multi-knotoids. 
\end{theorem}

\begin{proof}
The well-definiteness of the ${\rm O}$-mixed / universal  ${\rm O}$-mixed bracket polynomials is a direct consequence of the bijective correspondence between annular multi-knotoids and ${\rm O}$-mixed multi-knotoids (Theorem~\ref{annulartoOmixed}) and of the defining rules. Moreover, the invariance under Reidemeister moves R2 and R3 is a direct consequence of the invariance under Reidemeister moves R2 and R3 of the planar bracket polynomial (Section~\ref{sec:planar-bracket}), after noting that the endpoint move is invisible in the computations since mixed crossings are not smoothed.
\end{proof}

In analogy to the annular bracket polynomials we also have here closed state summation formuli. Finally, the ${\rm O}$-mixed / universal  ${\rm O}$-mixed bracket polynomials can also be normalized so as to also respect Reidemeister move R1. 

\begin{definition}
The ${\rm O}$-\textit{writhe} of an oriented ${\rm O}$-mixed multi-knotoid ${\rm O}\cup K$ shall be denoted $w({\rm O}\cup K)$ and is defined as the number of positive crossings minus the number of negative crossings of ${\rm O}\cup K$ (recall Figure~\ref{si}), while the mixed crossings do not contribute to the ${\rm O}$-writhe.
\end{definition}

\noindent This leads to obtaining the {\it normalized ${\rm O}$-mixed} and {\it universal ${\rm O}$-mixed bracket polynomials}.

\begin{theorem}
Let $K$ be an ${\rm O}$-mixed multi-knotoid diagram. The {\rm ${\rm O}$-mixed Jones polynomial} and the {\rm universal ${\rm O}$-mixed Jones polynomial} 
\[
V({\rm O}\cup K)\ =\ (-A^3)^{-w({\rm O}\cup K)}\, \langle {\rm O}\cup K \rangle\ \ \ {\rm and}\ \ \ V_U({\rm O}\cup K)\ =\ (-A^3)^{-w({\rm O}\cup K)}\, \langle {\rm O}\cup K \rangle_U\
\]
\noindent where $w({\rm O}\cup K)$ is the ${\rm O}$-writhe of ${\rm O}\cup K$,  $A = t^{-1/4}$, $\langle {\rm O}\cup K \rangle$ the ${\rm O}$-mixed bracket polynomial of ${\rm O}\cup K$ and $\langle {\rm O}\cup K \rangle_U$  the universal ${\rm O}$-mixed bracket polynomial of ${\rm O}\cup K$, are isotopy invariants of ${\rm O}$-mixed multi-knotoids.
\end{theorem}

\section{Toroidal knotoids}\label{sec:toroidal-knotoid}

Building on the above, we now extend the theory of annular knotoids and multi-knotoids to the torus, collectively referred to as toroidal knotoids. In Figure~\ref{ktt} we illustrate a knotoid, a linkoid, a multi-knotoid and a multi-linkoid in the torus. 

\begin{figure}[H]
\begin{center}
\includegraphics[width=6in]{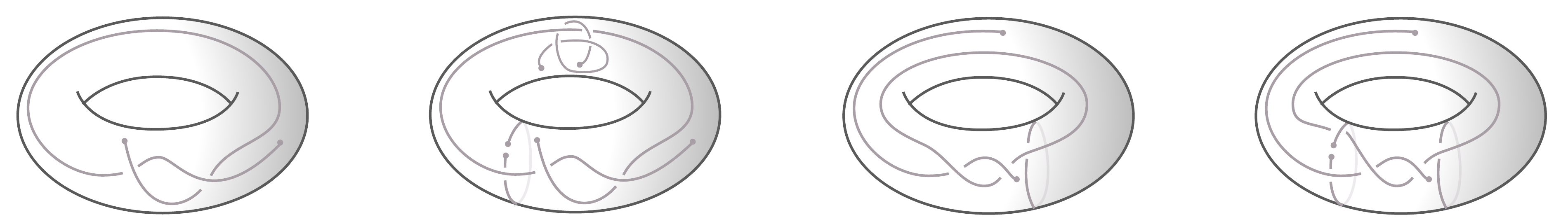}
\end{center}
\caption{A knotoid, a linkoid, a multi-knotoid and a multi-linkoid in the torus.}
\label{ktt}
\end{figure}

In analogy to the annular case, we study toroidal knotoids in different geometric contexts, such as through their lifts  in the thickened torus and through their representations as planar ${\rm H}$-mixed knotoids, that is, planar multi-knotoids containing a fixed Hopf link.  We  highlight inclusion relations between planar, annular and toroidal knotoids and we  define the Turaev loop bracket polynomial for toroidal knotoids and ${\rm H}$-mixed knotoids.

\subsection{Toroidal knotoid diagrams and their equivalence}\label{sec:toroidal-equivalence}

\begin{definition}
A {\it toroidal multi-knotoid diagram} is a smooth immersion of the unit interval $[0, 1]$ and a number of circles in the torus $T^2$, such that the only singularities are transversal double points, the crossings, endowed with over/under crossing information. A {\it toroidal knotoid diagram} is a multi-knotoid diagram with no closed components.

\noindent A toroidal knotoid diagram inherently possesses a natural orientation from its {\it leg} (the image of $0$) to its {\it head } (the image of $1$). An {\it oriented } toroidal multi-knotoid diagram is obtained by also assigning orientations to the closed components.
\end{definition}

An important type of essential closed components in toroidal multi-knotoid diagrams are torus knots and torus links. We recall that a \textit{$(p,q)$-torus knot}, for $p$ and $q$  relatively prime integers, is an essential simple closed curve in the torus $T^2$ winding $p$ times around the longitude and $q$ times along the meridian of the torus. Similarly, a \textit{$(p',q')$-torus link} or a \textit{$(p,q)$-torus link on $l$ components} is a set of $l$ parallel copies of a $(p,q)$-torus knot, where $p'= l \cdot p$ and $q'= l \cdot q$. In our definition of torus knots and links, we include the pairs $(p,0)$ and $(0,q)$. 

\begin{remark}\label{rem:essentialpq}
A $(p,q)$-torus knot, considered as a classical knot in $S^3$, is isotopic to a $(q,p)$-torus knot.  However, as essential curves embedded in the torus, these two knots are distinct in $T^2$. Still, note that in $T^2$ a $(p,q)$-torus knot is the same as a $(-p,-q)$-torus knot, and similarly, a $(p,-q)$-torus knot is the same as a $(-p,q)$-torus knot. 
\end{remark}

\begin{definition}\label{torpl}
An (oriented) {\it toroidal multi-knotoid} is defined as an equivalence class of (oriented)  toroidal multi-knotoid diagrams under surface isotopies and all versions of the (oriented)  classical Reidemeister moves, as exemplified in Figure~\ref{rmoves}, all together comprising the (oriented)  {\it Reidemeister equivalence} for toroidal multi-knotoids.  
\end{definition}

The toroidal Reidemeister equivalence, and in particular torus isotopy, includes the endpoint swing moves in their diagrammatic regions as well as the `toroidal moves'. These are global moves, analogous to the spherical move that makes a crucial difference between planar and spherical multi-knotoids.

\begin{definition}\label{def:toroidal-moves}
A {\it toroidal move} is a  move induced by an isotopy of $T^2$, whereby an arc can  slide either from the outer part to the inner part of the torus along a meridian and vice versa, and is called {\it longitudinal move}, or from the one side to the other side of the torus along a longitude, and is called {\it meridional move}.
\end{definition}

Examples of longitudinal moves are illustrated in Figures~\ref{toroidal-move} and ~\ref{toroidal-move-knotoid} (top). The longitudinal moves make a crucial distinction between toroidal and annular multi-knotoids, as illustrated in Figure~\ref{toroidal-move-knotoid} (bottom). In this figure, the trivial knotoid can slide from the one side of the essential curve to the other side, a move that cannot be realized in the annulus. However, the meridional move has a direct analogue in the annulus, so it cannot make a distinction between the two settings.

\begin{figure}[H]
\begin{center}
\includegraphics[width=5.7in]{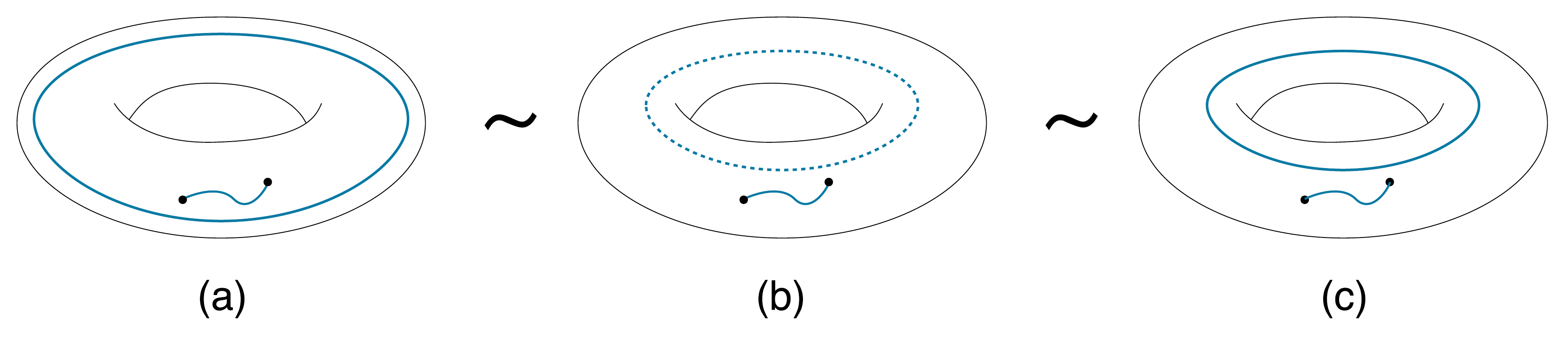}
\end{center}
\caption{A longitudinal toroidal  move allows an essential unknot to move along a meridian from (a) to (b) and then to (c).}
\label{toroidal-move}
\end{figure}

\begin{figure}[H]
\begin{center}
\includegraphics[width=4.5in]{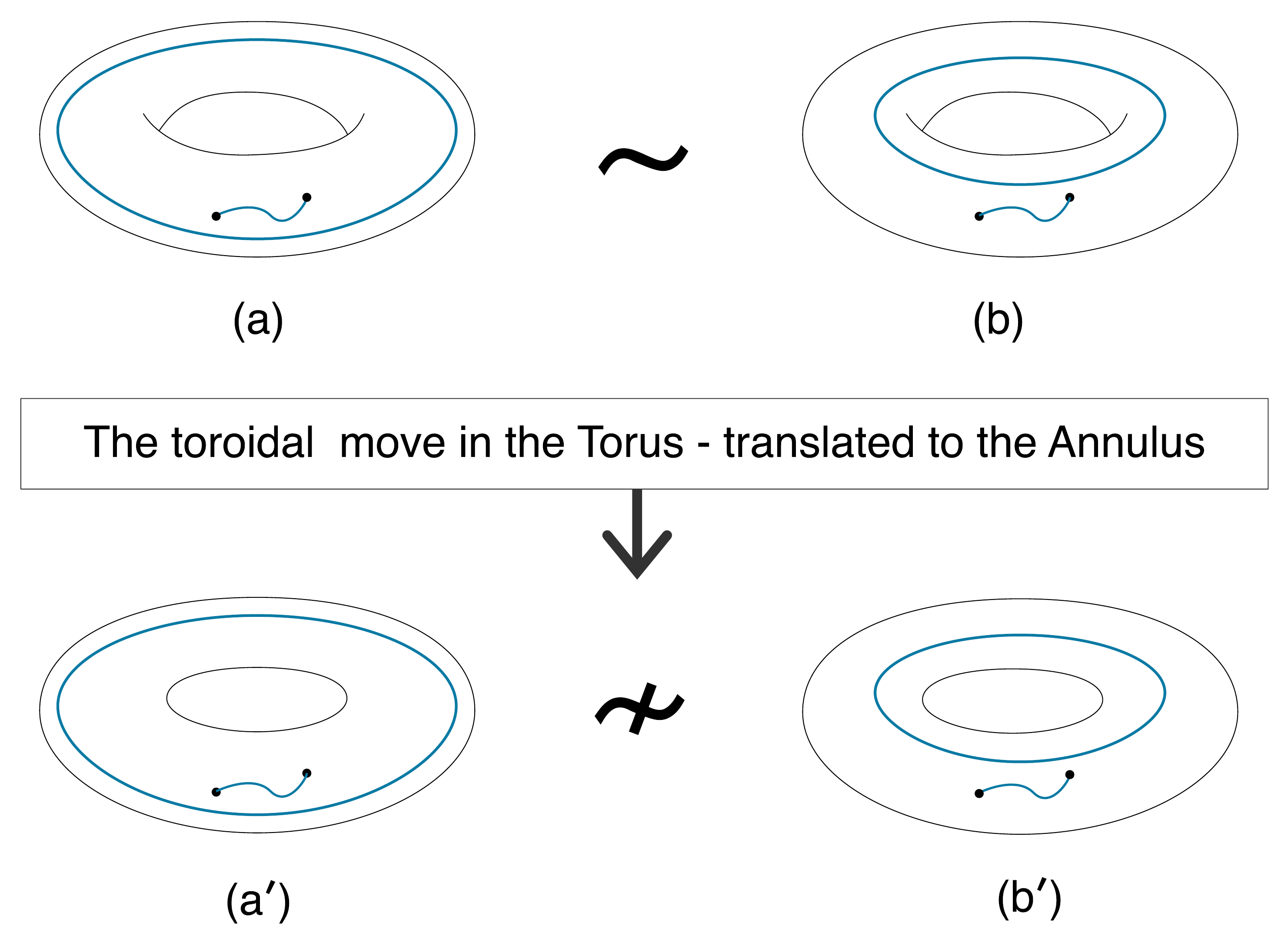}
\end{center}
\caption{Top: the longitudinal toroidal move allows the trivial knotoid to move along a meridian from (a) to (b). Bottom: the longitudinal toroidal move is not allowed for annular multi-knotoids due to the forbidden  moves.}
\label{toroidal-move-knotoid}
\end{figure}

\begin{remark}\label{trivial_pq}
As another interesting observation we point out that a knotoid in the shape of a $(p,q)$-torus curve is just the trivia knotoid, by surface isotopy. 
\end{remark}

\subsection{Lifting toroidal knotoids}\label{sec:toroidal-lift}

We shall now define the lift of a toroidal multi-knotoid as a `rail curve' and a collection of closed curves in the thickened torus, in analogy to spatial multi-knotoids and multi-knotoids  in the thickened annulus (recall Subsections~\ref{sec:planar-lift}) and~\ref{sec:annular-lift}). We consider the thickening $T^2 \times I$, where $I$ denotes the unit interval. View  Figure~\ref{ex3}, where the two dotted curves denote the inner toroidal boundary of $T^2 \times I$. 

\begin{figure}[H]
\begin{center}
\includegraphics[width=2.2in]{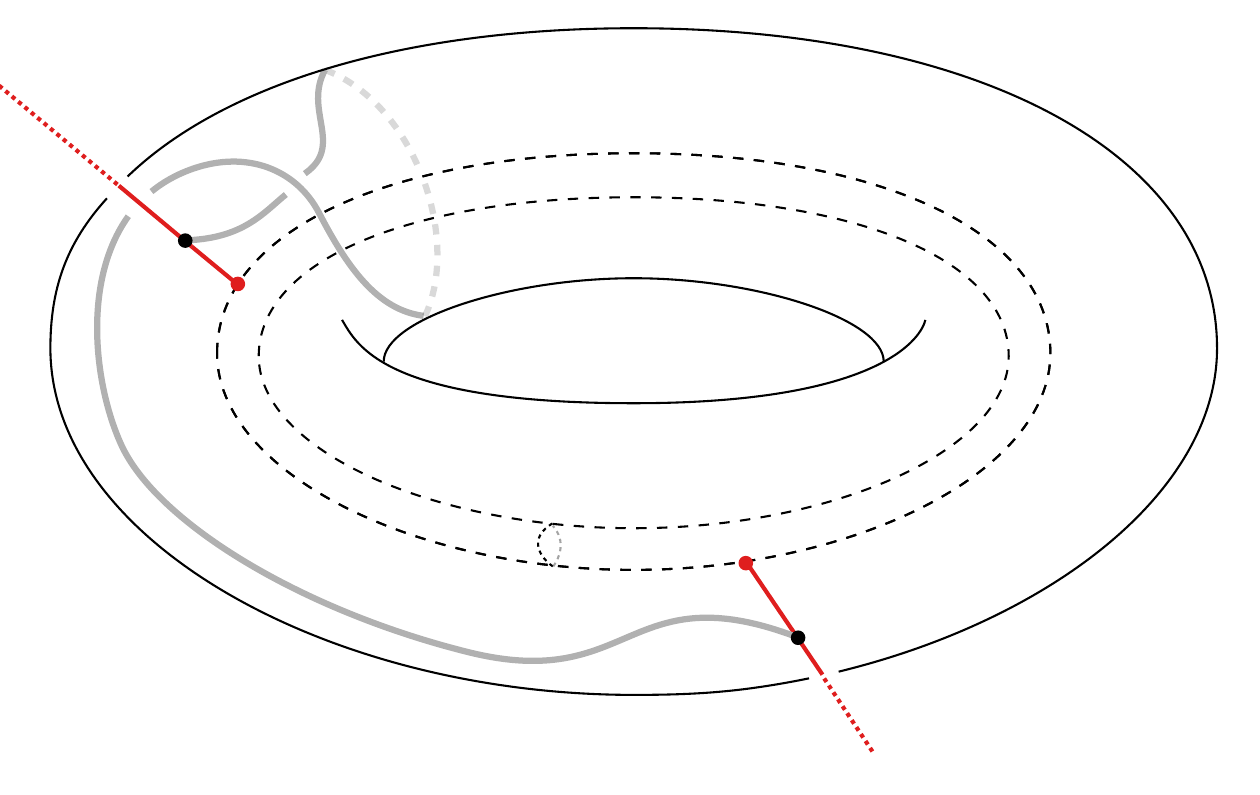}
\end{center}
\caption{Lifting a toroidal knotoid in the thickened torus, where the rails are extended for visual purposes.}
\label{ex3}
\end{figure}

\begin{definition} \label{thickenedlift}
The {\it lift} of an (oriented) toroidal multi-knotoid diagram in the thickened torus $T^2 \times I$ is defined so that: each classical crossing is embedded in a sufficiently small 3-ball that lies in the interior of the thickened torus, the simple arcs connecting the crossings can be replaced by isotopic ones in the interior of $T^2 \times I$, and the endpoints are attached to two parallel segments, the rails $L$ and $H$, which are perpendicular to the torus $T^2$ and lie in $T^2 \times I$. The resulting lifted multi-knotoid together with its two rails is called a {\it multi-knotoid in the thickened torus} and it is a collection of an open curve with its endpoints attached to the rails, and a number of closed curves, all constrained to the interior of the thickened torus $T^2 \times I$, with the exception of the endpoints which may lie in  $T^2 \times \{0\}$ and $T^2 \times \{1\}$. 
\end{definition}

An example is illustrated in Figure~\ref{ex3}, where the rails are extended for visual purposes. Further, for a better visualization of the lift, Figure~\ref{lift} bottom row illustrates comparative flat lifts of a planar, an annular and a toroidal multi-knotoid, where the annulus and the torus (in top row) are described as identification spaces. 

\begin{figure}[H]
\begin{center}
\includegraphics[width=5in]{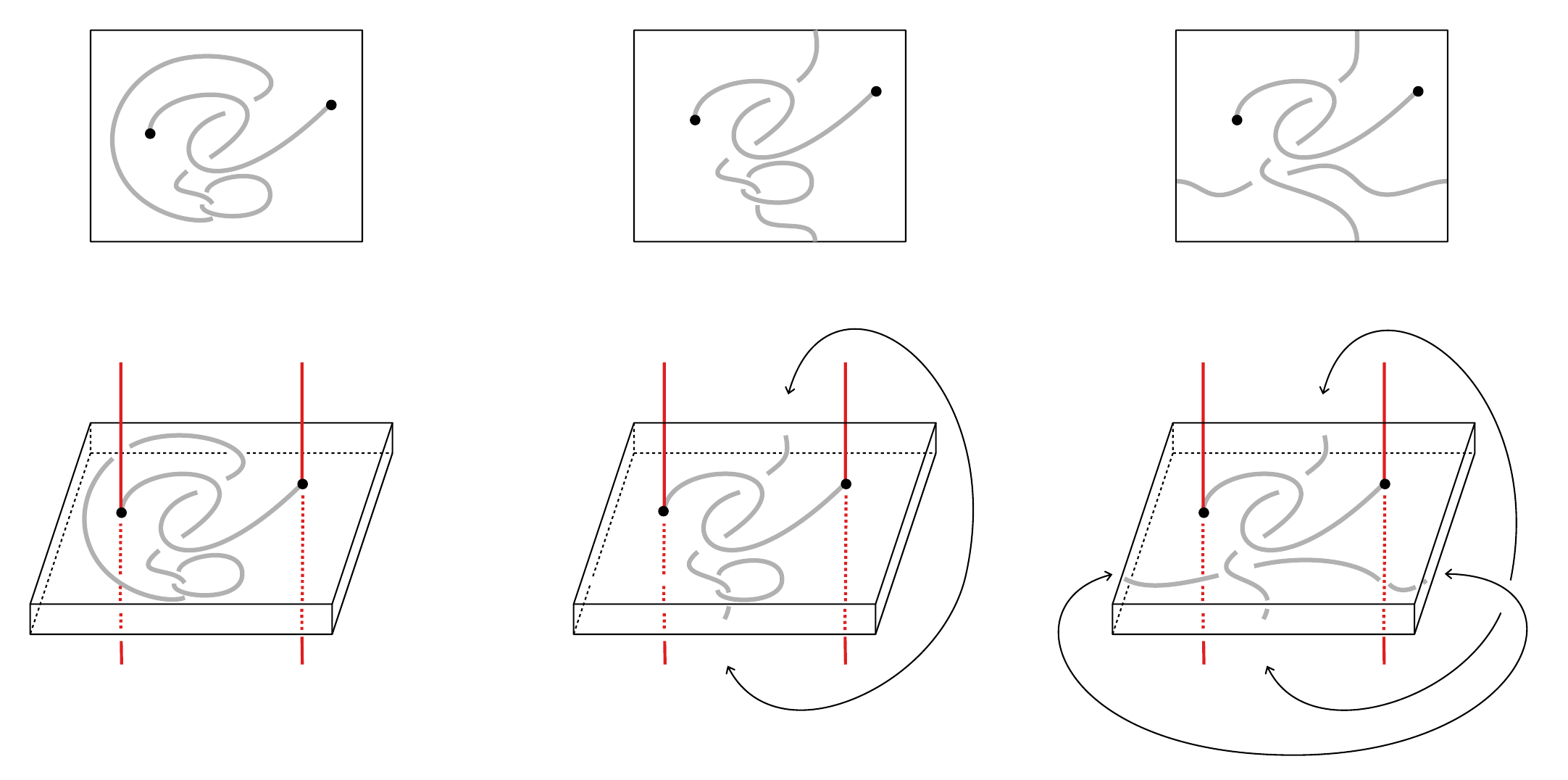}
\end{center}
\caption{Flat liftings of planar, annular and toroidal multi-knotoids, where the rails are extended beyond the thickenings for visual purposes, and where the identifications are pointed by arrows.}
\label{lift}
\end{figure}

\begin{definition} \label{thickenedliftisotopy}
Two (oriented) multi-knotoids in the thickened torus $T^2 \times I$ are said to be {\it rail isotopic} if they are related by an ambient isotopy of $T^2 \times I$ that takes place in its interior and in the complement of the two rails  $L$ and $H$, keeping the endpoints of the open curve attached to their respective rails, being allowed to slide along them ({\it rail sliding}), together with local shifts of the rails up to isotopy, the {\it rail shifts} (for an example see Figure~\ref{endpm} and compare with Subsections~\ref{sec:planar-lift}) and~\ref{sec:annular-lift}). A rail isotopy class shall be also referred to as {\it  multi-knotoid in the thickened torus}. 
\end{definition}

\begin{figure}[H]
\begin{center}
\includegraphics[width=3.2in]{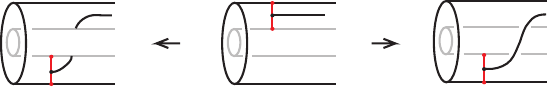}
\end{center}
\caption{A rail shift in $T^2 \times I$ induces endpoint swing moves in $T^2$.}
\label{endpm}
\end{figure}

We now state the following  theorem, whose proof is detailed in the next Subsection~\ref{sec:toroidal-inclusions}:

\begin{theorem}\label{isopTT}
Two multi-knotoids in the thickened torus $T^2 \times I$ are rail isotopic if and only if any two corresponding toroidal multi-knotoid diagrams of theirs, projected to the torus $T^2 \times \{0\}$, are Reidemeister equivalent.
\end{theorem}

\begin{proof}
For the proof we work in the PL category and we adapt the arguments of \cite{GK} for spatial knotoids to the setting of multi-knotoids in the thickened torus $T^2 \times I$. An isotopy move applied to a multi-knotoid in $T^2 \times I$ away from the rails, is local and can thus be confined within a 3-ball contained in the interior of $T^2 \times I$. By a general position argument, there is a  projection of the 3-ball to the outer torus $T^2 \times \{0\}$  which is regular with respect to  this move, so that it translates into a diagrammatic move of  Reidemeister equivalence away from the endpoints, as in Theorem 2.1 of  \cite{GK}. 
Finally, a sliding move of an endpoint along its rail can be sufficiently localized, so as to project perpendicularly on the outer torus, resulting in the move to be invisible on the diagrammatic level.  Hence the proof is concluded.
\end{proof}

\subsection{Inclusions}\label{sec:toroidal-inclusions}

The inclusion of a disc or a square in the torus enables us to view planar multi-knotoid diagrams as toroidal ones. View left hand side of Figure~\ref{planarannulartorus} and  left hand illustration of Figure~\ref{lift}. Further, the equivalence moves carry through,  so we have a well-defined map $\iota: \mathcal{K}(D^2) \longrightarrow \mathcal{K}(T^2)$, which is an injection, since the equivalence moves are the same in both settings. In terms of lifts, the inclusion of the thickened disc $D^2 \times I$  or the thickened square $I \times I \times I$ (which are homeomorphic to a three-ball)  in the thickened torus allows one to view spatial multi-knotoids as multi-knotoids in $T^2 \times I$, (recall Remark~\ref{thickdisc}). Namely, and combining with Theorem~\ref{isopTT}, we obtain  an injection on the level of isotopy classes of spatial multi-knotoids into the set of multi-knotoids in the thickened torus (view also Figure~\ref{lift}). Namely:

\begin{proposition}\label{disctorus}
The maps $\iota: \mathcal{K}(D^2) \longrightarrow \mathcal{K}(T^2)$ and $\iota: \mathcal{K}(D^2 \times I) \longrightarrow \mathcal{K}(T^2 \times I)$ are injections. 
\end{proposition}

\begin{figure} 
\begin{center} 
\includegraphics[width=6in]{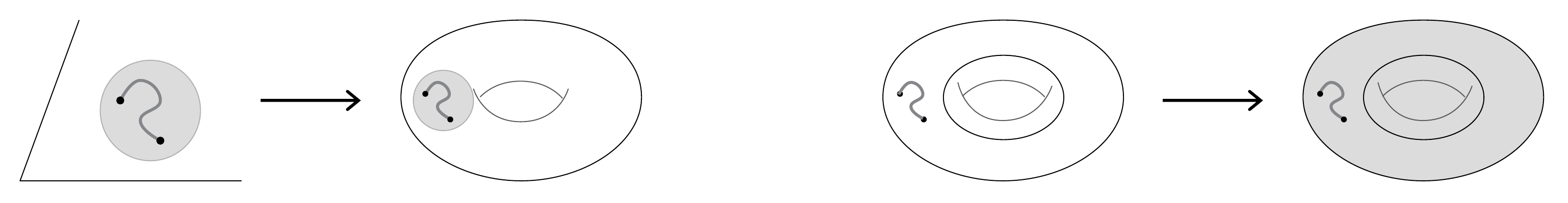} 
\end{center} 
\caption{Inclusion relations: a disc in the torus and the thickened torus in an enclosing 3-ball.} 
\label{planarannulartorus} 
\end{figure} 

\begin{remark}\label{rem:inclusionpq}
In Remark~\ref{rem:essentialpq} we argued that an essential $(p,q)$-torus closed curve is distinct from an essential $(q,p)$-torus closed curve in the theory of toroidal multi-knotoids. Interestingly, however,  a $(p,q)$-torus knot and a $(q,p)$-torus knot are elements of the same isotopy class as non-essential curves in $T^2 \times I$, by the injection of spatial multi-knotoids to multi-knotoids in $T^2 \times I$ (Proposition~\ref{disctorus}).  Hence, they are Reidemeister equivalent in $T^2$. 
\end{remark}

\begin{remarks} \rm \label{torus_no_surjection} 
\,\\ 
     (i) \ On the other hand, the inclusion of the thickened torus in an enclosing three-ball (see right hand side of Figure~\ref{planarannulartorus}), allows us to view any multi-knotoid in the thickened torus as a spatial (or even spherical) multi-knotoid,  by transforming at the same time the two piercing rail segments into two piercing rails of the three-ball and subsequently of the three-space. However, this construction does not induce a well-defined map of the theory of toroidal multi-knotoids onto planar multi-knotoids, since isotopic multi-knotoids in the thickened torus (resp. equivalent toroidal multi-knotoid diagrams) may not give rise to isotopic spatial  or  spherical multi-knotoids  (resp. equivalent planar or  spherical multi-knotoid diagrams). The reason is that, unlike the annular case, when including the thickened torus in a 3-ball, the two rail segments do not extend initially to the full interior of the 3-ball, so we have no control on isotopies taking place below the rails before extending them to rail lines, and this on a projection to a disc may perform unwished forbidden moves. 

    \noindent (ii) \ By the same arguments, we cannot prove Theorem~\ref{isopTT} by interpreting toroidal multi-knotoids as spatial ones, as we could do in the case of  Theorem~\ref{isopST} for the annular setting.
\end{remarks}

\begin{figure}[H] 
\begin{center} 
\includegraphics[width=6in]{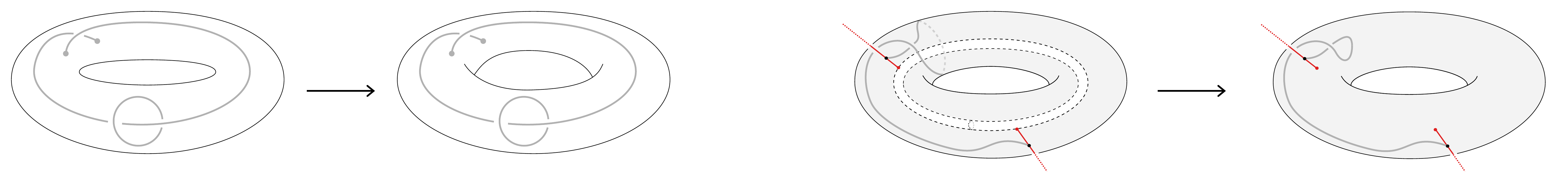} 
\end{center} 
\caption{Inclusion relations: an annulus in the torus; and the thickened torus in the solid torus.} 
\label{annulartorus} 
\end{figure} 

We shall now discuss the inclusion of the annulus in the torus. Recall from Subsection~\ref{sec:sphere-plane} that the 2-sphere is the gluing along the circular boundaries of two discs or the one-point identification of the boundary of a disc or the one-point compactification of the plane. Equivalently, the plane is homeomorphic to the sphere with one point (the `point at infinity') removed. Likewise, the torus $T^2$ can be obtained as the gluing along the inner and outer circular boundaries of two annuli or as the identification of the two boundary circles of an annulus. See left hand side of Figure~\ref{annulartorus}. Equivalently, the annulus $\mathcal{A}$ is homeomorphic to the torus with one longitudinal circle removed, inducing an inclusion map $\iota: \mathcal{A} \hookrightarrow T^2$. This inclusion allows one to view any annular multi-knotoid diagram  as a  toroidal multi-knotoid diagram, excluding the longitudinal toroidal move, recall Definition~\ref{def:toroidal-moves}, 
since such a move would require to cross the longitudinal circle. Note that in the absence of knotoid, a move analogous to the longitudinal toroidal move can be realized for knots and links in $\mathcal{A}$ with no obstruction, using the Reidemeister moves. 

To pass now to equivalence classes of multi-knotoid diagrams, any equivalence move of annular multi-knotoid diagrams is an equivalence move of toroidal multi-knotoid diagrams. Hence, the inclusion map $\iota$ induces a well-defined map $\mu: \mathcal{K}(\mathcal{A}) \longrightarrow \mathcal{K}(T^2)$. However, $\mu$ is not injective due to the longitudinal toroidal move. Note that, if the longitudinal toroidal moves are excluded from the equivalence in $T^2$ we have an injection $\mu^\prime: \mathcal{K}(\mathcal{A}) \longrightarrow \mathcal{K}^\prime(T^2)$ of annular multi-knotoids to toroidal multi-knotoids, where $\mathcal{K}^\prime(T^2)$ stands for the set of restricted equivalence classes in $T^2$. To summarize, annular multi-knotoids can be viewed as toroidal multi-knotoids for which the longitudinal toroidal move is not permitted. 

Using further the description of the thickened torus as the gluing of two thickened annuli and the inclusion of a thickened annulus in the thickened torus, we also have an inclusion of the theory of multi-knotoids in the thickened annulus into the theory of multi-knotoids in the thickened torus. 

\begin{proposition}\label{prop:no-injective}
    There is a well-defined map $\mu: \mathcal{K}(\mathcal{A}) \longrightarrow \mathcal{K}(T^2)$ from the set of annular multi-knotoids $\mathcal{K}(\mathcal{A})$ to the set of toroidal muti-knotoids $\mathcal{K}(T^2)$, which is not injective. Similarly, there is a well-defined map $\mu: \mathcal{K}(\mathcal{A} \times I) \longrightarrow \mathcal{K}(T^2 \times I)$ from the set of multi-knotoids in the thickened annulus $\mathcal{K}(\mathcal{A} \times I)$ to the set of  multi-knotoids in the thickened torus $\mathcal{K}(T^2 \times I)$, which is not injective.
\end{proposition}

Figure~\ref{lift} illustrates the above inclusion relations in terms of identification spaces and their flat lifts for a better visual perception. 

\begin{remark}\label{rem:incl-TT-ST}
The inclusion of the thickened torus in the solid torus (we refer the reader to the right-hand side of Figure~\ref{annulartorus}), thus the thickened annulus, allows us to view knotoids in the thickened torus as knotoids in the thickened annulus, where the meridional windings trivialize, by extending at the same time the two piercing rail segments in the width of the thickened annulus. However, this inclusion does not induce a well-defined map, namely a surjection, of the theory of toroidal multi-knotoids onto annular multi-knotoids, since isotopic multi-knotoids in the thickened torus (resp. equivalent toroidal multi-knotoid diagrams) may not give rise to isotopic multi-knotoids in the thickened annulus (resp. equivalent annular multi-knotoid diagrams). The reason is that, when including the thickened torus in the thickened annulus, the two rail segments do not extend initially to the full interior of the thickened annulus (see right-hand side of Figure~\ref{annulartorus}), so we have no control on isotopies taking place below the rail segments, before extending them, and this on a projection to the annulus may perform unwished forbidden moves. 
\end{remark}

\subsection{Representation via mixed planar knotoids}\label{sec:H-mixed}

In this extension and in analogy to the annular case, we now represent (oriented) toroidal multi-knotoids as mixed multi-knotoids in the plane, and (oriented) multi-knotoids in the thickened torus as spatial mixed multi-knotoids. In particular, viewing $T^2 \times I$ as the complement of a Hopf link in $S^3$, an (oriented) link in $T^2 \times I$ is represented uniquely by a mixed link in $S^3$ that contains a pointwise fixed Hopf link, ${\rm H}$, as a fixed sublink and the moving part representing the multi-knotoid in the thickened torus $T^2 \times I$. This approach  was also used in \cite{DLM-pseudo} for representing toroidal pseudo knots as planar pseudo knots. See Figure~\ref{ex3} for an example. 

The fixed Hopf link ${\rm H}$ can be assumed to be almost flat, in the sense that arcs are coplanar and the two crossings are embedded in local 3-balls. So ${\rm H}$ defines naturally a projection plane. Using now the spatial lift of a multi-knotoid (recall Subsection~\ref{sec:planar-lift}), we define:

\begin{definition}\label{def:H-mixed}
 A {\it spatial ${\rm H}$-mixed multi-knotoid} is a spatial multi-knotoid ${\rm H} \cup K$  which consists of the Hopf link ${\rm H = 1_{\rm H} \cup 2_{\rm H}}$ as a pointwise fixed part with labeled components $1_{\rm H}$ and $2_{\rm H}$, and the spatial multi-knotoid $K$ (including its two rails) resulting by removing ${\rm H}$, as the \textit{moving part} or {\it $K$-part} of the mixed multi-knotoid. Specifying orientations on the closed components of $K$ results in an \textit{oriented spatial ${\rm H}$-mixed multi-knotoid}. 
 
 \noindent Furher, two spatial ${\rm H}$-mixed multi-knotoids are {\it rail isotopic} if they are related via a finite sequence of isotopies in the complement of the two rails, {\it keeping the ${\rm H}$-part pointwise fixed}, together with sliding moves of the endpoints along the rails %(see Figure~\ref{railsliding}) 
 and rail shifts. The rail isotopy class of a spatial ${\rm H}$-mixed multi-knotoid shall be also referred to as {\it spatial ${\rm H}$-mixed multi-knotoid}. 
\end{definition}

\begin{figure}[H]
\begin{center}
\includegraphics[width=1.7in]{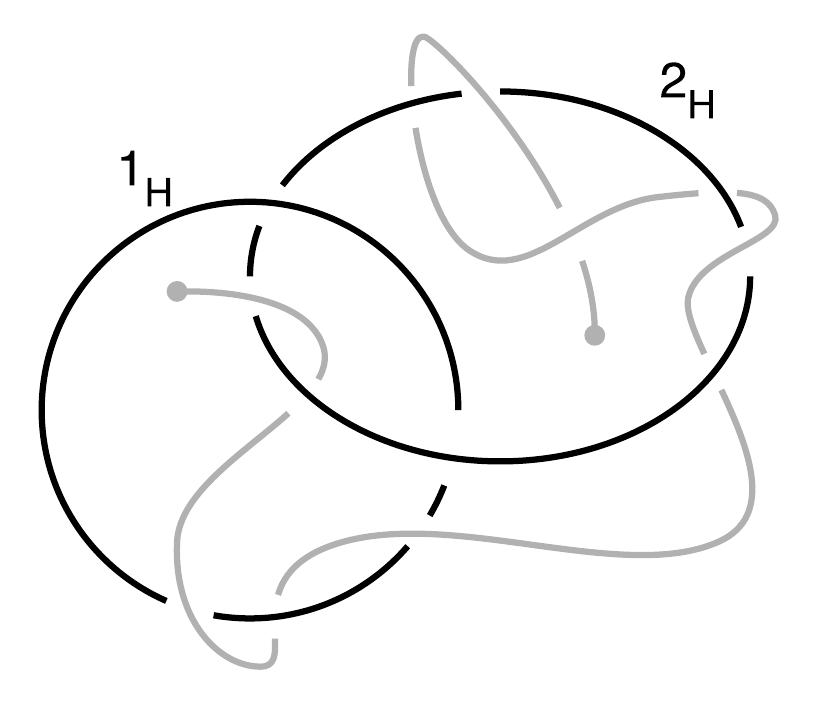}
\end{center}
\caption{A mixed knotoid representing a knotoid in $T^2\times I$.}
\label{mixed-H}
\end{figure}

\begin{remarks} \rm 
\,\\ 
    (i) As in the case of spatial ${\rm O}$-mixed multi-knotoids, one can conceptualize a spatial ${\rm H}$-mixed multi-knotoid ${\rm H} \cup K$ as a {\it trichromatic multi-knotoid}, where the two components $1_{\rm H}$ and $2_{\rm H}$ of the ${\rm H}$-part, and the $K$-part, except for the rails, are distinguished by color. For an illustration, see Figure~\ref{tricolore-mixed}.
    
    \noindent (ii) In the ${\rm H}$-mixed setting, a $(p,q)$-essential curve and a $(q,p)$-essential curve are clearly distinguished by the distinction of the two components $1_{\rm H}$ and $2_{\rm H}$ of ${\rm H}$, as argued in Remark~\ref{rem:essentialpq}. For an example, see Figure~\ref{tricolore-mixed}.

\begin{figure}[H]
\begin{center}
\includegraphics[width=1.7in]{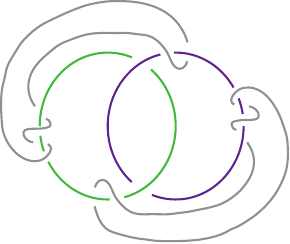}
\end{center}
\caption{A $(2,1)$-essential curve is distinguished from a $(1,2)$-essential curve in the ${\rm H}$-mixed setting.}
\label{tricolore-mixed}
\end{figure}
\end{remarks}

By analogous arguments as for the theory of spatial ${\rm O}$-mixed multi-knotoids and annular knotoids (see Theorem~\ref{isopmixed}), one can prove that a multi-knotoid $K$ in the thickened torus is represented unambiguously by a spatial ${\rm H}$-mixed multi-knotoid ${\rm H} \cup K$. Namely:

\begin{theorem}\label{H-isopmixed}
Rail isotopy classes of multi-knotoids in the thickened torus are in bijective correspondence with rail isotopy classes of spatial ${\rm H}$-mixed multi-knotoids.
\end{theorem}

For mixed knotoids representing knotoids in the thickened torus, we first establish an analogous set of Reidemeister moves as in the case of knotoids in the thickened annulus. These moves account for the interactions between the fixed part, the Hopf link, and the moving part, i.e. the multi-knotoid. Note that, in contrast to the theory of knotoids in the thickened annulus, the fixed part of the mixed knotoids now involves a crossing, which a moving strand can freely cross, resembling a mixed Reidemeister 3 move (for an illustration see Figure~\ref{mr3}).

\begin{figure}[H]
\begin{center}
\includegraphics[width=2.1in]{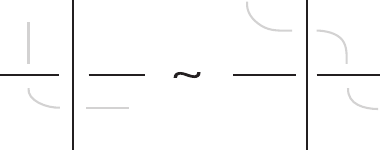}
\end{center}
\caption{An  ${\rm H}$-mixed R3 move.}
\label{mr3}
\end{figure}

\begin{definition}\label{def:H-mixed-diagram}
An ${\rm H}$-\textit{mixed multi-knotoid diagram} is a regular projection ${\rm H}\cup D$ of a spatial ${\rm H}$-mixed multi-knotoid ${\rm H}\cup K$ on the plane of ${\rm H}$, which is considered perpendicular to the rails. The double points are crossings, which are either crossings of arcs of the moving part or \textit{mixed crossings} between arcs of the moving and the fixed part, or the fixed crossing between arcs of the fixed part, and are endowed with over/under information.  
\end{definition}

To establish a diagrammatic equivalence, we note that the endpoints of a toroidal multi-knotoid, which is represented as an ${\rm H}$-mixed multi-knotoid, are permitted to pass over or under the arcs of the ${\rm H}$-part, contrasting with the restrictions imposed in standard knotoid theory (recall the forbidden moves illustrated in Figure~\ref{forbidoid}). Figure~\ref{mixedfm} illustrates this interaction between the fixed and moving parts of an ${\rm H}$-mixed multi-knotoid diagram. 
We shall call these allowed moves {\it endpoint moves}. 

\begin{figure}[H]
\begin{center}
\includegraphics[width=5.7in]{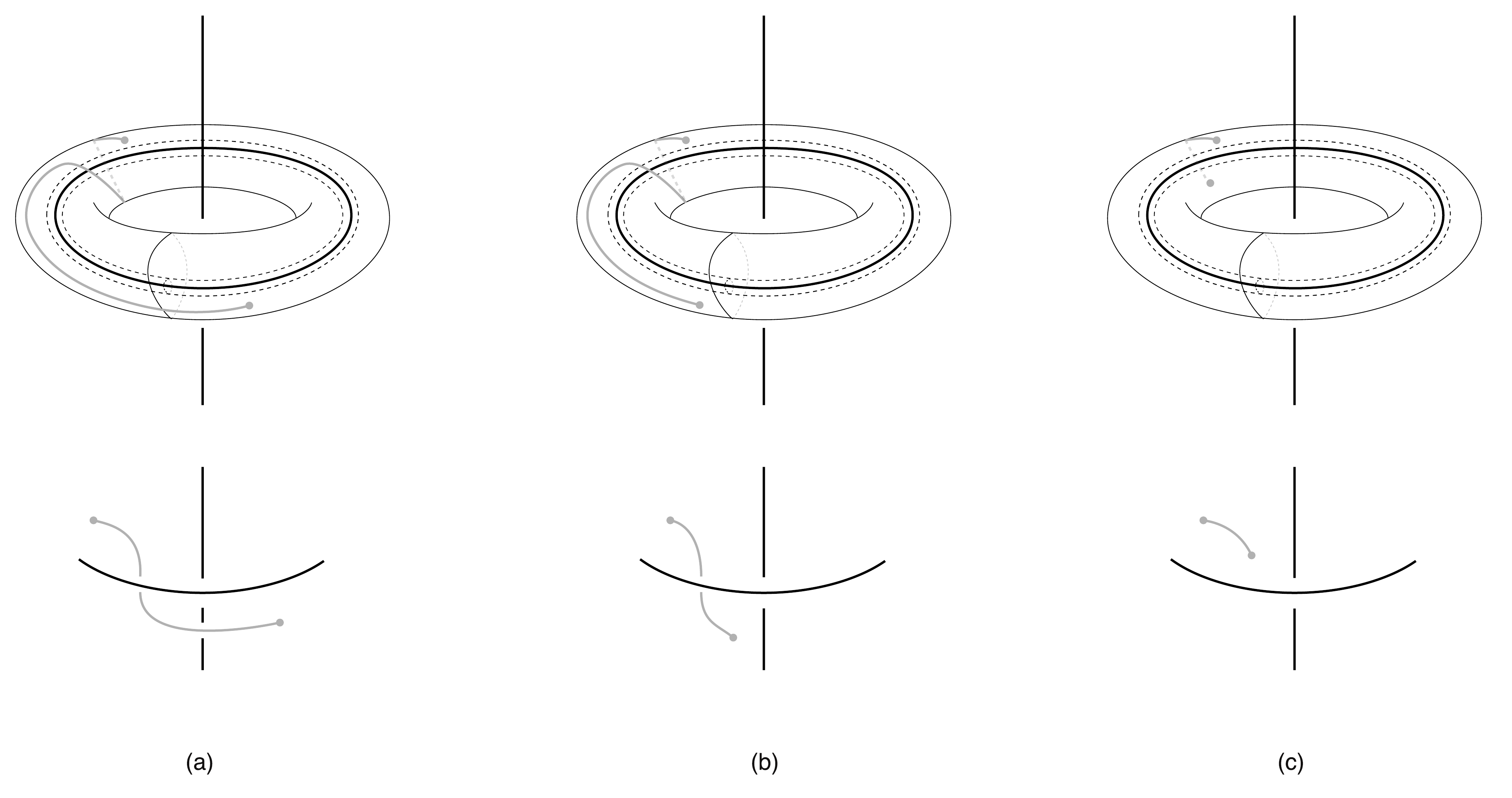}
\end{center}
\caption{The endpoint moves in the ${\rm H}$-mixed setting.}
\label{mixedfm}
\end{figure}

\begin{definition}\label{def:H-mixed-diagram-equiv}
Two (oriented) ${\rm H}$-mixed multi-knotoid diagrams are said to be {\it ${\rm H}$-mixed equivalent} if they differ by planar isotopies (which include the endpoint swing moves), a finite sequence of the classical Reidemeister moves (recall Figure~\ref{rmoves}) for their moving parts and that take place away from the endpoints of the mixed knotoid, the mixed Reidemeister moves  that involve the fixed and the moving parts (Figure~\ref{mreidm}), the mixed R3 move (Figure~\ref{mr3}) and the endpoint moves (Figure~\ref{mixedfm}) that allow the endpoints of the knotoid to pass over or under the fixed part of the mixed knotoid. An equivalence class of an ${\rm H}$-mixed multi-knotoid diagram shall be referred to as {\it ${\rm H}$-mixed multi-knotoid}. 
\end{definition}

Combining the equivalence of planar multi-knotoid diagrams (recall Section~\ref{sec:sphere-plane}) and the theory of mixed links (\cite{LR1, La2}) we obtain the discrete diagrammatic equivalence of spatial ${\rm H}$-mixed multi-knotoids, whose proof follows the same reasoning as for spatial ${\rm O}$-mixed multi-knotoids:

\begin{theorem} \label{Hmixedreid}
[The ${\rm H}$-mixed Reidemeister equivalence]\label{H-reidplink}
Two (oriented) spatial ${\rm H}$-mixed multi-knotoids are rail isotopic if and only if any two (oriented) ${\rm H}$-mixed multi-knotoid diagrams of theirs are ${\rm H}$-mixed equivalent.
\end{theorem}

\begin{remark}
Theorem~\ref{Hmixedreid} also applies to the case of linkoids. The definitions and moves extend naturally to linkoids, and the same Reidemeister equivalence holds for their spatial ${\rm H}$-mixed representations.
\end{remark}

In the above so far, we lifted toroidal knotoid diagrams in the thickened torus and we related knotoids in the thickened torus to spatial ${\rm H}$-mixed multi-knotoids and, subsequently, to planar ${\rm H}$-mixed multi-knotoids. Then, Theorems~\ref{isopTT},~\ref{H-isopmixed} and~\ref{Hmixedreid} combine into the following diagrammatic equivalence of the two different settings.

\begin{theorem}\label{toroidaltoHmixed}
Two toroidal multi-knotoid diagrams are Reidemeister equivalent if and only if any two corresponding planar ${\rm H}$-mixed multi-knotoid diagrams of theirs are ${\rm H}$-mixed Reidemeister equivalent.
\end{theorem}

\begin{remark}
One of the primary motivation behind the extension of the theory of knotoids to the thickened torus is the study of {\it DP-tangloids}, that is, the universal cover of a multi-linkoid embedded in the thickened torus (see \cite{DLM} and \cite{DLM3}). For an example, see Figure~\ref{DP-multilinkoid}. Our aim is to provide the theoretical framework for interdisciplinary applications of DP-tangloids in materials science, in molecular chemistry and in biology, to mention some.

\begin{figure}[H]
\begin{center}
\includegraphics[width=6in]{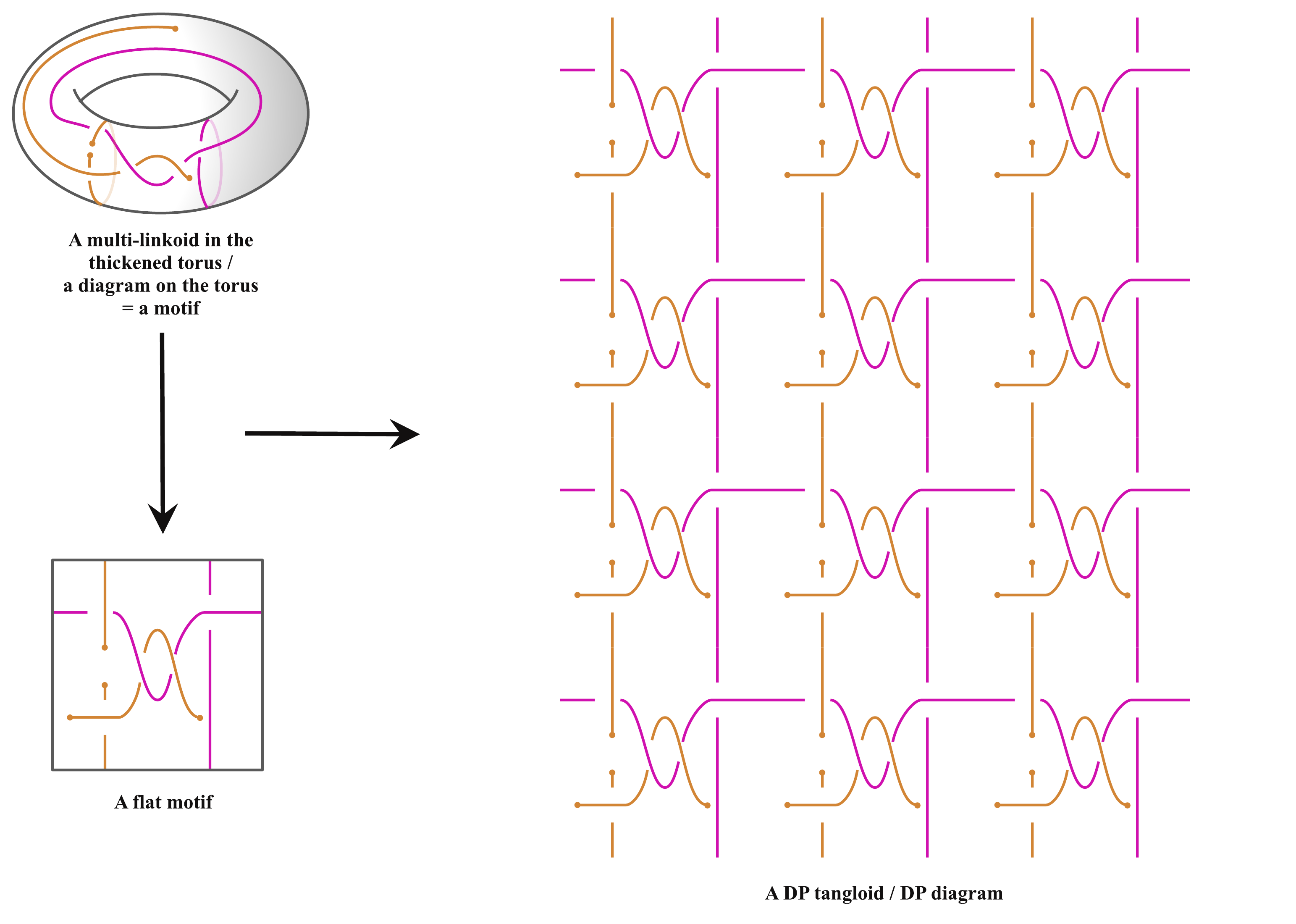}
\end{center}
\caption{DP tangloid.}
\label{DP-multilinkoid}
\end{figure}

\end{remark}

\section{The (universal) bracket polynomial for toroidal knotoids}\label{sec:tor-bracket} 

Recall that Turaev defines the analogue of the classical Kauffman bracket polynomial for multi-knotoids in any oriented surface, where any loop, essential, null-homotopic or nesting the trivial knotoid, gets the  same value $d$ \cite{T}. In this section we extend the bracket polynomials as well as the universal bracket polynomials for planar and annular multi-knotoids to toroidal  multi-knotoids by distinguishing these different classes of components, as respectively defined in Definitions~\ref{pkaufb} and~\ref{u-planarskein} and Definitions~\ref{pkaufbst} and~\ref{annularskein}.

\smallbreak
Consider a toroidal multi-knotoid diagram and apply to its crossings smoothings as presented in \cite{T} and as illustrated in Figure~\ref{smoothing}. After smoothing all crossings we arrive at a {\it state diagram} which may contain torus knots or torus links (recall Subsection~\ref{sec:toroidal-equivalence}), null-homotopic unknots and non-essential unknots enclosing the trivial knotoid. 

A generic {\it state} of a toroidal multi-knotoid diagram  may contain a number $l$ of $(p,q)$-torus knots, a number $m$ of non-essential unknots enclosing the trivial knotoid, a number $k$ of null-homotopic unknots, and the trivial knotoid (in the case that it is not enclosed by non-essential unknots, i.e. $m=0$). View an abstract example in Figure~\ref{fig:toroidal-turaev}. 

\begin{figure}[H]
\begin{center}
\includegraphics[width=5.2in]{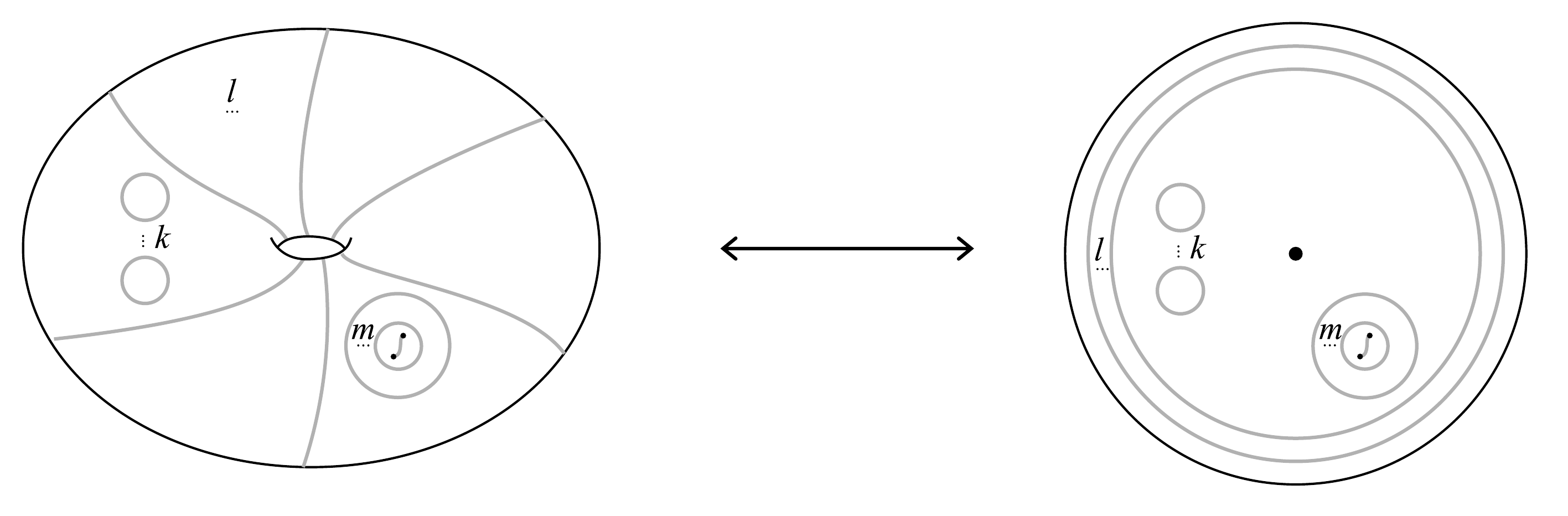}
\end{center}
\caption{Generic state of the toroidal bracket and its abstraction in a punctured disk for visualization purpose.}
\label{fig:toroidal-turaev}
\end{figure}

We first observe that all $k$ null-homotopic unknots can be enclosed in a disc not intersecting other arcs, placed anywhere in the torus, by the Reidemeister equivalence. We further observe that the relative positions of the $m$ non-essential unknots enclosing the trivial knotoid, for $m\geq 0$, are interchangeable by Reidemeister equivalence.

Let us now discuss the position of the trivial knotoid in a state diagram, enclosed or not by non-essential unknots. We denote this local disc formation as $K$. If $l=1$, then $K$ can lie anywhere in the complement of the torus knot in $T^2$. Otherwise, $K$ will be trapped in the essential ribbon defined by two adjacent components of the torus link, due to the forbidden moves. Observe, finally, that these two boundary components of the essential ribbon may be any pair of the $l$ components of the torus link since, by the Reidemeister equivalence, any two adjacent components not enclosing $K$ may interchange their positions.

\subsection{The  toroidal bracket polynomial}

We are now ready to introduce:
 
\begin{definition}\label{pkaufbtt}\rm
The {\it toroidal bracket polynomial} of a toroidal (multi)-knotoid diagram $K$, $\langle {\rm H}\cup K \rangle$, is defined by means of the following rules, inductively on the number of crossings in the diagram and closed components in its state:
\begin{itemize}
\item[i.] $\langle \ \raisebox{-2pt}{\includegraphics[scale=0.8]{4-1-1.pdf}} \ \rangle = A \langle \ \raisebox{0pt}{\includegraphics[scale=0.5]{skein-2.pdf}} \ \rangle + A^{-1} \langle \ \raisebox{-2.5pt}{\includegraphics[scale=0.6]{skein-3.pdf}} \ \rangle$
   \vspace{.1cm}  
\item[ii.]  $\langle \ \raisebox{-2pt}{\includegraphics[scale=0.85]{4-1-2.pdf}} \ \rangle = A^{-1} \langle \ \raisebox{0pt}{\includegraphics[scale=0.5]{skein-2.pdf}} \ \rangle + A \langle \ \raisebox{-2.5pt}{\includegraphics[scale=0.6]{skein-3.pdf}} \ \rangle $
  \vspace{.1cm}
    \item[iii.]$\langle L \sqcup \, \mathrm{O}^k \rangle  = d^k \, \langle L \rangle$, where $d = -A^2-A^{-2}$ and $\mathrm{O}^k$ stands for $k$ null-homotopic unknots.
   \vspace{.1cm}
     \item[iv.] $\langle  \raisebox{-9pt}{\includegraphics[scale=0.05]{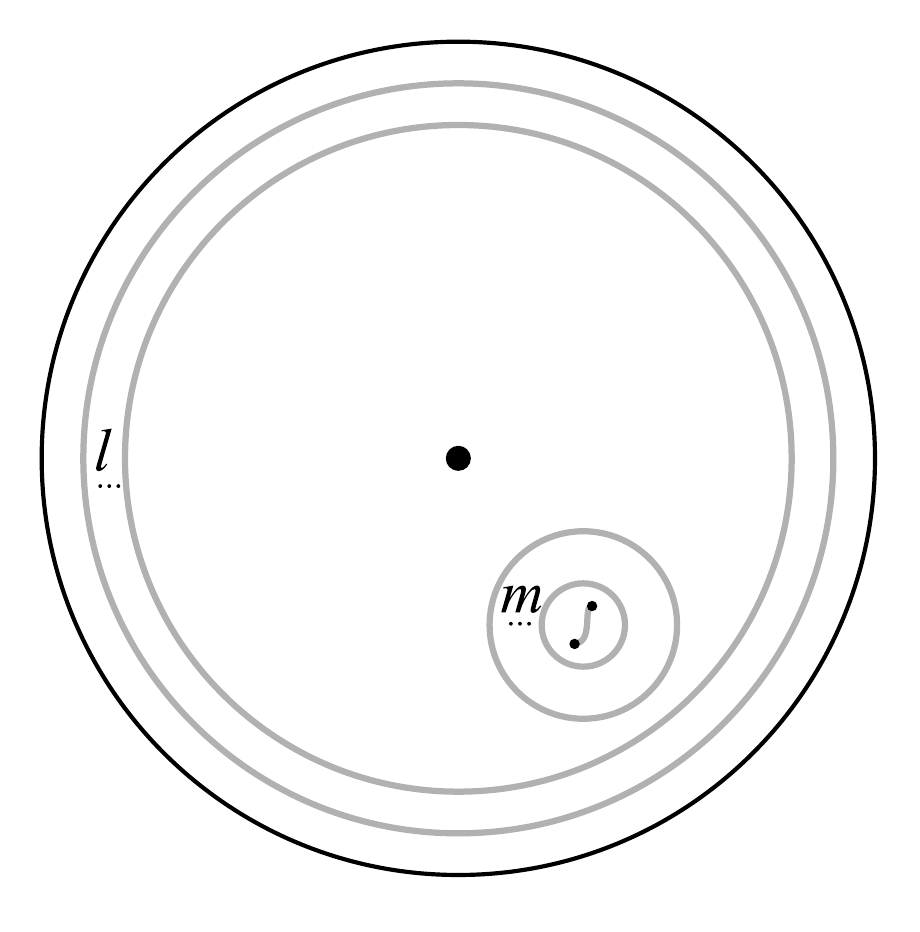}} \rangle  = s_{p,q}^l \ v^m  \langle \raisebox{-1pt}{\includegraphics[scale=0.1]{arc.pdf}} \rangle$ ,  where $l$ is the number of components of an $(lp,lq)$ torus link, $m$ is the number of non-essential unknots enclosing the trivial knotoid   (see Figure~\ref{fig:toroidal-turaev-zoom} for an enlarged picture).
  \vspace{.1cm} 
     \item[v.] $\langle \raisebox{-1pt}{\includegraphics[scale=0.1]{arc.pdf}} \rangle = 1$.
\end{itemize}
\end{definition}

\begin{figure}[H]
\begin{center}
\includegraphics[width=1.7in]{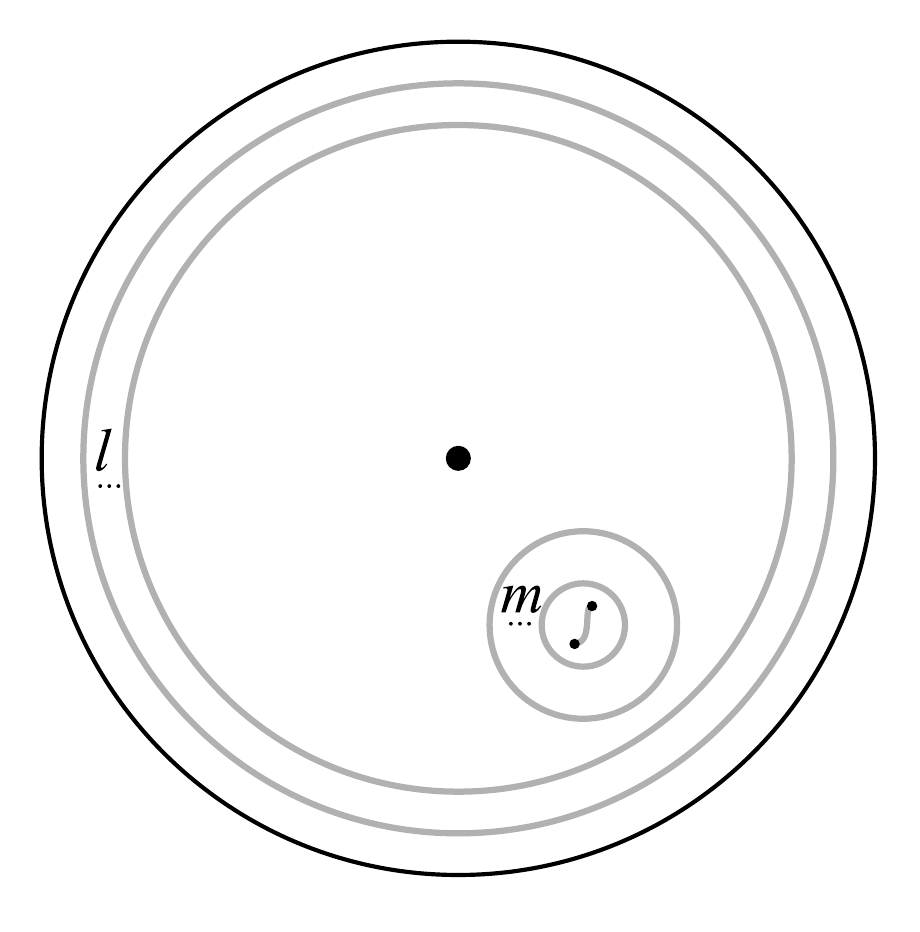}
\end{center}
\caption{A generic state for a toroidal multi-knotoid diagram after removing the null-homotopic components.}
\label{fig:toroidal-turaev-zoom}
\end{figure}

Setting $s_{p,q}=v=d$ \, the toroidal bracket polynomial recovers the Turaev bracket for surface $\Sigma = \mathcal{A}$~\cite{T}.

\begin{theorem}\label{th:pkaufbtt}
The toroidal bracket polynomial of an toroidal (multi)-knotoid diagram $L$ is a well defined infinite variable Laurent polynomial $ \langle L \rangle  \in \mathbb{Z}[A^{\pm 1}, v, s_{p,q}]$, for $p$ and $q$ coprime integers with $q$ non negative, or $(p, q) \in \{(1,0), (0, 1)\}$. Moreover, the toroidal bracket polynomial is a regular isotopy invariant, i.e. invariant under Reidemeister moves R2 and R3.
\end{theorem}

\begin{proof} 
For the well-definedness we argue as in the case of the universal planar bracket polynomial (Theorem~\ref{th:universal-bracket-regular-isotopy}):  crossing smoothings use only rules (i)-(ii) for reducing to states, so independence of the sequence of crossings follows immediately as for the Turaev bracket for surface $\Sigma = T^2$. Then we are left with a number of null-homotopic unknots, torus knots and the trivial knotoid with a number of unknots enclosing it, which can be all assumed unordered, since rules (iii) and (iv) do not `see' any ordering.  Now, the forbidden moves ensure that all unknotted components are locked in their positions up to ordering, so there is no ambiguity in applying rule (iii) for the null-homotopic ones. As in the annular case, note that a null-homotopic unknot in $T^2$ may lie in between the essential unknots.  Yet, its position is invisible when applying rule (iii). For this reason, in a generic state all $k$ null-homotopic unknots can be assumed to lie within a local disc in the same region of the diagram. We end up with generic state diagrams consisting of a number of non-essential unknots (torus knots) and the trivial knotoid, as in the right-hand side of  Figure~\ref{fig:toroidal-turaev}. In this reduced state diagram, we are now in a position to apply rule (iv). Note that the relative positions of the unknots are interchangeable, since this is not visible by rule (iv), which only counts their numbers. So the integers $k,l,m$ are fixed in every state. Therefore,  application of rules (iii) and (iv) is unambiguous. Finally rule (v) applies to the trivial knotoid. Hence, the the well-definedness of the toroidal bracket polynomial is concluded.
Finally, the invariance under regular isotopy of the toroidal bracket polynomial involves only rules (i)-(iii) and follows from the same arguments as in the planar and annular cases.
\end{proof}

\begin{remark}
    Unlike the annular case, there is no non-essential unknot that encloses the trivial knotoid in a state of a toroidal multi-knotoid diagrams since the notion of `inner' and `outer' essential components does not hold here. This can be shown by applying a finite sequence of toroidal moves and Reidemeister moves R2 to the components of a torus link in a state diagram, as illustrated in Figure~\ref{no-trap-knotoid-toroidal}.

\begin{figure}[H] 
\begin{center} 
\includegraphics[width=6in]{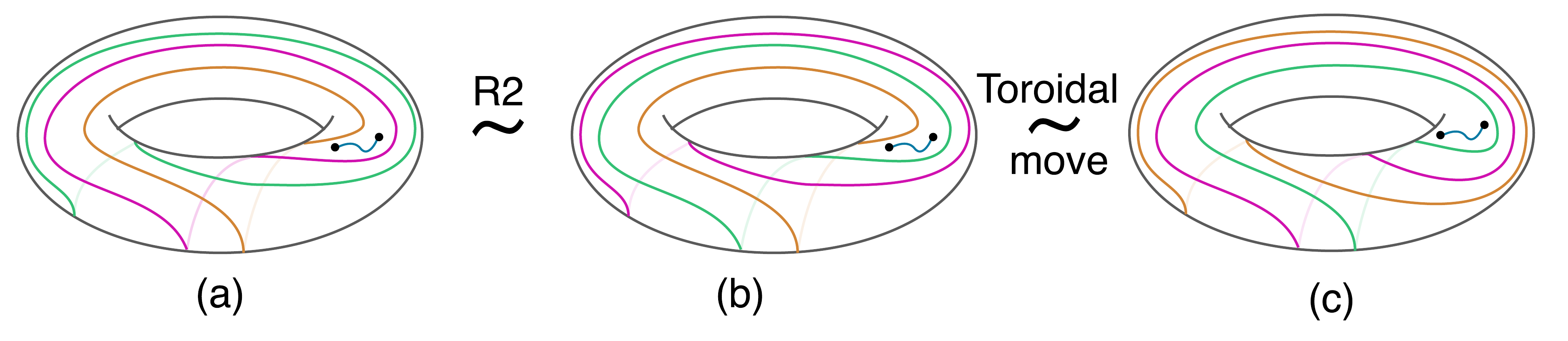} 
\end{center} 
\caption{The trivial knotoid cannot be trapped by essential components in a toroidal state diagram, by the toroidal Reidemeister equivalence.} 
\label{no-trap-knotoid-toroidal} 
\end{figure} 

\end{remark}

\subsection{The universal toroidal bracket polynomial}

In analogy to the planar case, rule (iv) can be generalized to an infinite variable expression:

\begin{definition}\label{toroidalskein}
The {\it universal toroidal bracket polynomial} for toroidal multi-knotoids, denoted $ \langle \cdot \rangle_U $, is defined by means of rules (i), (ii), (iii), (v) of Definition~\ref{pkaufbtt} and rule (iv$^{\prime}$) below:
\begin{center}
iv$^\prime$. \qquad $\langle  \raisebox{-9pt}{\includegraphics[scale=0.05]{toroidal-turaev-zoom.pdf}} \rangle_U = v_m s_{p,q,l} \langle \raisebox{-1pt}{\includegraphics[scale=0.1]{arc.pdf}} \rangle_U$
\end{center}
 where we alternatively  substitute $v^m$ by $v_m$ and $s_{p,q}^l$ by $s_{p,q,l}$, with $m, l$ being non-negative integers. 
\end{definition}

Definition~\ref{toroidalskein} leads to the following result.

\begin{theorem}\label{th:univ-tor}
The universal toroidal bracket polynomial is an infinite variable Laurent polynomial $ \langle \cdot \rangle_U  \in \mathbb{Z}\left[A^{\pm 1}, v_m, s_{p,q,l},\, \ m, l \in \mathbb{N}\cup \{0\}\right]$, that is a regular isotopy invariant for toroidal multi-knotoids, and  realizes the Kauffman bracket skein module of toroidal multi-knotoids.
\end{theorem}

\begin{proof}
As with the  universal annular bracket polynomial, the proof of the theorem is completely analogous to the proof of Theorem~\ref{th:pkaufbtt}, in combination with the proof of Theorem~\ref{th:universal-bracket-regular-isotopy} and taking  into account Definition~\ref{toroidalskein}.
\end{proof}

\begin{remark}
As a consequence of Theorem~\ref{th:univ-tor}, the Kauffman bracket skein module of toroidal multi-knotoids is freely generated by the state diagrams illustrated in the left-hand side of  Figure~\ref{fig:toroidal-turaev}.
\end{remark}

\subsection{The reduced toroidal bracket polynomial}

One can also define a simplified version of the toroidal bracket polynomial  that can be obtained by introducing two new variables $x, y$ in place of the infinitely many variables $s_{p,q}$  in the fourth rule. This reduced oroidal bracket polynomial can be useful in computations. More precisely: 

\begin{definition}\label{def:reduced-tor-bracket}
The \textit{reduced toroidal bracket polynomial} $\langle K \rangle_R$ of a toroidal multi-knotoid diagram $K$ is defined inductively by means of the same rules as for the toroidal bracket polynomial except for the fourth rule, in which $s_{p,q}$ is replaced by $x^p y^q$.
\end{definition}

\noindent  Note that for the toroidal, the universal toroidal and the reduced toroidal bracket polynomials, the last rule can only be applied to a toroidal state diagram that contains only a trivial knotoid. We also point out that two torus knot components of a toroidal multi-knotoid diagram, which have different slopes, will necessarily form crossings, which will be smoothed by the skein relations (i) and (ii) of Definitions~\ref{pkaufbtt}, ~\ref{toroidalskein} and ~\ref{def:reduced-tor-bracket}. So, a state diagram of the toroidal bracket polynomial can only contain a single torus knot or link. Therefore, together with the comment above Definition~\ref{pkaufbtt} we have that the toroidal bracket polynomial, the unviversal one  and the reduced one are well defined. In Section~\ref{computing_bracket} we provide some computations of the toroidal bracket polynomial on specific examples.

\begin{theorem}\label{th:pkaufbtt-reduced}
The reduced toroidal bracket polynomial of a toroidal (multi)-knotoid diagram $K$ is a well defined  4-variable Laurent polynomial $ \langle K \rangle_R  \in \mathbb{Z}\left[A^{\pm 1}, v, x^{\pm 1}, y \right]$, for $p$ and $q$ coprime integers with $q$ non negative, or $(p, q) \in \{(1,0), (0, 1)\}$. Moreover, the reduced toroidal bracket polynomial is a regular isotopy invariant, i.e. invariant under Reidemeister moves R2 and R3.
\end{theorem}

\begin{remark} 
If we have a toroidal multi-knotoid diagram that contains no essential closed curves, by using the inclusion of a disc in $T^2$ (see left-hand illustration in Figure~\ref{planarannulartorus}), such toroidal diagrams can be viewed as planar diagrams and resolve into states as in the planar case. In contrast, using the inclusion of $T^2$ in a 3-ball (recall right-hand illustration in Figure~\ref{planarannulartorus}), the $(p,q)$ torus knots become null-homotopic. So, substituting $s_{p,q}^l$ (resp. $s_{p,q,l}$) by $d^{l} = (-A^2-A^{-2})^{l}$, the toroidal (resp. universal toroidal) bracket polynomial specializes to the planar (resp. universal planar) bracket of Definition~\ref{pkaufb} (resp. Definition~\ref{u-planarskein}) for planar knotoids. Moreover, if the essential unknots of a toroidal multi-knotoid diagram are of type $(p,0)$-torus knots, namely longitudinal curves, then by using the inclusions relations between the annulus and the torus, we can substitute $s_{p,q}^n$ by $x^n$  and recover the annular bracket polynomial. The same reasoning applied to the reduced toroidal bracket polynomial for each case.
\end{remark}

\begin{remark}
Like for annular multi-knotoids the notion of a {\it knot type toroidal multi-knotoid}, whereby the endpoints lie in the same diagrammatic region, is not well-defined as explained by the following example: the trivial knotoid can be local and correspond to a null-homotopic unknot. But it can also be isotoped to a $(p,q)$-torus knot with a segment removed, as highlighted in Remark~\ref{trivial_pq}. So, the corresponding knot type in this case would be the $(p,q)$-torus knot. This implies that the under- or over-closure \cite{T} of an annular multi-knotoid is also not well defined. So, the evaluation of the toroidal, universal toroidal and reduced toroidal bracket polynomials of the corresponding under- or over-closure link  in the thickened torus is not well defined. 
\end{remark}

\subsection{State summation formuli for the toroidal bracket polynomials}

In complete analogy to the classical bracket and the spherical, (universal) planar and annular bracket polynomials, one can arrive at an alternative definition of the toroidal bracket, the universal toroidal bracket and the reduced bracket toroidal polynomials via a state summation: 

\begin{definition}
The toroidal bracket, the universal toroidal bracket and the reduced toroidal bracket polynomials of a toroidal multi-knotoid diagram $K$ are defined by means of state summations as: 
\begin{align*}\label{sskbst}
\langle K \rangle\ &= \underset{s\in S(K)}{\sum}\, A^{\sigma_s}\,  d^{k_s}\, v^{m_s} \, s_{p_s,q_s}^{l_s}  \\\\  \langle K \rangle_U\ &= \underset{s\in S(K)}{\sum}\, A^{\sigma_s}\, d^{k_s}\,  v_{m_s} \, s_{p_s,q_s,l_s} 
\\\\  \langle K \rangle_R\ &= \underset{s\in S(K)}{\sum}\, A^{\sigma_s}\, d^{k_s}\,  v^{m_s} \, (x^{p_s} \, y^{q_s})^{l_s} 
\end{align*} 
\noindent where $d= -A^2-A^{-2}$,  $\sigma_s$ is the number of A-smoothings minus the number of B-smoothings applied to the crossings of $K$ in order to obtain the state $s$ (recall Figure~\ref{smoothing}), $k_s$ is the number of null-homotopic unknots in $s$, $m_s$ is the number of non-esssential unknots in $s$ that enclose the trivial knotoid, and $l_s$ is the number of essential unknots (torus knots) in the state $s$, winding $p_s$ times along the torus longitude and $q_s$ times along the torus meridian.
\end{definition}

\subsection{The toroidal Jones polynomials}

As in the case of the planar, annular and universal annular bracket polynomials,  the toroidal, universal toroidal and reduced toroidal bracket polynomials can be normalized so as to also respect Reidemeister move 1, by considering the product of $\langle K \rangle$ , $\langle K \rangle_U$ and $\langle K \rangle_R$ by the writhe of a toroidal multi-knotoid diagram $K$, which is defined as follows.

\begin{definition} \label{def:toroidal_writhe}
The \textit{writhe} of an oriented toroidal multi-knotoid $K$, denoted as $w(K)$, is defined as the number of positive crossings minus the number of negative crossings of $K$ (recall Figure~\ref{si}). 
\end{definition}

\noindent This leads to obtaining the {\it normalized (universal/reduced) toroidal bracket (or Jones) polynomials}.

\begin{theorem}\label{th:jones-tor}
Let $K$ be a toroidal multi-knotoid diagram. The {\rm toroidal Jones}, the {\rm universal toroidal Jones} and the {\rm reduced toroidal Jones polynomials} defined as: 
\[
V(K)\ =\ (-A^3)^{-w(K)}\, \langle K \rangle\ \ \ {\rm ,}\ \ \ V_U(K)\ =\ (-A^3)^{-w(K)}\, \langle K \rangle_U\ \ \ {\rm and}\ \ \ V_R(K)\ =\ (-A^3)^{-w(K)}\, \langle K \rangle_R
\]
\noindent where $w(K)$ is the writhe of $K$,  $A = t^{-1/4}$, $\langle K \rangle$ the toroidal, $\langle K \rangle_U$ the universal and $\langle K \rangle_R$ the reduced toroidal bracket polynomials of $K$, are isotopy invariants of toroidal multi-knotoids.
\end{theorem}

\subsection{The ${\rm H}$-mixed bracket  polynomials}

In this subsection, we consider the planar ${\rm H}$-mixed multi-knotoids approach to toroidal multi-knotoids (recall Subsection~\ref{sec:H-mixed}). In view of Theorem~\ref{toroidaltoHmixed}, the toroidal bracket polynomial can be adapted to the setting of ${\rm H}$-mixed multi-knotoids. 

Here the crossings of ${\rm H}$  as well as the mixed crossings are not subjected to skein relations, since ${\rm H}$ represents the thickened torus, so must remain fixed throughout. We further consider the torus links of type $(lp,lq)$, which form the states for the toroidal  bracket polynomial of Definition~\ref{pkaufbtt}, and which we want to correspond to ${\rm H}$-mixed  links. Figure~\ref{Hmixedtorusknot} illustrates a $(3,2)$-torus knot represented in two ways as an ${\rm H}$-mixed link. The left-hand illustration of the figure uses the common representation of ${\rm H}$, while in the right-hand illustration ${\rm H}$ is represented in closed braid form. As we observe, in the left-hand illustration appear three moving crossings, which should be smoothened according to the bracket rules, a situation we can avoid using the representation of the $(3,2)$-torus knot in the right-hand illustration.

\begin{figure}[H]
    \centering
    \includegraphics[width=4.3in]{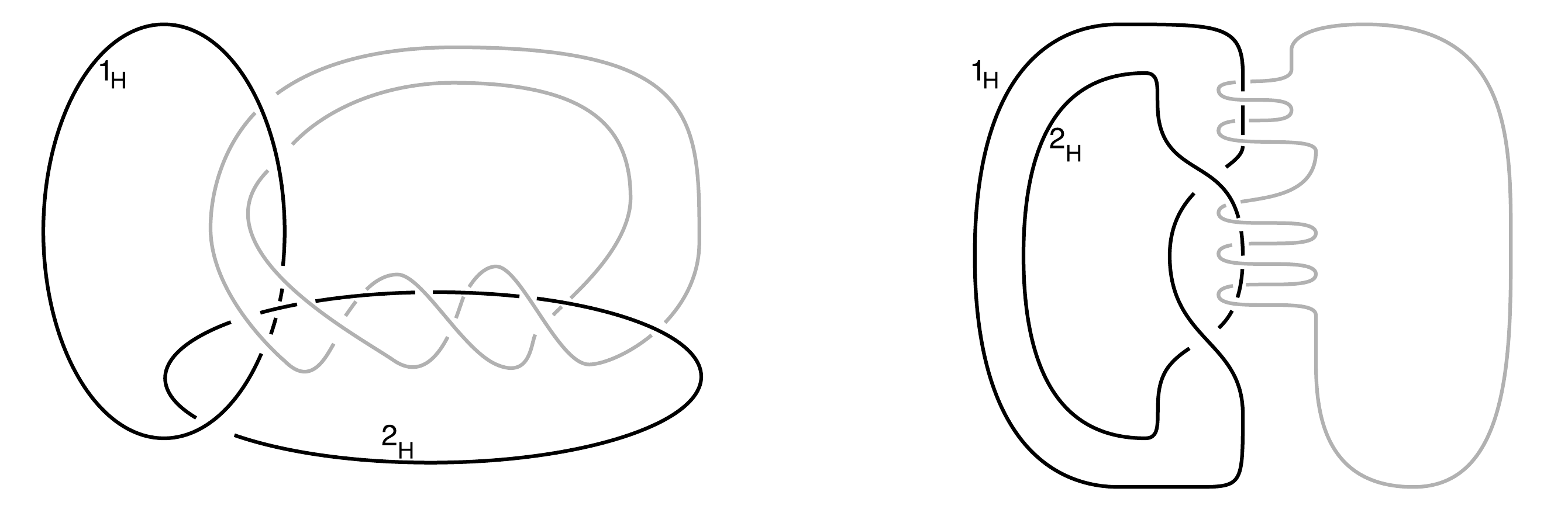}
    \caption{ A $(3,2)$-torus knot represented in two ways as an ${\rm H}$-mixed link.}
    \label{Hmixedtorusknot}
\end{figure}

We are now ready to translate the toroidal bracket polynomial in the ${\rm H}$-mixed setting:

\begin{definition}\label{mix-pkaufbtt}\rm
Let $K$ be a toroidal multi-knotoid diagram and let ${\rm H}\cup K$ the corresponding ${\rm H}$-mixed multi-knotoid diagram, with ${\rm H}$ being represented in closed braid form. The \textit{${\rm H}$-mixed bracket polynomial} $\langle {\rm H}\cup K \rangle$ of ${\rm H}\cup K$ is an infinite variable Laurent polynomial in the ring $\mathbb{Z}\left[A^{\pm 1}, v, s_{p,q} \right]$,  defined by means of the same inductive rules as the ones for the toroidal bracket polynomial, except for the  $(lp,lq)$-torus link diagram in the fourth rule of Definition~\ref{pkaufbtt}, which is substituted by the corresponding ${\rm H}$-mixed link diagram, as illustrated in Figure~\ref{Hmixedtoruslink}. Also, $L$ in the fourth rule stands now for an oriented ${\rm H}$-mixed multi-knotoid  $L \cup {\rm H}$.  We further define the \textit{universal ${\rm H}$-mixed bracket polynomial}  $\langle {\rm H}\cup K \rangle_U$ of ${\rm H}\cup K$  by substituting in the fourth rule $v^m s_{p,q}^l$ by $v_m s_{p,q,l}$. 
\end{definition}

\begin{figure}[H]
    \centering
   \includegraphics[width=2.2in]{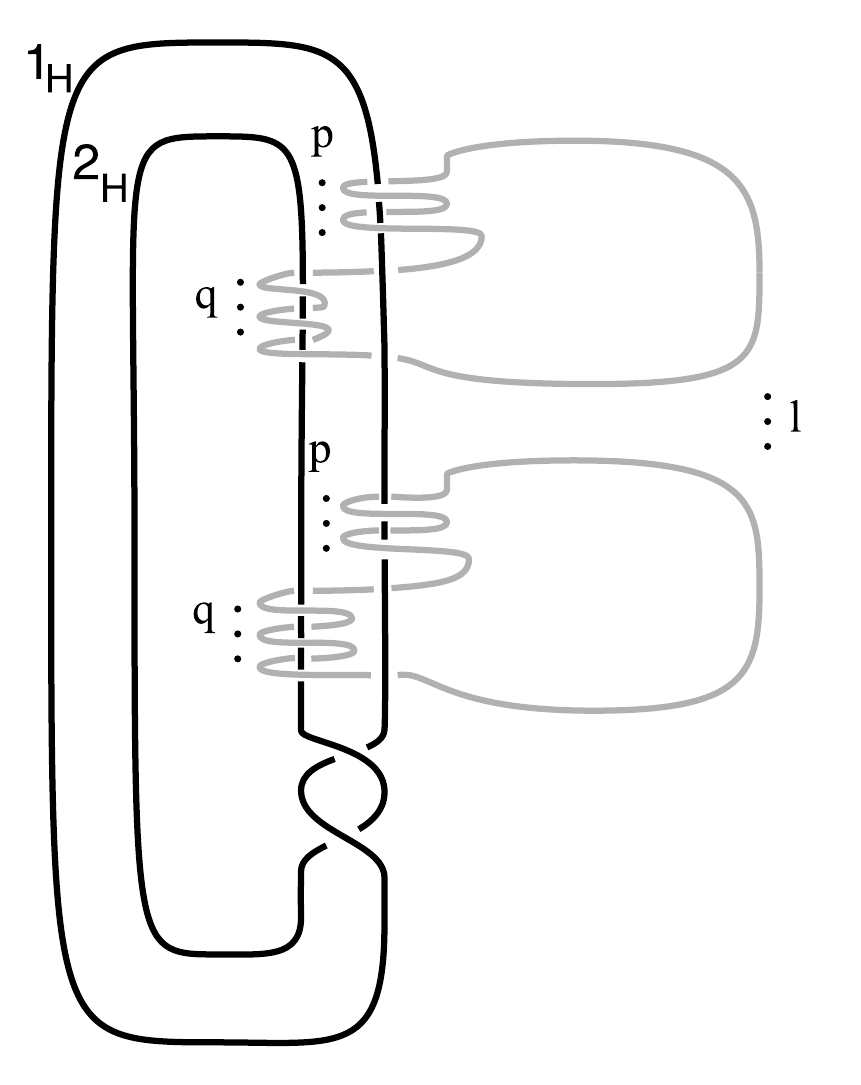}
    \caption{ An ${\rm H}$-mixed link diagram corresponding to the $(lp,lq)$-torus link.}
    \label{Hmixedtoruslink}
\end{figure}

We may also define a simplified version of the ${\rm H}$-mixed bracket polynomial that can  be obtained by introducing two new variables $x$ and $y$ in place of $s_{p,q}$, which can prove handy in computations. More precisely: 

\begin{definition}\label{def:reduced-H-bracket}
Let ${\rm H}\cup K$ be a ${\rm H}$-mixed multi-knotoid diagram, with ${\rm H}$ being represented in closed braid form. The \textit{reduced ${\rm H}$-mixed bracket polynomial}  $\langle {\rm H}\cup K \rangle_R$ of ${\rm H}\cup K$  is a 4-variable Laurent polynomial in the ring $\mathbb{Z}\left[A^{\pm 1}, v, x^{\pm 1}, y \right]$, defined inductively by means of the same rules as for the ${\rm H}$-mixed bracket polynomial, except for the fourth rule, in which $s_{p,q}$ is replaced by $x^p y^q$.
\end{definition}

\begin{remark} 
If we start from an ${\rm H}$-mixed multi-knotoid and do all smoothings of moving crossings, we would arrive at diagrams where different types of essential torus knots may be present, which do not cross in the mixed setting,  as illustrated in the left hand side of Figure~\ref{state-Hmixed}, and not necessarily copies of the same $(p,q)$ torus knot, as depicted in Figure~\ref{Hmixedtoruslink}.  Diagrams of this form  have been used in the braid approach to skein modules of various 3-manifolds.  For details cf. \cites{D2} and  references therein.   So, we would have states as illustrated in the right hand side of Figure~\ref{state-Hmixed} including also null-homotopic components. In other words,  diagrams as in Figure~\ref{Hmixedtoruslink} comprise only a subset of the complete theory of ${\rm H}$-mixed links. The above reasoning would lead eventually to an alternative  definition of the ${\rm H}$-mixed bracket polynomial with extended initial conditions. More precisely, we would define the \textit{alternative ${\rm H}$-mixed bracket polynomial}  $\langle {\rm H}\cup K \rangle$ of ${\rm H}\cup K$  by means of rules (i), (ii), (iii), (v) of Definition~\ref{pkaufbtt} and rule (iv$^{\prime}$) below:

\begin{center}
iv$^\prime$. \qquad $\langle  \raisebox{-9pt}{\includegraphics[scale=0.05]{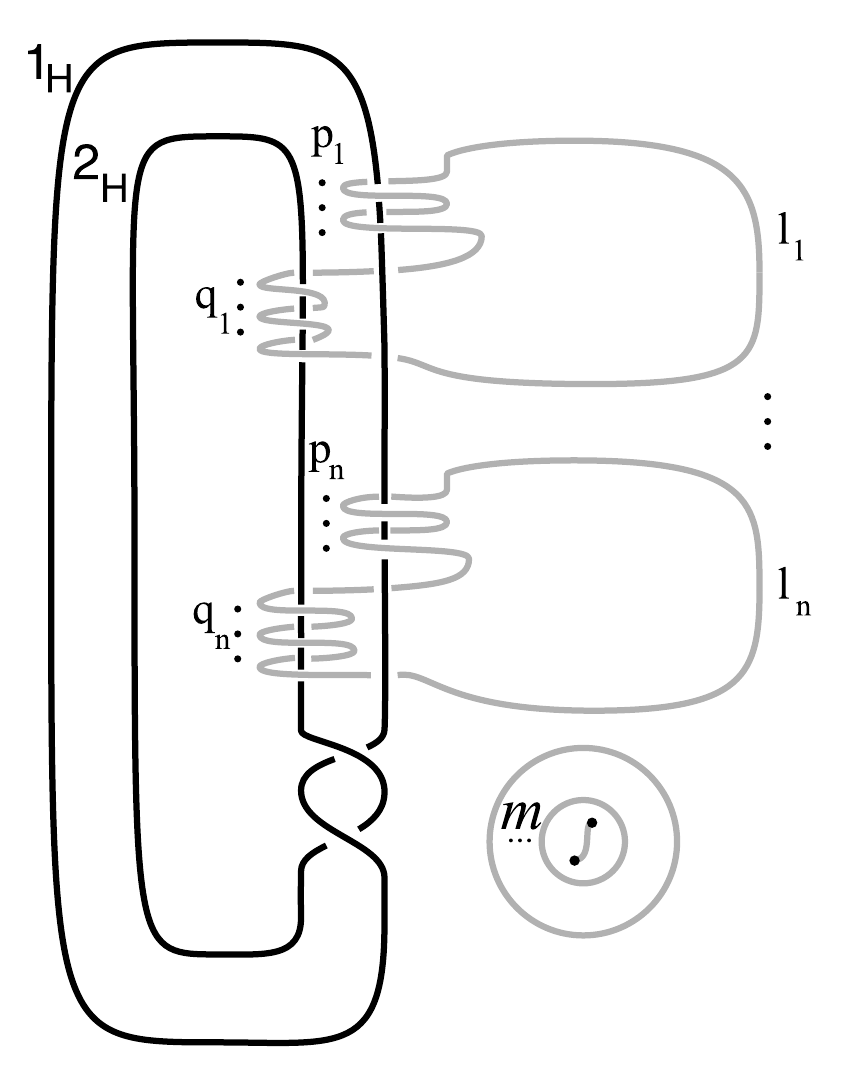}} \rangle = v^m \prod_i s_{p_i,q_i}^{l_i} \langle \raisebox{-1pt}{\includegraphics[scale=0.1]{arc.pdf}} \rangle$
\end{center}
\noindent where we substitute $s_{p,q}^l$ by the product $\prod_i s_{p_i,q_i}^{l_i}$ of the distinct $(p_i,q_i)$-torus links of $l_i$ components, with the $l_i$ being non-negative integers. See Figure~\ref{state-Hmixed} (right) for an enlarged picture of the diagram depicted in rule iv$^\prime$. Moreover, $L$ in rule (iii) stands now for an (oriented) ${\rm H}$-mixed multi-knotoid $ \mathrm{H} \cup \, L$. 
 This alternative definition is related to Definition~\ref{pkaufbtt} via a change of basis of the Kauffman bracket skein module of the thickened torus. 

 \begin{figure}[H]
    \centering
   \includegraphics[width=5.2in]{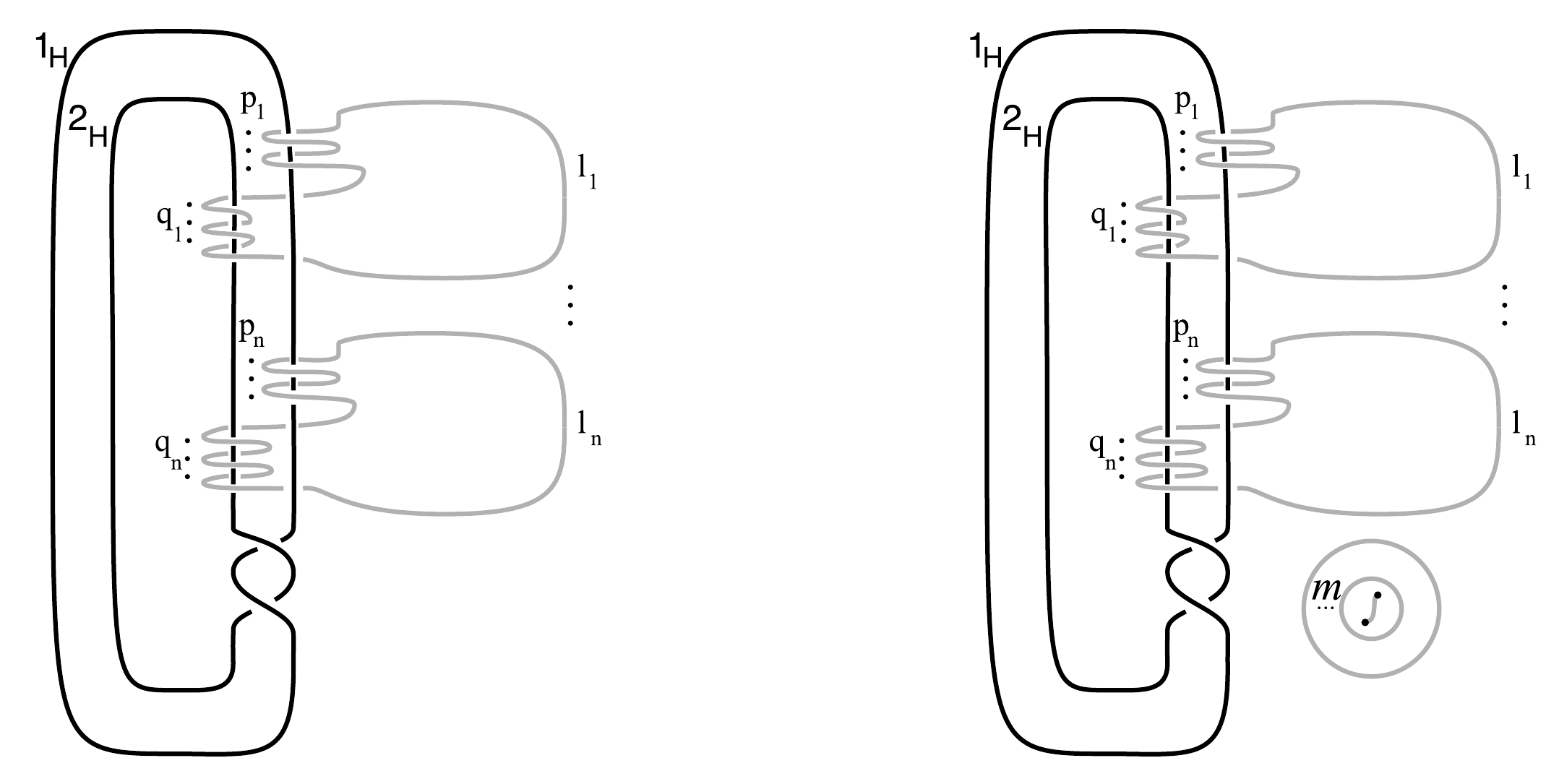}
    \caption{On the left, an ${\rm H}$-mixed link diagram corresponding to  different torus links. On the right, a generic state for an ${\rm H}$-mixed link diagram after removing the null-homotopic components.}
    \label{state-Hmixed}
\end{figure}

\end{remark}

In analogy now to the annular case (Theorem~\ref{th:O-pkaufbst}), we then have for planar ${\rm H}$-mixed multi-knotoids the following:

\begin{theorem}\label{th:th-mix-in}
The (universal) ${\rm H}$-mixed and the reduced ${\rm H}$-mixed bracket polynomials are invariant under regular isotopy of \, ${\rm H}$-mixed multi-knotoids and are equivalent to the  (universal) toroidal and the reduced toroidal bracket polynomials, respectively, for toroidal multi-knotoids.
\end{theorem}

\begin{proof}
We consider the restriction of the theory of planar multi-knotoids to the subset of  ${\rm H}$-mixed multi-knotoids. Then, by Theorems~\ref{th:pkaufbtt}, ~\ref{th:univ-tor}, and ~\ref{th:pkaufbtt-reduced}, the (universal/reduced) ${\rm H}$-mixed bracket polynomial is invariant under the classical Reidemeister moves R2 and R3. Regarding now the mixed Reidemeister moves MR2, MR3, and mixed R3 (recall Figures~\ref{mreidm} and~\ref{mr3}), since there is no rule for smoothing mixed crossings, the moves MR2 and the mixed R3 moves are undetectable by the (universal/reduced) ${\rm H}$-mixed bracket, so it remains invariant. Further, invariance under the moves MR3 and MPR3 follows by the same arguments as in the annular case. 

Finally, the equivalence of the (universal/reduced) ${\rm H}$-mixed bracket polynomial for ${\rm H}$-mixed multi-knotoids to the (universal/reduced) toroidal pseudo bracket polynomial for toroidal multi-knotoids follows immediately by Theorem~\ref{toroidaltoHmixed}.
\end{proof}

In analogy to the toroidal bracket polynomials we also have here closed state summation formuli.

\subsection{The ${\rm H}$-mixed Jones polynomials}

Finally, in analogy to the planar and the annular cases, the universal/reduced ${\rm H}$-mixed bracket polynomials can also be normalized so as to also respect Reidemeister move 1. 

\begin{definition}\label{def:toroidal_Hwrithe}
The ${\rm H}$-\textit{writhe} of an oriented ${\rm H}$-mixed multi-knotoid ${\rm H}\cup K$ shall be denoted $w({\rm H}\cup K)$ and is defined as the number of positive crossings minus the number of negative crossings of ${\rm H}\cup K$ (recall Figure~\ref{si}), while the mixed crossings do not contribute to the ${\rm H}$-writhe.
\end{definition}

This leads to obtaining the {\it normalized (universal/reduced) ${\rm H}$-mixed bracket (or Jones) polynomials}.

\begin{theorem}\label{th:Jones-H}
Let ${\rm H}\cup K$ be an ${\rm H}$-mixed multi-knotoid diagram. 
The {\it ${\rm H}$-mixed Jones}, the {\it universal ${\rm H}$-mixed Jones polynomials} and the {\it reduced ${\rm H}$-mixed Jones polynomials}
$$V({\rm H}\cup K)\ =\ (-A^3)^{-w({\rm H}\cup K)}\, \langle {\rm H}\cup K \rangle$$
$$V_U({\rm H}\cup K)\ =\ (-A^3)^{-w({\rm H}\cup K)}\, \langle {\rm H}\cup K \rangle_U$$
$$V_R({\rm H}\cup K)\ =\ (-A^3)^{-w({\rm H}\cup K)}\, \langle {\rm H}\cup K \rangle_R$$
\noindent where $w({\rm H}\cup K)$ is the ${\rm H}$-writhe of ${\rm H}\cup K$,  $A = t^{-1/4}$, $\langle {\rm H}\cup K \rangle$ the ${\rm H}$-mixed bracket polynomial of ${\rm H}\cup K$, $\langle {\rm H}\cup K \rangle_U$ the universal ${\rm H}$-mixed bracket polynomial of ${\rm H}\cup K$ and $\langle {\rm H}\cup K \rangle_R$ the reduced ${\rm H}$-mixed bracket polynomial of ${\rm H}\cup K$, are isotopy invariants of ${\rm H}$-mixed multi-knotoids.
\end{theorem}

\begin{proof}
The invariance of the (universal/reduced) ${\rm H}$-mixed Jones polynomial follows analogously from the invariance of the (universal/reduced) ${\rm H}$-mixed bracket polynomial (Theorem~\ref{th:th-mix-in}), the definition of the ${\rm H}$-writhe (Definition~\ref{def:toroidal_Hwrithe} and Theorem~\ref{th:pl-bracket} for planar multi-knotoids. 
\end{proof}

\section{Computing Examples}\label{computing_bracket}

In this final section we  compute some examples of the bracket polynomial for multi-knotoids in three different contexts: on the plane, in the thickened annulus and in the thickened torus. These examples emphasize the differences in the polynomial calculations across these settings. Specifically, we provide two examples of multi-knotoids. For each example, we consider the multi-knotoid in the above mentioned three scenarios, highlighting the distinct characteristics of the bracket polynomial in each context.

\subsection{Example 1}

To illustrate some limitations of the bracket polynomial for a multi-knotoid across different contexts, we start by considering the knotoids illustrated in Figure~\ref{kstt}.

\begin{figure}[H]
\begin{center}
\includegraphics[width=5.5in]{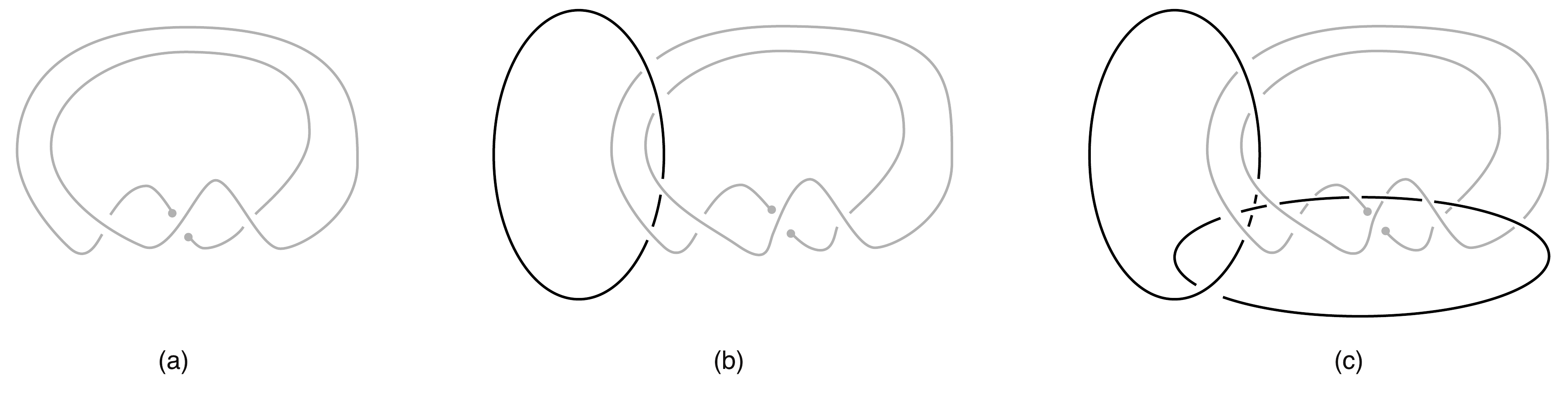}
\end{center}
\caption{(a) The knotoid $K_p$ in $\mathbb{R}^2 \times I$, (b) $K_a$ in ST and (c) $K_t$ in $T^2 \times I$.}
\label{kstt}
\end{figure}

In Figure~\ref{kstt}(a), we present a knotoid $K$ on the plane, denoted by $K_p$, and we compute the planar bracket polynomial according to Definition~\ref{pkaufb}. We have that:

\[
\langle K_p \rangle = A^2 + 1 - A^{-4}
\]

In Figure~\ref{kstt}(b), the knotoid $K$ is presented in the thickened annulus, denoted by $K_a$, and in Figure~\ref{kstt}(c), $K$ is considered in $T^2 \times I$, denoted by $K_t$.
Computing the annular and toroidal bracket polynomial of $K_a$ and of $K_t$ (according to the rules of Definition~\ref{pkaufbst} and Definition~\ref{pkaufbtt} respectively), we observe that:

\[
\langle K_a \rangle = \langle K_t \rangle = \langle K_p \rangle = A^2 + 1 - A^{-4}
\]

By comparing the bracket polynomials for knotoids $K_p, K_a$, and $K_t$, we observe that in this case, $\langle ; \rangle$ does not capture the topology of the space the knotoid lives in.

\subsection{Example 2}

We now present an example that highlights the differences in the bracket polynomial for a multi-knotoid across the different contexts. We consider the multi-knotoids illustrated in Figure~\ref{ex-bracket}.

\begin{figure}[H]
\begin{center}
\includegraphics[width=6in]{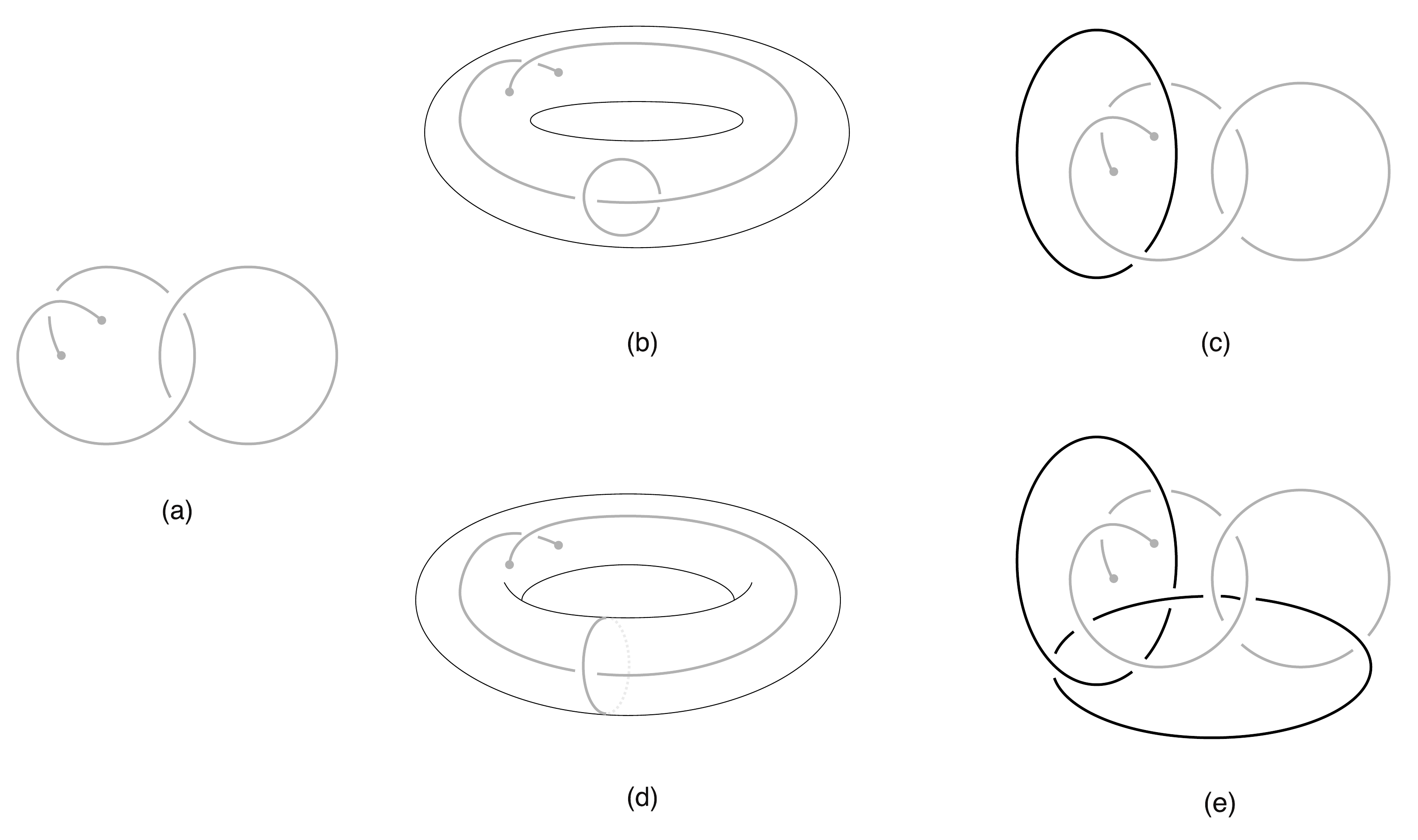}
\end{center}
\caption{(a) A planar multi-knotoid $L_p$, (b) $L_p$ on the surface of the torus, (c) $L_a$ as an annular multi-knotoid, and (d) $L_t$ as a toroidal multi-knotoid.}
\label{ex-bracket}
\end{figure}

We first compute the planar bracket polynomial for the planar multi-knotoid $L_p$ and we have that:
\[
\langle
L_p\rangle = \left(-A^{-3}\right) (1 + A^8) v + \left(- A^{-5}\right) (1 + A^8)
\]

In Figure~\ref{ex-bracket-planar} we illustrate the corresponding tree diagram of $L_p$ resulting from the smoothing of each crossing.

\begin{figure}[H]
\begin{center}
\includegraphics[width=5.7in]{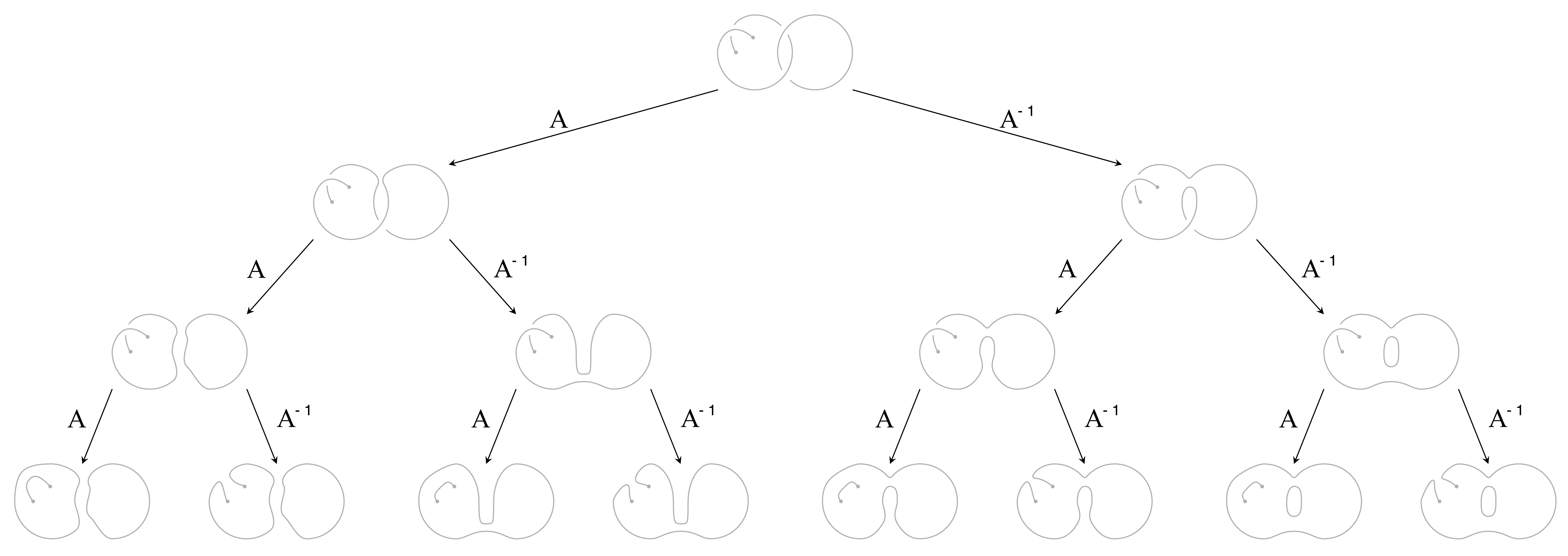}
\end{center}
\caption{The tree diagram for computing the planar bracket polynomial of $L_p$.}
\label{ex-bracket-planar}
\end{figure}

In Figure~\ref{ex-bracket-annular} and Figure~\ref{ex-bracket-torus}  we illustrate the tree diagram for the annular and the toroidal cases, respectively. 

\begin{figure}[H]
\begin{center}
\includegraphics[width=6in]{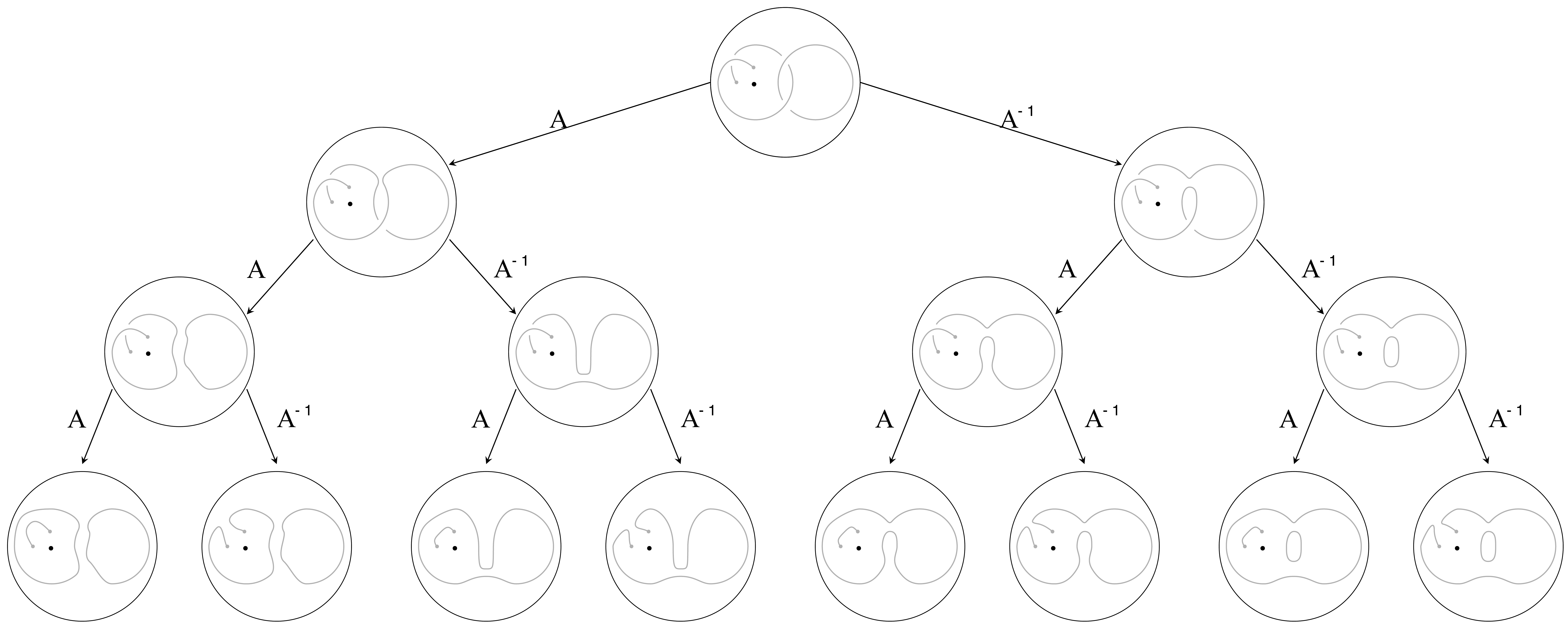}
\end{center}
\caption{The skein tree diagram for computing the annular bracket polynomial of the annular multi-knotoid $L_a$.}
\label{ex-bracket-annular}
\end{figure}

\begin{figure}[H]
\begin{center}
\includegraphics[width=6in]{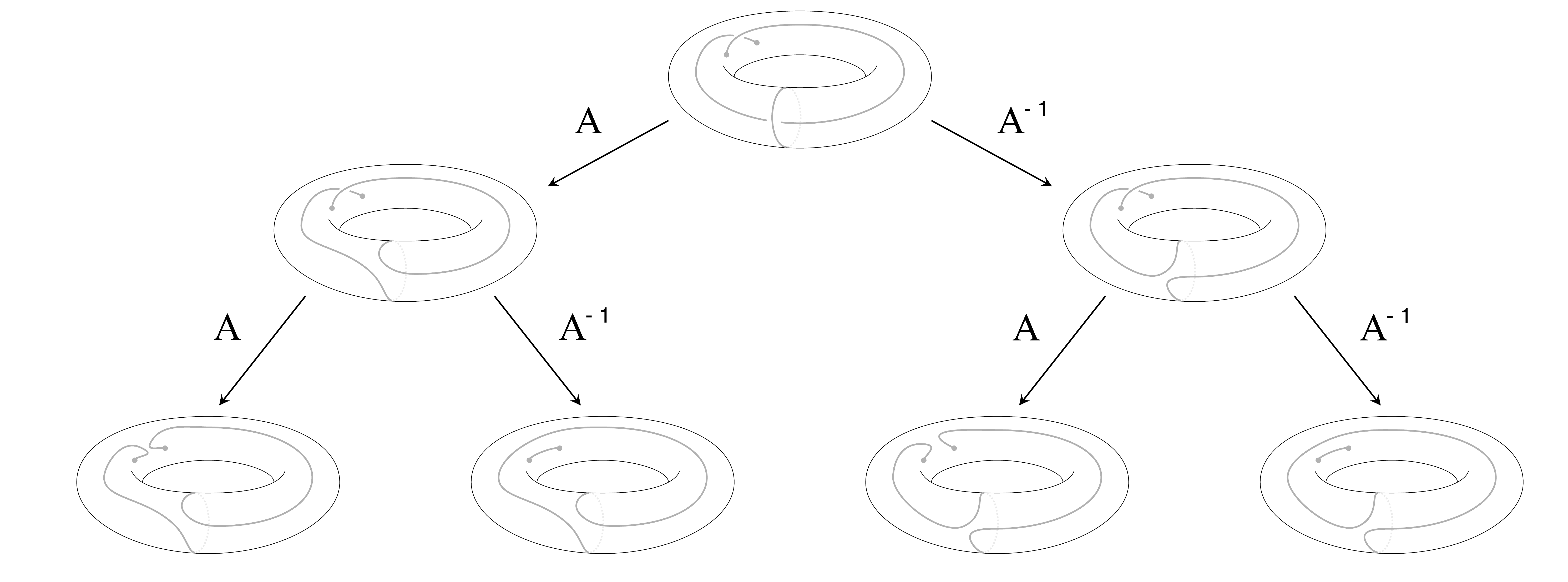}
\end{center}
\caption{The tree diagram for computing the toroidal bracket polynomial of $L_t$.}
\label{ex-bracket-torus}
\end{figure}

We compute the annular and toroidal bracket polynomial for $L_a$ and $L_t$, respectively, and we obtain the following equations:

\[
\begin{array}{lcl}
\langle L_a \rangle & = &  \left(-A^{-3}\right) (1 + A^8) t + \left(- A^{-5}\right) (1 + A^8) \\
&&\\
\langle L_t \rangle & = & ( 1 + A^2) \, +\, s_{-1,1} \, +\, A^{-2} s_{1,1}
\end{array}
\]
The presence of additional variables and terms in $L_a$ and $L_t$ demonstrates how the topology of
the underlying 3-manifold affects the computation of the bracket polynomial.

%%%%%%%%%%
%%%%%%%%%%%
%%%%%%%%%%%

\end{document}